\documentclass[11pt,leqno]{amsart}

\usepackage{amsthm,amsfonts,amssymb,amsmath,oldgerm,color,bm,multicol}

\usepackage{amsthm,amsfonts,amssymb,amsmath,color}

\numberwithin{equation}{section}

\renewcommand\d{\partial}
\renewcommand\a{\alpha}
\renewcommand\b{\beta}
\renewcommand\o{\omega}

\newcommand\R{\mathbb R}\newcommand\Z{\mathbb Z}

\def\g{\gamma}
\def\de{\delta}

\def\OO{\Omega}
\def\th{\theta}

\def\l{\lambda}

\def\epsilon{\varepsilon}
\def\e{\varepsilon}

\def\bbeta{\bar{\eta}}
\def\fF{\frak{F}}

\newcommand\br{\begin{rem}}
\newcommand\er{\end{rem}}
\newcommand\bp{\begin{pmatrix}}
\newcommand\ep{\end{pmatrix}}
\newcommand\be{\begin{equation}}
\newcommand\ee{\end{equation}}
\newcommand\ba{\begin{equation}\begin{aligned}}
\newcommand\ea{\end{aligned}\end{equation}}

\newcommand\nn{\nonumber}

\usepackage{layout}

\newcommand{\calE}{\mathcal{E}}
\newcommand{\calR}{\mathcal{R}}

\newcommand{\TT}{{\mathbb T}}

\newcommand{\vr}{\varrho}

\newcommand{\dx}{{\rm d} {x}}
\newcommand{\dy}{{\rm d} {y}}
\newcommand{\dt}{{\rm d} t }

\def\h{{\rm h}}
\def\rmv{{\rm v}}
\def\fD{\frak{D}}
\def\eqdefa{\buildrel\hbox{\footnotesize def}\over =}
\newcommand{\andf}{\quad\hbox{and}\quad}
\newcommand{\with}{\quad\hbox{with}\quad}

\newcommand{\ti}{{\tilde I}}

\newcommand{\cF}{\mathcal{F}}

\newcommand{\dive}{{\rm div\,}}
\newcommand{\divee}{{\rm div}_\e}
\newcommand{\curl}{{\rm curl\,}}
\newcommand{\beq}{\begin{equation}}
\newcommand{\eeq}{\end{equation}}
\newcommand{\ben}{\begin{eqnarray}}
\newcommand{\een}{\end{eqnarray}}

\newtheorem{defi}{Definition}[section]
\newtheorem{theorem}[defi]{Theorem}

\newtheorem{corollary}[defi]{Corollary}
\newtheorem{remark}[defi]{Remark}

\newtheorem{lem}{Lemma}[section]

\newtheorem{prop}{Proposition}[section]

\numberwithin{equation}{section}

\begin{document}

\title[Global solutions of 2D compressible Navier-Stokes system]
{Global solutions of 2D isentropic compressible Navier-Stokes equations with one slow variable}

\author[Y. Lu]{Yong Lu}
\address[Y. Lu]{Department of Mathematics, Nanjing University, Nanjing 210093, China.}
\email{luyong@nju.edu.cn}

\author[P. Zhang]{Ping Zhang}
\address[P. Zhang]{Academy of Mathematics $\&$ Systems Science
and  Hua Loo-Keng Key Laboratory of Mathematics, Chinese Academy of
Sciences, Beijing 100190, China, and School of Mathematical Sciences,
University of Chinese Academy of Sciences, Beijing 100049, China.}
\email{zp@amss.ac.cn}

\date{\today}

\begin{abstract} Motivated by \cite{CG10,CZ6}, we prove the global existence of solutions to the two-dimensional isentropic compressible Navier-Stokes equations with smooth initial data which are slowly varying in one direction and with initial density being away from vacuum. In particular, we present examples of initial data which generate unique global smooth solutions to  2D compressible Navier-Stokes equations with constant viscosity and with initial data which are neither small perturbation of some constant equilibrium state nor of small energy.

\end{abstract}

\maketitle

\noindent {\sl Keywords:}
compressible Navier-Stokes equations, strong solution, large data, slow variable.

 \vskip 0.2cm
\noindent {\sl AMS Subject Classification (2000):} 35Q35, 35B65, 76N10.

 \tableofcontents

%\renewcommand{\refname}{References}

%%%%%%%%%%%%%%%%%%%%%%%%%%%%%%%%%%%%%%%%%%%%%%%%%%%%%%%%%%%%%%%%%%%%%%%%%%%%%%%%%%%%%%%%%%

\section{Introduction}

We investigate the global existence of smooth solutions to the following
 two-dimensional isentropic compressible Navier-Stokes equations describing the motion of viscous barotropic compressible flows:
\be\label{CNS}
\left\{\begin{aligned}
&\d_t \rho  + \dive (\rho u)= 0,\qquad (t,x,y)\in\R^+\times\TT\times\R,\\
&\d_t(\rho u) + \dive(\rho u \otimes u) - \mu \Delta u - \nabla\left(\mu'  \dive u\right) + \nabla p(\rho) = 0,
\end{aligned}\right.
\ee
where $\rho(t,x,y) \in \R_+$ and  $u(t,x,y)\in \R^2$ designate the  the density and the velocity field of the fluid respectively.
  Here the shear viscosity $\mu$ and bulk viscosity $\mu'$  are assumed to be constant with $\mu>0$ and $ \mu+\mu' >0$. The pressure function $p$
 is a smooth non-decreasing function of  the density $\rho.$  For simplicity, we just consider  power law pressure with
\be\label{pressure}
p(\rho) = a \rho^\g, \quad a>0, \quad \g \geq 1.
\ee
The typical range of the adiabatic constant is $1<\g\leq \frac{5}{3}$ where the maximum value $\g = \frac{5}{3}$ is related to monatomic gases, the intermedius value $\g = \frac{7}{5}$ is related to diatomic gases including air, and lower values close to $1$ is related to polyatomic gases at high temperature. In this paper, we allow any $\g \in [1, 2]$.  One may check \cite{Lions-C} for more background of the system \eqref{CNS}.

\smallskip

The global existence of smooth solution to the multi-dimensional compressible Navier-Stokes equations is one of the most important
problems in the theory which describes the motion of the viscous compressible fluids. In \cite{MN83},  Matsumura and Nishida proved the global existence of solutions to the full compressible Navier-Stokes equations when the initial data is close to a constant equilibrium state in $H^3(\R^n).$ Danchin \cite{Dan00} worked out similar result yet with initial data in certain optimal function spaces. Huang, Li and Xin \cite{HLX12} established the global existence and uniqueness of classical solutions to the Cauchy problem of the barotropic compressible Navier-Stokes equations in three space dimensions with smooth initial data which are of small energy but possibly large oscillations. Concerning general smooth data without size restriction and with initial density being away from vacuum, the first well-posedness result is due to Va\u{\i}gant and  Kazhikhov \cite{VK95} where the authors proved the global unique solution to the two-dimensional compressible Navier-Stokes system \eqref{CNS} in $\R^+\times\TT^2$ with the bulk viscosity $\mu'=b\rho^\b$ and $b>0, \b>3.$ Lately this result was improved by Huang and Li \cite{HL16} for $\b>\frac43.$ For further results in higher dimensional setting, we refer to \cite{Valli83, CK03, FNP} and the references therein.

\smallskip

On the other hand, Chemin and Gallagher \cite{CG10} proved the global existence of smooth solution to three-dimensional incompressible Navier-Stokes equations $(NS)$ with initial data which is slowly varying in one direction. This type of result was extended by Chemin and the second author \cite{CZ6}
to the three-dimensional incompressible  inhomogeneous Navier-Stokes equations and by Liu and the second author \cite{LZ4} for $(NS)$ with unidirectional derivative of the initial velocity being sufficiently small in some critical functional space.

Motivated by \cite{CG10,CZ6,LZ4},  we are looking for global solutions of \eqref{CNS} that are periodic in $x$ variable, and have slow variation in $y$ variable.
  This means that we implement the system \eqref{CNS} with  initial data of the form
\be\label{ini-data}
\rho(0,x,y) =  [\varsigma_0]_\e(x, y), \quad u(0,x,y)   =  \big([w_0]_\e(x, y), [\frak{w}_0]_\e(x, y) \big),
\ee
where $(x,y)\in \TT \times \R$ with  $\TT $ being the torus $ \R/\Z$. Here and in all that follows, we shall always denote
$$
[f]_\e(x,y) \eqdefa f (x, \e y).
$$

We aim at showing the global well-posedness of \eqref{CNS}--\eqref{ini-data} without assuming any size restriction on $(\varsigma_0,w_0,\frak{w}_0)$
provided that $\e$ is sufficiently small.
We remark that data of the form \eqref{ini-data} are neither of  small perturbation of some constant equilibrium state, say $(1,0),$ nor of small energy. Hence in particular, our result gives examples indicating that 2D compressible Navier-Stokes equations with constant viscosity might
 be globally well-posed even for general initial data with initial density being away from vacuum.

 \smallskip

Formally with initial data given by \eqref{ini-data}, we may seek solution of \eqref{CNS} as follows
\ba\label{vr-v}
\rho(t,x,y) = [{\xi}]_\e(t,x,y), \quad u(t,x,y) =[v]_\e(t,x,y),
\nn \ea
then it follows from \eqref{CNS} that
\be\label{CNS-e}
\left\{\begin{aligned}
&\d_t \xi  + \divee (\xi v)= 0,\\
&\d_t(\xi v) + \divee(\xi v \otimes v) - \mu \Delta_\e v - \mu' \nabla_\e \divee v + \nabla_\e p(\xi) = 0,
\end{aligned}\right.
\ee
with initial data
\ba\label{ini-data-v}
\xi(0,x,y) = \varsigma_0(x,y), \quad v(0,x,y) =   \big(w_0(x, y),\frak{ w}_{0}(x,y) \big).
\nn \ea
Here
\be\label{nabla-e}
\nabla_\e \eqdefa \bp\d_x\\ \e \d_y \ep, \quad \divee \eqdefa \nabla_\e \cdot \andf \Delta_\e \eqdefa  \d_x^2 + \e^2 \d_y^2.
\ee

Formally passing $\e \to 0$ in \eqref{CNS-e} leads to  a limit system of the form with a parameter $y:$
\be\label{CNS-limit}
\left\{\begin{aligned}
&\d_t \eta + \d_x ( \eta w)= 0,\qquad (t,x,y)\in\R^+\times\TT\times\R,\\
& \eta (\d_t w + w \d_x w) - \nu \d_x^2 w - \d_x p( \eta ) = 0 \with \nu\eqdefa \mu + \mu',
\end{aligned}\right.
\ee
and
\be\label{tw}
\eta ( \d_t\frak{w}  + w \d_x \frak{w}) - \mu \d_x^2 \frak{w}  = 0, \qquad (t,x,y)\in\R^+\times\TT\times\R.
\ee
Accordingly, the initial data are
\be\label{CNS-limit-initial}
 \eta (0,x,y) = \varsigma_0(x,y), \quad w(0,x,y) = w_{0}(x,y), \quad \frak{w}(0,x,y) = \frak{w}_{0}(x,y).
\ee

We observe that the equations of $(\eta ,w)$ and the equation of $\frak{w}$ are decoupled. %A consequence is that there admits a unique trivial
Furthermore,
  $( \eta ,w)$ satisfies one dimensional compressible Navier-Stokes equations with a parameter $y.$ We shall study in detail about this system in Section \ref{sec:1dNS}.

Before proceeding,  we  assume that
\ba\label{ass-ini-1}
&\underline\varsigma_0 \leq \varsigma_0 \leq \bar \varsigma_0, \quad (\d_{y}^{j}(\varsigma_0-1), \d_{y}^{j}w_{0}, \d_{y}^{j}\frak{w}_{0})\in (L_{\rmv}^{2} \cap L_{\rmv}^{\infty})(H_{\h}^{5})\with\\
&A_{k}(y) \eqdefa \sum_{j=0}^{2} \|\d_{y}^{j}(\varsigma_0-1, w_{0}, \frak{w}_{0})(\cdot,y)\|_{H^k_\h}^2 \andf \bar  A_{k} \eqdefa \sup_{y\in \R} A_{k} (y), \quad 0\leq k\leq 5,
\ea
for some positive constants $\underline\varsigma_0 $ and $\bar \varsigma_0$, and for  $j = 0,1,2$. Here the subscript $\h$ (resp. ${\rmv}$) denotes the norm on $\TT_{x}$ (resp. $\R_{y}$). We assume moreover  that
\ba\label{ini-1}
\int_{\TT} \varsigma_0(x,y)\,\dx = 1,\quad \int_{\TT} (\varsigma_0 w_0)(x,y)\,\dx = \int_{\TT} (\varsigma_0 \frak{w}_0)(x,y)\,\dx  = 0, \quad \forall\, y\in \R.
\ea

Our first result is concerned with  the large time exponential decay  for the solutions to the equations \eqref{CNS-limit} and \eqref{tw}
with initial data \eqref{CNS-limit-initial}.

\begin{theorem}\label{thm1}
{\sl Let $(\varsigma_{0}, w_{0}, \frak{w}_0)$ satisfy \eqref{ass-ini-1} and \eqref{ini-1}. Then for each $y\in \R$,
the system \eqref{CNS-limit}--\eqref{CNS-limit-initial} has a unique global-in-time strong solution $(\eta, w, \frak{w}) \in C([0,\infty); H^2(\TT))$ so that
\beq\label{thm1-2}
\underline \eta \leq \eta(t,x,y) \leq \bar \eta, \quad \mbox{for some positive constants $\underline \eta$ and $\bar \eta$.}
\eeq
 Moreover,
there exist positive constants  $C$ and $\a$ solely depending on $(a, \g, \nu, \bar \varsigma_0, \underline{\varsigma}_0, \bar A_{4})$ such that
for all $ t\in \R_+$ and $ y\in \R$
\ba\label{thm1-0}
&\Bigl( \| (\eta -1)(t,\cdot,y) \|_{H^{4}_\h} + \| (w, \frak{w})(t,\cdot,y)\|_{H^{5}_\h}  + \| (\eta_{t}, w_{t}, \frak{w}_{t})(t,\cdot,y)\|_{H^3_\h}\\
&\ + \|  (\eta_{tt}, w_{tt}, \frak{w}_{tt})(t,\cdot,y)\|_{H^1_\h} \Bigr) + \sum_{j=1}^{2} \Bigl( \|\d_{y}^{j} \eta  \|_{H^{3-j}_\h} + \|\d_{y}^{j}  (w, \frak{w})(t,\cdot,y)\|_{H^{4-j}_\h}\\
&\qquad\qquad\qquad\qquad\qquad\qquad\qquad\qquad +\| \d_{y}^{j} (\eta_{t},w_{t}, \frak{w}_{t})(t,\cdot,y)\|_{H^{2-j}_\h}\Bigr)  \leq C A_{5}^{\frac{1}{2}}(y) e^{-\a t}.
\ea
}\end{theorem}

 We point out that to prove Theorem \ref{thm1},
  it is crucial to establish  the related estimates for  $(\eta, w)$ in \eqref{thm1-0}.
  Indeed with thus obtained estimates for  $(\eta, w)$,
 those estimates for $\frak{w}$ follow immediately.  The proof of Theorem \ref{thm1} will be
 presented from Section \ref{sec:1dNS} to Section \ref{Sect6}.

The main result of this paper states as follows:

\begin{theorem}\label{thm2}
{\sl Let $(\varsigma_{0}, w_{0}, \frak{w}_0)$ satisfy \eqref{ass-ini-1} and \eqref{ini-1}. Then
\eqref{CNS}--\eqref{ini-data} has a unique global-in-time strong solution $(\rho_\e, u_\e) \in C([0,\infty); H^2(\TT\times \R)).$
Furthermore, let $(\eta, w,\frak{w})$ be the global solution to  (\ref{CNS-limit}--\ref{CNS-limit-initial})
obtained  in Theorem \ref{thm1}, we denote
\be\label{thm2-1}
\vr_\e  \eqdefa \rho_\e - [\eta]_\e  \andf R_\e \eqdefa u_\e  - \left([w]_\e, [\frak{w}]_\e\right)^{\rm T},
\ee
and the energy functional
\ba\label{thm2-2}
E_\e(T) \eqdefa & \sup_{0 < t < T} \int_{\TT\times \R} \left( |R_\e|^{2} + |\vr_\e|^{2} +  |\nabla R_\e|^{2} +  |\fD_{t}R_\e|^{2} + |\nabla \o_\e|^{2}\right)\,\dx\,\dy \\
& + \int_{0}^{T}\int_{\TT\times \R}  \left( |\nabla R_\e|^{2}  +  |\fD_{t}R_\e|^{2} + |\nabla \o_\e|^{2} + \  |\nabla \fD_{t}R_\e |^{2} \right)\,\dx\,\dy\,\dt,\\
\th_\e(T)\eqdefa & \sup_{0 < t < T} \|\vr_\e(t)\|_{L^{\infty}(\TT\times \R)},
\ea
where $\o_\e \eqdefa  curl R_\e = \d_{y} R_\e^{1} - \d_{x} R_\e^{2}$, $\fD_{t}R_\e \eqdefa \d_{t}R_\e + u_\e \cdot \nabla R_\e$.  Then  there exists a constant  $C$  solely depending on $(a, \g, \mu, \mu', \bar \varsigma_0, \underline{\varsigma}_0, \|A_{5}\|_{L^{1}\cap L^{\infty}(\R_{y})})$ such that
\ba\label{thm2-3}
E_\e(\infty) + \th_\e^{2}(\infty) \leq C \e.
\ea
}
\end{theorem}

We remark that the main idea used to prove Theorem \ref{thm2} is to approximate 2D compressible Navier-Stokes equations with a slow variable via 1D compressible Navier-Stokes equations with a parameter. This is inspired by the study in \cite{CG10,CZ6} where 3D incompressible Navier-Stokes equations with a slow variable can be well approximated by 2D Navier-Stokes equations which is globally well-posed.

\smallskip

Let us end this section with some notations that we shall use throughout this paper.

\noindent{\bf Notations:} In the whole paper,
we designate  $\OO\eqdefa\TT\times \R,$ $ Q_{T} \eqdefa (0,T)\times \OO$ and
$Q_{\infty} \eqdefa (0,+\infty)\times \OO.
$
We shall always denote $C$ to be a uniform constant
which may vary from line to line, and $a\lesssim b$ means that $a\leq Cb$, and $a\thicksim b$
means that both $a\leq Cb$ and $b\leq Ca$ hold.
And we use subscript $\h$ (resp. ${\rmv}$) to denote  the norm taken
on $\TT_{x}$ (resp. $\R_{y}$).

\section{Ideas and structure of the proof}\label{sec2}

In this section, we shall sketch the main ideas of the proof to Theorems \ref{thm1} and \ref{thm2}.

For each fixed $y\in \R$, the global existence and uniqueness of strong solutions to 1D compressible Navier-Stokes equations \eqref{CNS-limit} is well known (see
for instance \cite{Hoff87,KS77, Solo76, Solo80, SZ02}).
Although the results in the above references focus on domains either being the whole line $\R$ or being a bounded interval,
the well-posedness results can be easily modified to the torus $\TT$. More precisely, under the assumptions that
\be\label{ini-data-1d-ass}
 0<\underline\varsigma_0 \leq \varsigma_0 \leq \bar \varsigma_0 <\infty, \quad \varsigma_0-1 \in H^1(\TT), \quad   w _0\in H^1(\TT),
\ee
\eqref{CNS-limit}--\eqref{CNS-limit-initial} has a unique global solution $( \eta , w )$  so that for each $T>0$
\ba\label{sl-1d-0}
&0<\underline\varsigma(T) \leq  \eta \leq \bar \varsigma(T)<\infty, \quad  \eta-1 \in C([0,T];H^1(\TT)),\quad  \eta _t\in C([0,T];L^2(\TT)),\\
&  w  \in C([0,T]; H^1(\TT)),\quad  w  \in L^2(0,T;H^2(\TT)), \andf  w _t \in L^2(0,T;L^2(\TT)).
\ea

In general, the related estimates in \eqref{sl-1d-0} depend on the time interval $[0,T]$. We shall show below that
such estimates hold uniformly in time. Actually, we shall show  exponential decay-in-time estimates of the solutions.   In  the case when spatial domain is a bounded interval and
 with homogeneous
 boundary condition for $w,$  Stra\v{s}kraba and Zlotnik \cite{SZ02} established exponential  decay-in-time estimate for the $H^1$ norms of $ \eta - 1$ and $w$. It is not obvious that the same exponential decay estimate holds for the case of torus uniformly in $y\in \R$. One difficulty lies in that, unlike the case in  bounded domain with homogeneous boundary condition for $w,$ we can not have a similar version of  Poincar\'e inequality on torus $\TT$.
 Another difficulty  is due to the presence of the  parameter $y\in \R.$ We will have  to show that the corresponding decay estimates are uniform with respect to $y\in \R$ and $(\d_y^j \eta ,\d_y^j w ),$ $j=1,2,$ share the same decay estimate.   To achieve this, we need to show the density function $\eta$ admits a uniform finite upper bound and a uniform positive lower bound with respect to $y\in \R$. In particular, in order to derive the uniform strictly positive lower bound,  one needs to show that as time goes to infinity, the kinetic energy goes to zero and the integral of the pressure goes to a strictly positive number  with a speed independent of $y\in \R$.  The result in \cite{SZ02}
  can not be directly applied here.
   Instead, we shall  first establish the uniform upper bound of the density by using the idea in \cite{SZ02}. Then we find that this is enough to derive the exponential decay for the basic energy, see Proposition \ref{prop-decay-exp-L2}.  Since $(\d_y^j \eta ,\d_y^j w ),$ $j=1,2,$ does not fulfill a typical 1D compressible Navier-Stokes system, there arise new difficulties in this procedure of deriving the exponential decay estimates for  $(\d_y^j \eta,\d_y^j w),$ $j=1,2$. A further difficulty comes from the Gagliardo-Nirenberg interpolation inequality in $\OO$. Since $\OO = \TT \times \R$ which is essentially bounded in $x$ variable, thus the classical Gagliardo-Nirenberg interpolation inequality reads ($p>2$):
$$
\|f\|_{L^{p}(\OO)} \leq C  \| f \|_{L^{2}(\OO)}^{\frac 2p} \| f \|_{H^{1}(\OO)}^{1-\frac{2}{p}} \leq C \| f \|_{L^{2}(\OO)} + C \| f \|_{L^{2}(\OO)}^{\frac 2p} \| \nabla f \|_{L^{2}(\OO)}^{1-\frac{2}{p}} , \quad \forall f \in H^{1}(\OO).
$$
However, it happens to us that we can control  the integral of $\| \nabla f \|_{L^{2}(\OO)}^{2}$ in time variable over $\R^+$ for some $f$, but not for the quantity $\| f \|_{L^{2}(\OO)}^{2},$ see for instance \eqref{vflux-2} and \eqref{vflux-4} on the effective viscous flux in the proof of Lemma \ref{prop-nablaR-6}. So we need a modified version of Gagliardo-Nirenberg interpolation inequality in $\OO$ which always involves the term $\| \nabla f \|_{L^{2}(\OO)},$ which  is  Lemma \ref{lem-GN} below.

\medskip

In Section \ref{sec:1dNS}, we shall present the detailed decay-in-time estimates for solutions of \eqref{CNS-limit}--\eqref{CNS-limit-initial}.

\begin{prop}\label{S2prop1}
{\sl Under the assumptions of Theorem \ref{thm1}, \eqref{CNS-limit}--\eqref{CNS-limit-initial} has a unique global solution $(\eta,w)$ and there exist positive constants $\underline{\eta}, \bar{\eta},$ $C$ and $\a$ solely depending on $(a, \g, \nu, \bar \varsigma_0, \underline{\varsigma}_0, \bar A_{4})$ such that \eqref{thm1-2} and
\ba\label{S2eq1}
\| (\eta -1)(t,\cdot,y) \|_{H^{4}_\h} + \|w(t,\cdot,y)\|_{H^{5}_\h}  &+ \sum_{i=1}^2\bigl\| (\d_{t}^i\eta, \d_{t}^iw)(t,\cdot,y)\bigr\|_{H^{5-2i}_\h}\leq C A_{5}^{\frac{1}{2}}(y) e^{-\a t}.
\ea
hold for any $ t\in \R_+$ and $ y\in \R.$
}
\end{prop}

To avoid notational complexity, we shall denote  $E(y)\in L^1(\R)\cap L^{\infty}(\R)$ to  a universal function, which is determined by the initial conditions. The positive constants $C$ and $\a$ are also determined by the initial conditions. From now on, we shall not point out the precise dependence of $E(y)$ and the constants $C, \alpha,$ and we neglect the dependence on the $y$ variable.

\begin{prop}\label{S2prop2}
{\sl Under the assumptions of Theorem \ref{thm1},  there exist positive constants $C$ and $\a$ so that
the  global solution $(\eta,w)$ of  \eqref{CNS-limit}--\eqref{CNS-limit-initial} verifies
\ba\label{decay-exp-all-y-2}
\| \eta_y(t) \|_{H^2_\h} + \| w_{y}(t) \|_{H^3_\h} + \|(\eta _{yt},w_{yt})(t)\|_{H^1_\h} \leq C E^{\frac 12}(y) e^{-\a t}, \quad \forall \, t\in \R^{+}, \ y\in \R.
\ea
}
\end{prop}

\begin{prop}\label{S2prop3}
{\sl Under the assumptions of Theorem \ref{thm1},  there exist positive constants $C$ and $\a$ so that the  global solution $(\eta,w)$ of  \eqref{CNS-limit}--\eqref{CNS-limit-initial} fulfills
\ba\label{decay-exp-all-yy-2}
\| \eta_{yy}(t) \|_{H^1_\h} + \| w_{yy}(t) \|_{H^2_\h} + \|(\eta_{yyt},w_{yyt})(t)\|_{L^2_\h} \leq
E^{\frac 12}(y)  e^{-\a t}, \quad \forall \, t\in \R^{+}, \ y\in \R.
\ea
}
\end{prop}

\begin{prop}\label{S2prop4}
{\sl Under the assumptions of Theorem \ref{thm1}, \eqref{tw}--\eqref{CNS-limit-initial} has a unique global solution
 $\frak{w}$ and there exists positive constants $C$ and $\a$ so that for all $t\in \R^{+}$ and
$y\in \R$
\ba\label{S2eq4}
 \sum_{j=1}^{2} \bigl(\|\d_t^j\frak{w}(t)\|_{H^{5-2j}_\h} +\|\d_{y}^{j} \frak{w}(t)\|_{H^{4-j}_\h} +\| \d_{y}^{j} \frak{w}_{t}(t)\|_{H^{2-j}_\h}\bigr)
 \leq C E^{\frac 12}(y) e^{-\a t} \quad \forall \, t\in \R^{+}, \ y\in \R.
\ea
}
\end{prop}

The proof of Proposition \ref{S2prop2} will be presented in Section  \ref{sec:1dNS-y}. While we shall outline the proof of Propositions \ref{S2prop3} and \ref{S2prop4} in Sections \ref{sec:1dNS-yy}
and \ref{Sect6} respectively. By combing Propositions  \ref{S2prop1}-\ref{S2prop4}, we conclude the proof of Theorem \ref{thm1}.

 Let $(\eta,w,\frak{w})$ be the unique global solution of \eqref{CNS-limit}-\eqref{CNS-limit-initial}, which has been
 constructed in Theorem \ref{thm1}. We define
\be\label{zeta-W-def}
\rho^{\rm a}_\e \eqdefa [\eta]_\e, \quad u^{\rm a}_\e \eqdefa ([w]_\e, [\frak{w}]_\e)^{\rm T}.
\ee
Then in view of  \eqref{CNS-limit} and \eqref{tw}, we have
\be\label{CNS-1d-W}
\left\{\begin{aligned}
&\d_t \rho^{\rm a}_\e + \dive (\rho^{\rm a}_\e u^{\rm a}_\e ) = \e [(\eta \frak{w})_{y}]_\e,\\
&\rho^{\rm a}_\e(\d_t  u^{\rm a}_\e +  u^{\rm a}_\e \cdot \nabla u^{\rm a}_\e ) - \mu \Delta u^{\rm a}_\e- \mu' \nabla \dive u^{\rm a}_\e  + \nabla p(\rho^{\rm a}_\e) = - G_\e,
\end{aligned}\right.
\ee
where $G_\e = (G_{1,\e}, G_{2,\e})^{\rm T}$ is given by
\ba\label{def-G12}
& G_{1,\e}\eqdefa  - \e [\eta\frak{w}w_{y} ]_\e  + \mu\e^{2 }[w_{yy}]_\e  + \mu' \e [\frak{w}_{xy}]_\e, \\
&  G_{2,\e}\eqdefa   - \e [\eta\frak{w}\frak{w}_{y} ]_\e  + \nu \e^{2} [\frak{w}_{yy} ]_\e +  \mu' \e [w_{xy}]_\e  - \e [p(\eta)_{y}]_\e.
\ea
Then by virtue of \eqref{thm1-0}, we deduce that
\ba\label{def-G12-est}
\|\d_{t} G_{\e}\|_{L^{2}(\Omega)} \leq C \e^{\frac{1}{2}}e^{-\a t}, \quad \| G_{\e}\|_{L^{2}(\Omega)} \leq C \e^{\frac{1}{2}}e^{-\a t}, \quad \| G_{\e}\|_{L^{\infty}(\Omega)} \leq C \e e^{-\a t}.
\ea

On the other hand, with initial data given by \eqref{ini-data} and $(\varsigma_0, w_0,\frak{w}_0)$ satisfying \eqref{ass-ini-1} and \eqref{ini-1},
there exists a positive time $T^\ast_\e>0$ so that   \eqref{CNS}--\eqref{ini-data} has a unique solution $( \rho_\e , u_\e )$ on
 $[0,T^\ast_\e)$  and for any $T<T^\ast_\e,$
\ba\label{sl-1d-00}
&0<\underline\rho(T) \leq \rho_\e  \leq \bar \rho(T)<\infty, \quad  \rho_\e \in C([0,T];H^2(\OO)),\quad  \d_t\rho_\e \in C([0,T];H^1(\OO)),\\
&  u_\e  \in C([0,T]; H^2(\OO))\cap L^2(0,T;H^3(\OO)), \quad   \d_tu_\e \in C([0,T];L^2(\OO)) \cap L^2(0,T;H^1(\OO)).
\ea
One may check \cite{Nash62, Serrin59} for details.

Let $T^\ast_\e$ be the maximal existence time of $(\rho_\e, u_\e)$ so that \eqref{sl-1d-00} holds. We are going to prove that
$T^\ast_\e=\infty$ for $\e$ being sufficiently small. In order to do so, we define the remaining term
\be\label{error}
\vr_\e \eqdefa \rho_\e  - \rho^{\rm a}_\e \andf R_\e \eqdefa u_\e - u^{\rm a}_\e .
\ee
Recall that $\frak{D}_t\eqdefa \d_t  + u_\e \cdot \nabla.$ Then it follows from \eqref{CNS} and \eqref{CNS-1d-W} that
\be\label{CNS-error-new}
\left\{\begin{aligned}%&\d_t \rho  + \dive (\rho_\e R) + (\rho w)_x = 0,\\
&\d_t \vr_\e  + \dive (\rho_\e R_\e) + \dive (\vr_\e u^{\rm a}_\e) = - \e [(\eta \frak{w})_{y}]_\e,\\
&\rho_\e \frak{D}_t R_\e- \mu \Delta R_\e- \mu' \nabla\dive R_\e + \nabla \big(p(\rho_\e)  - p(\rho^{\rm a}_\e)  \big) \\
&\qquad+ \rho_\e R_\e\cdot \nabla u^{\rm a}_\e +  \vr_\e (\d_{t}u^{\rm a}_\e  + u^{\rm a}_\e \cdot \nabla u^{\rm a}_\e) =    G_\e,\\
&\vr_\e(0,x,y) = 0, \quad R_\e(0,x,y) = 0.
\end{aligned}\right.
\ee

We define
\ba\label{a-priori-r}
T^\star_\e\eqdefa \sup\big\{ T<T^\ast_\e; \ \theta_\e(T)  \eqdefa  \sup_{0<t<T}\| \vr_\e \|_{L^{\infty}(\OO)} \leq  \frac{1}{2} \min\left\{1,\underline\eta\right\} \big\}.
\ea
It is easy to observe that for all $t< T^\star_\e$:
\ba\label{upper-lower-rho-0}
 0<{\underline \eta}/{ 2 } \leq \rho_\e  \leq \bar \eta +  {\underline \eta}/{ 2 } < \infty.
\ea

We shall first derive the basic energy estimate for $(\vr_\e,R_\e)$ for $t < T^\star_\e.$
 A standard way to perform this estimate  is to test the momentum type equation in \eqref{CNS-error-new} by $R_\e$ and using integration by parts. A tricky term to handle is the one related to the pressure:
\be\label{tricky-p-diveR}
\int_{\Omega}\big( p(\rho_\e) - p(\rho_\e^{\rm a}) \big) \dive R_\e\,\dx\,\dy.
\ee
Similar term $p(\rho_{\e}) \dive u_{\e}$ appears in the  renormalized formulation to the continuity equation of the
compressible Navier-Stokes equations:
$$
\d_{t}P(\rho_{\e})  + \dive (P(\rho_{\e}) u_{\e}) + (\rho_{\e} P'(\rho_{\e}) - P(\rho_{\e}))\dive u_{\e} = 0,
$$
where $P(\rho)$ is called the pressure potential, which is determined by
$$
\rho P'(\rho) - P(\rho) = p(\rho).
$$
As a result, it comes out
$$
\int_{\Omega} p(\rho_{\e})  \dive u_{\e} \,\dx\,\dy= -\frac{\rm d}{\dt}  \int_{\Omega} P(\rho_{\e})\,\dx\,\dy.
$$
In particular, for our case $p(\rho) =  a \rho^{\g}$ with $\g\geq 1$,  the corresponding pressure potential is
\be\label{pre-pot}
P(\rho)\eqdefa \frac{1}{\g-1} \rho^{\g} \ \mbox{if $\g>1$}, \quad  P(\rho) \eqdefa a (\rho \log \rho + 1) \ \mbox{if $\g = 1$}.
\ee

However, it is not clear at a first glance how to use such an argument to deal with the term \eqref{tricky-p-diveR}. Here we shall employ the well-known  {\rm relative entropy inequality} to derive the basic energy estimate for $\vr_\e$ and $R_\e$. 
 Relative entropy inequality  is widely used in the study of uniqueness and stability for Navier--Stokes equations and some related models, see \cite{FJN12,Germain10} for compressible Navier--Stokes equations, \cite{FNS14} for compressible Navier--Stokes--Fourier equations,  and \cite{Lu-Zhang18} for a compressible Oldroyd model. The result states as follows:

\begin{prop}\label{prop-rela}
{\sl Let $(\rho_\e, u_\e)$ be the local-in-time strong solution of \eqref{CNS}-\eqref{ini-data} satisfying \eqref{sl-1d-00}.  For each  pair $(\tilde \rho, \tilde u)$ satisfying the same regularity assumption as that of  $(\rho_\e, u_\e)$ listed in \eqref{sl-1d-00}, we define the {\em relative energy functional}
\ba\label{def-rel-fuc}
\calE_{1}\bigl((\rho_\e, u_\e) | (\tilde \rho, \tilde u)\bigr) (t)\eqdefa \int_{\Omega} \Bigl(\frac{1}{2} \rho_\e |u_\e-\tilde u|^{2} + P(\rho_\e) - P(\tilde \rho) - P'(\tilde \rho) (\rho_\e - \tilde \rho)\Bigr)\,\dx\,\dy,
\ea
where the pressure potential $P$ is defined by  \eqref{pre-pot}.  Then  for any $t\in (0,T^\ast_\e)$, the following \emph{relative entropy equality}   holds
\ba\label{ineq-entropy}
 \calE_1\bigl((\rho_\e,u_\e) | (\tilde\rho,\tilde u)\bigr)(t)
 + \int_0^t  \int_\Omega \bigl(\mu \left| \nabla (u_\e-\tilde u)  \right|^2 +& \mu' |\dive (u_\e-\tilde u) |^2\bigr)\,\dx\,\dy \,\dt' \\
 =& \calE_1\bigl((\rho_\e,u_\e)|_{t=0} | (\tilde\rho,\tilde u)|_{t=0}\bigr)  + \int_0^t \calR_\e (t') \,\dt',
\ea
where
\ba\label{R-def}
\calR_\e&(t)  \eqdefa \int_\Omega \rho_\e\frak{D}_t \tilde u\cdot (\tilde u - u_\e)\,\dx\,\dy
+ \int_\Omega \bigl(\mu \nabla \tilde u :\nabla (\tilde u - u_\e) + \mu' \dive \tilde u \,\dive (\tilde u-u_\e)\bigr)\,\dx\,\dy\\
& + \int_\Omega\bigl( (\tilde \rho - \rho_\e) \d_t P'(\tilde \rho_\e) + (\tilde \rho \tilde u -\rho_\e u_\e) \cdot \nabla P'(\tilde \rho)\bigr)\,\dx\,\dy  - \int_\Omega \dive \tilde u (p(\rho_\e ) - p(\tilde \rho))\,\dx\,\dy.
\ea
}
\end{prop}

There are also similar relative entropy inequalities related to finite energy weak solutions of different models: for instance, Theorem 2.4 of \cite{FJN12} for the compressible Navier--Stokes equations and Proposition 5.3 of \cite{Lu-Zhang18} for a compressible Oldroyd-B model.  The proof of Proposition \ref{prop-rela} follows the same line as the  argument in  \cite{FJN12} or \cite{Lu-Zhang18}, and we skip the details here.

Thanks to Proposition \ref{prop-rela}, we shall prove in Section \ref{Sect7.1} that

\begin{prop}\label{prop-energy-basic}
{\sl Let $T_\e^\star$ be given by  \eqref{a-priori-r}. Then for all $ t < T^\star_\e$, one has
\ba\label{energy-bas}
\int_{\Omega}  \bigl(|R_\e|^{2}  + \vr_\e^{2}\bigr)(t) \,\dx\,\dy + \int_{0}^{t} \int_{\Omega} |\nabla R_\e|^{2}\,\dx\,\dy\,\dt' \leq C \e,
\ea
where the positive constants $C$ is independent of $T^\star_{\e}$.
}
\end{prop}

In Section \ref{Sect7.4}, we shall prove the energy estimate for the derivatives of $R_\e.$

\begin{prop}\label{energy}
{\sl Let the energy functional  $E_\e(T)$ be given by \eqref{thm2-2}. Then for all $T \leq T^\star_\e,$ one has
 \be\label{E(T)-1}
 E_\e(T) \leq C \e + C\int_{0}^{T} \int_{\OO} \bigl(|\nabla R_\e|^{3} + |\nabla R_\e|^{4}\bigr)\,\dx\,\dy \,\dt.
 \ee}
 \end{prop}

  In order to close the energy estimate, \eqref{E(T)-1}, we  need to handle the estimates of the cubic and quadratic terms:
 $$
 \int_{0}^{t} \int_{\OO} |\nabla R_\e|^{3}\,\dx\,\dy \,\dt' \andf  \int_{0}^{t} \int_{\OO} |\nabla R_\e|^{4}\,\dx\,\dy \,\dt'.
 $$
 Here we introduce the following refined Gagliardo-Nirenberg interpolation inequality in $\OO$:
\begin{lem}\label{lem-GN}
{\sl  Let $2<p<\infty$, there exists a constant $C$ depending solely on $p$ such that for all $f\in H^{1}(\OO)$ there holds
\ba\label{GN-O}
\|f\|_{L^{p}(\OO)} \leq C \Bigl(  \| f \|_{L^{2}(\OO)}^{\frac 2p} \|\nabla f \|_{L^{2}(\OO)}^{1-\frac{2}{p}} +  \| f \|_{L^{2}(\OO)}^{\frac{1}{2} + \frac{1}{p}}  \|\nabla f \|_{L^{2}(\OO)}^{\frac{1}{2} - \frac{1}{p}} \Bigr).
\ea
%In particular,
%\ba\label{GN-OO}
%&\|f\|_{L^{4}(\OO)} \leq C \left(  \| f \|_{L^{2}(\OO)}^{\frac 12} \|\nabla f \|_{L^{2}(\OO)}^{\frac 12} +  \| f \|_{L^{2}(\OO)}^{\frac 34} \|\nabla f \|_{L^{2}(\OO)}^{\frac 14} \right), \\
%& \|f\|_{L^{6}(\OO)} \leq C \left(  \| f \|_{L^{2}(\OO)}^{\frac 13} \|\nabla f \|_{L^{2}(\OO)}^{\frac 23} +  \| f \|_{L^{2}(\OO)}^{\frac 23} \|\nabla f \|_{L^{2}(\OO)}^{\frac 13} \right).
%\ea
}
\end{lem}
The proof of Lemma \ref{lem-GN} will be presented in Appendix \ref{appb}.

\medskip

Motivated by  \cite{Hoff95, Lions-C}, we define the effective viscous flux $\frak{F}$ for the system \eqref{CNS-error-new} as follows
\be\label{visflux-def}
\fF_\e\eqdefa \nu \dive R_\e -\big(p(\rho_\e)  - p(\rho_\e^{\rm a})  \big).
\ee

\begin{lem}\label{lem-p-p}
{\sl For $T\leq T^\star_\e,$ one has \be\label{p-p-0}
\int_{0}^{T}\int_{\OO} | \rho_\e-\rho^{\rm a}_\e |^{6}\,\dx\,\dy \,\dt \leq C \e^{3} + C \e^{2} E_{\e}(T) + C \int_0^T\int_{\OO} |\fF_\e|^{6}\,\dx\,\dy \,\dt.
\ee
}
\end{lem}

\begin{proof}
In view  of the continuity equations of \eqref{CNS} and \eqref{CNS-error-new}, we write
 \ba
\fD_{t} (\log \rho_\e - \log \rho^{\rm a}_\e )   + R_\e \cdot \nabla \log \rho^{\rm a}_\e + \dive R_\e = - \e (\rho_\e^{\rm a})^{-1}  [(\eta \frak{w})_{y}]_\e,
\nn \ea
which together with \eqref{visflux-def} implies
 \ba
\fD_{t} (\log \rho_\e - \log \rho_\e^{\rm a} )
  = - R_\e \cdot \nabla \log \rho_\e^{\rm a}    - \nu^{-1} \bigl(\fF_\e+ \big(p(\rho_\e) - p(\rho_\e^{\rm a})\big)\bigr) - \e (\rho_\e^{\rm a})^{-1}  [(\eta \frak{w})_{y}]_\e.
\nn \ea
Multiplying the above equation by $6 (\log \rho_\e - \log \rho_\e^{\rm a} )^{5} $ and then applying Young's inequality gives
 \begin{align*}
& \fD_{t} (\log \rho_\e - \log \rho_\e^{\rm a} )^{6}  +  6 \nu^{-1} \big(p(\rho_\e) - p(\rho_\e^{\rm a})\big) (\log \rho_\e - \log \rho_\e^{\rm a} )^{5}  \\
&  =- 5 \left( R_\e\cdot \nabla \log \rho_\e^{\rm a}    +\nu^{-1} \fF_\e +  \e (\rho_\e^{\rm a})^{-1}  [(\eta \frak{w})_{y}]_\e\right) (\log \rho_\e - \log \rho_\e^{\rm a} )^{5}\\
& \leq C_{\de} \left(|R_\e \cdot \nabla \log \rho_\e^{\rm a}|^{6} +  |\fF_\e|^{6} +   | \e (\rho_\e^{\rm a})^{-1}  [(\eta \frak{w})_{y}]_\e |^{6}\right) + \de |\log \rho_\e - \log \rho_\e^{\rm a} |^{6}.
\end{align*}
Notice that for $t < T^\star_\e,$ there holds \eqref{upper-lower-rho-0} so that
 \ba
6 \nu^{-1}  \big(p(\rho_\e) - p(\rho_\e^{\rm a})\big) (\log \rho_\e - \log \rho_\e^{\rm a} )^{5}  \geq  2 C \de |\log \rho_\e - \log \rho_\e^{\rm a} |^{6} \geq \de |\log \rho_\e - \log \rho_\e^{\rm a} |^{6} + \de (\rho_\e - \rho_\e^{\rm a})^{6},
\nn \ea
for some small positive constant $\de$ and some $C\geq 2$. We thus obtain
 \ba\label{p-p-7}
\fD_{t} (\log \rho_\e - \log \rho_\e^{\rm a} )^{6}  + \de (\rho_\e - \rho_\e^{\rm a})^{6}  \leq C_{\de}\left( | R_\e \cdot \nabla \log \rho_\e^{\rm a}|^{6} +  |\fF_\e|^{6} +   | \e (\rho_\e^{\rm a})^{-1}  [(\eta \frak{w})_{y}]_\e |^{6}\right).
\ea
Let us introduce the particle trajectory $X_\e(t,x,y)$ of $u_{\e}$ via
\be\label{flow-def}
\frac{\rm d}{\dt} X_\e(t,x,y) = u_\e(t, X_\e(t,x,y)); \quad X_\e(0,x,y) = (x,y).
\ee
Then we deduce from \eqref{p-p-7} that
 \ba\label{p-p-8}
& \frac{\rm d}{\dt} \bigl(\log \rho_\e - \log \rho_\e^{\rm a} \bigr)^{6} (t, X_\e(t,x,y)) + \de \bigl(\rho_\e - \rho_\e^{\rm a}\bigr)^{6} (t, X_\e(t,x,y))  \\
& \leq C_{\de}\left( | R_\e \cdot \nabla \log \rho_\e^{\rm a}|^{6} +   |\fF_\e|^{6} + \e^6 |  (\rho_\e^{\rm a})^{-1}  [(\eta \frak{w})_{y}]_\e |^{6}\right) (t, X_\e(t,x,y)) .
\ea
Let $T\leq T_{\e}^{\star}$. Integrating the above inequality over $(0,T)\times\Omega$ yields
 \ba\label{p-p-11}
&\int_{\OO}\bigl(\log \rho_\e - \log \rho_\e^{\rm a} \bigr)^{6} (T, X_\e(T,x,y))\,\dx\,\dy + \delta\int_{0}^{T} \int_{\OO} (\rho_\e - \rho_\e^{\rm a})^{6} (t, X_\e(t,x,y))\dx\dy\dt \\
 &\leq C  \int_{0}^{T} \int_{\OO} \bigl(| R_\e \cdot \nabla \log \rho_\e^{\rm a}|^{6}(t, X_\e(t,x,y))+|\fF_\e|^{6}(t, X_\e(t,x,y))\bigr) \dx\dy\dt\\
&
+ C  \e^6  \int_{0}^{T} \int_{\OO}  |(\rho_\e^{\rm a})^{-1}  [(\eta \frak{w})_{y}]_\e |^{6} (t, X_\e(t,x,y))\,\dx\dy\dt.
\ea
Recall  \eqref{upper-lower-rho-0} and Lemma 3.2 of \cite{Hoff95} that there exists $C$ depending only on the positive lower bound and upper bound of $\rho_{\e}$ such that for any nonnegative smooth integrable function $g,$ there holds for all $t<T_{\e}^{\star}$,
 \ba
C^{-1} \int_{\OO} g(t,x,y) \dx\dy \leq \int_{\OO} g(t,X_\e(t,x,y)) \dx\dy \leq C\int_{\OO} g(t,x,y) \dx\dy.
\nn \ea
We thus deduce from \eqref{thm1-2} and \eqref{p-p-11} that
 \ba\label{p-p-12}
&\int_{0}^{T} \int_{\OO} (\rho_\e - \rho_\e^{\rm a})^{6} (t, x,y)\dx\dy\dt   \leq C  \int_{0}^{T} \int_{\OO} | R_\e \cdot \nabla \log \rho_\e^{\rm a}|^{6}(t, x,y) \dx\dy\dt\\
&\quad +  C \int_{0}^{T} \int_{\OO} |\fF_\e|^{6}(t, x,y)\dx\dy\dt+ C  \e^6 \int_{0}^{T} \int_{\OO}  | (\rho_\e^{\rm a})^{-1}  [(\eta \frak{w})_{y}]_\e |^{6} (t, x,y)\,\dx\dy\dt.
\ea
It follows from  Theorem \ref{thm1}, Proposition \ref{prop-energy-basic}  and  Lemma \ref{lem-GN} that
 \begin{align*}
& \int_{0}^{T} \int_{\OO} | R_\e \cdot \nabla \log \rho_\e^{\rm a}|^{6}(t, x,y) \,\dx\,\dy\,\dt\\
 & \leq C \int_{0}^{T}\int_{\OO} |R_\e|^{6} e^{-\a t}\,\dx\,\dy\,\dt\\
 & \leq C \int_{0}^{T}e^{-\a t} \big(\|R_\e\|_{L^{2}(\OO)}^{2} \|\nabla R_\e\|_{L^{2}(\OO)}^{4} + \|R_\e\|_{L^{2}(\OO)}^{4} \|\nabla R_\e\|_{L^{2}(\OO)}^{2} \big)\dt \\
 &  \leq C \sup_{0<t<T} \big(\|R_\e\|_{L^{2}(\OO)}^{2} \|\nabla R_\e\|_{L^{2}(\OO)}^{2} + \|R_\e\|_{L^{2}(\OO)}^{4}  \big) \int_{0}^{T} e^{-\a t} \|\nabla R_\e\|_{L^{2}(\OO)}^{2}\,\dt\\
 & \leq C \e^{2}(E_{\e}(T) + \e),
 \end{align*}
 and
  \begin{align*}
 \e^6\int_{0}^{T} \int_{\OO}  | (\rho_\e^{\rm a})^{-1}  [(\eta \frak{w})_{y}]_\e |^{6} (t, x,y)\,\dx\,\dy\,\dt
 & \leq \underline \eta^{-1}  \e^{6} \int_{0}^{T}\int_{\OO}\e^{-1} |(\eta \frak{w})_{y}|^{6}  \,\dx\,\dy\,\dt\\
 & \leq C \e^{5} \int_{0}^{T}e^{-\a t}  \dt  \leq C \e^{5}.
 \end{align*}
By inserting the above estimates into  \eqref{p-p-12}, we conclude the proof of \eqref{p-p-0}.
\end{proof}

\begin{lem}\label{prop-nablaR-6}
{\sl For $T\leq T^\star_\e,$ there holds
\ba\label{nablaR-6-0}
\int_{0}^{T} \int_{\OO} |\nabla R_\e|^{6} \,\dx\,\dy\,\dt  \leq C\e^{3} + C \e^{2} E_{\e}(T)  + C \e E^{2}_\e(T).
\ea
}
\end{lem}

\begin{proof}
Recall that $ \o_\e=\d_yR^1_\e-\d_xR^2_\e,$ we  observe that
$$
\Delta R_\e = \nabla \dive R_\e + \nabla^{\perp} \o_\e, \quad \nabla^{\perp} \eqdefa \bp \d_{y} \\ - \d_{x}\ep,
$$
which together with \eqref{visflux-def}
implies
\ba
\nu  \Delta R_\e = \nabla \fF_\e + \nabla \big(p(\rho_\e)  - p(\rho_\e^{\rm a})  \big) + \nu \nabla^{\perp} \o_\e.
\nn \ea
Then we get, by using  standard elliptic estimates that
\ba\label{cubic-0}
\|\nabla R_\e\|_{L^{6}(\OO)} \leq C \big(\|\o_\e\|_{L^{6}(\OO)}  + \|\fF_\e\|_{L^{6}(\OO)}  +  \|p(\rho_\e)  - p(\rho_\e^{\rm a}) \|_{L^{6}(\OO)} \big).
\ea
Next let us estimate term by term on the right-hand side of above equation.

By applying Lemma \ref{lem-GN}, we obtain
\ba
\|\o_\e\|_{L^{6}(\OO)} \leq C  \Bigl(  \| \o_{\e} \|_{L^{2}(\OO)}^{\frac 13} \|\nabla \o_{\e} \|_{L^{2}(\OO)}^{\frac 23} +  \| \o_{\e} \|_{L^{2}(\OO)}^{\frac 23} \|\nabla \o_{\e} \|_{L^{2}(\OO)}^{\frac 13} \Bigr).
\nn \ea
Together with \eqref{thm2-2} and \eqref{energy-bas}, we infer
\ba\label{vorticity-new-0}
\int_{0}^{T}\int_{\OO} |\o_\e|^{6}\,\dx\,\dy\,\dt & \leq C  \int_{0}^{T}  \left(  \| \o_{\e} \|_{L^{2}(\OO)}^{ 2} \|\nabla \o_{\e} \|_{L^{2}(\OO)}^{ 4} +  \| \o_{\e} \|_{L^{2}(\OO)}^{ 4} \|\nabla \o_{\e} \|_{L^{2}(\OO)}^{ 2} \right)\,\dt \\
& \leq C \sup_{0 < t < T} \big(\|\nabla \o_{\e}\|_{L^{2}}^{4} + \|\o_\e\|_{L^{2}}^{2} \|\nabla \o_\e\|_{L^{2}}^{2} \big)\int_{0}^{T}\|\nabla R_\e\|_{L^{2}(\OO)}^{2}\,\dt \\
&  \leq C \e E^{2}_\e(T).
\ea

While in view of \eqref{energy-bas}  and \eqref{visflux-def}, we have
\be\label{vflux-2}
\| \fF_\e \|_{L^{2}(\OO)} \leq  \nu \| \nabla R_\e \|_{L^{2}(\OO)} + C   \| \rho_\e - \rho_\e^{\rm a} \|_{L^{2}(\OO)} \leq C \big(\| \nabla R_\e \|_{L^{2}(\OO)} +  \e^{\frac 12} \big).
\ee
By \eqref{visflux-def}  and  the  $R_\e$  equation of \eqref{CNS-error-new}, we write
\ba%\label{vflux-3}
\Delta \fF_\e & =  \dive \big(  \mu  \Delta  R_\e  + \mu' \nabla \dive R_\e  -  \nabla \big(p(\rho_\e )  - p(\rho_\e ^{\rm a})  \big) \big) \\
& = \dive \left(\rho_\e  \fD_{t}R_\e + \rho_\e  R_\e\cdot \nabla u^{\rm a}_\e +   \vr_\e (\d_{t} u^{\rm a}_\e + u^{\rm a}_\e \cdot \nabla u_\e^{\rm a}) -  G_\e\right),\nonumber
\ea
from which, we infer
\ba\label{vflux-4}
 \|\nabla \fF_\e\|_{L^{2}(\OO)}  
 & \leq C\big( \|  \fD_{t}R_\e  \|_{L^{2}(\OO)}   +  \|R_\e\| _{L^{2}(\OO)}  \|\nabla u_\e^{\rm a}\|_{L^{\infty}(\OO)} \\
  &\qquad+   \|\vr_\e\|_{L^{2}(\OO)}   \bigl\|\d_{t} u_\e^{\rm a} + u_\e^{\rm a} \cdot \nabla u_\e^{\rm a}\bigr\|_{L^{\infty}(\OO)}   +  \|G_\e\|_{L^{2}(\OO)}\big)   \\
 & \leq C \|  \fD_{t}R_\e  \|_{L^{2}(\OO)}  + C e^{-\a t} \bigl( \|R_\e\| _{L^{2}(\OO)}   + \|\vr_\e\|_{L^{2}(\OO)}   + \e^{\frac 12} \bigr)\\
 & \leq C \bigl(\|  \fD_{t}R_\e  \|_{L^{2}(\OO)} +  \e^{\frac 12} e^{-\a t}\bigr).
\ea

Thanks to \eqref{vflux-2} and \eqref{vflux-4}, we get, by applying  Lemma \ref{lem-GN}, that
\ba\label{vflux-0}
\int_{0}^{T}\int_{\OO} |\fF_\e|^{6}\,\dx\,\dy\,\dt 
& \leq C  \int_{0}^{T}\big(\|\fF_\e\|_{L^{2}(\OO)}^{2} \|\nabla \fF_\e\|_{L^{2}(\OO)}^{4} + \|\fF_\e\|_{L^{2}(\OO)}^{4} \|\nabla \fF_\e\|_{L^{2}(\OO)}^{2} \big)\,\dt\\
& \leq C  \int_{0}^{T} \big(\| \nabla R_\e \|_{L^{2}(\OO)}^{2} +  \e \big) \big(\| \fD_{t} R_\e  \|_{L^{2}(\OO)}^{2} +  \e e^{-\a t}\big)\\
 &\qquad\qquad\times\big(\| \nabla R_\e \|_{L^{2}(\OO)}^{2} +  \e+ \| \fD_{t} R_\e  \|_{L^{2}(\OO)}^{2} +  \e e^{-\a t}\big) \,\dt\\
&\leq C  \e^{3} + C \e  E^{2}_\e(T).
\ea

By inserting the estimates \eqref{p-p-0}, \eqref{vorticity-new-0} and \eqref{vflux-0} into \eqref{cubic-0}, we achieve \eqref{nablaR-6-0}.
 This completes the proof of Lemma \ref{prop-nablaR-6}.
\end{proof}

Now we are ready to give the following lemma in order to close the energy estimates:
\begin{lem}\label{prop-nablaR-3-4}
{\sl For $T\leq T^\star_\e,$ there holds
\ba\label{nablaR-4-0}
&\int_{0}^{T} \int_{\OO} |\nabla R_\e|^{4} \,\dx\,\dy\,\dt  \leq C\e^{2} + C \e E_\e(T),\\
&\int_{0}^{T} \int_{\OO} |\nabla R_\e|^{3}\,\dx\,\dy\,\dt   \leq C\e^{\frac 32} + C \e E^{\frac 12}_{\e}(T).
\ea
}
\end{lem}

\begin{proof}
Indeed it follows from Proposition \ref{prop-energy-basic}, Lemma \ref{prop-nablaR-6} and H\"older's inequality that
\ba
&\|\nabla R_\e\|_{L^{4}((0,T)\times \OO)}^{4} \leq \|\nabla R_\e\|_{L^{2}((0,T)\times \OO)} \|\nabla R_\e\|_{L^{6}((0,T)\times \OO)}^{3} \leq C \e^{\frac 12} \big(\e^{\frac 32} + \e^{\frac 12} E_{\e}(T)\big),\\
&\|\nabla R_\e\|_{L^{3}((0,T)\times \OO)}^{3} \leq \|\nabla R_\e\|_{L^{2}((0,T)\times \OO)}^{\frac 32} \|\nabla R_\e\|_{L^{6}((0,T)\times \OO)}^{\frac 32} \leq C \e^{\frac 34} \big(\e^{\frac 34} + \e^{\frac 14} E^{\frac 12}_{\e}(T)\big).
\nonumber
\ea
Then \eqref{nablaR-4-0} follow immediately.
\end{proof}

Next let us turn to the {estimate of $\th_\e(T).$}

\begin{lem}\label{prop-thT}
{\sl For $T\leq T^\star_\e,$ there holds
\ba\label{thT-0}
\th_\e^{2}(T) \leq  C \e + C E_\e(T).
\ea
}
\end{lem}

\begin{proof}
 Notice that for all $T\leq T^\star_{\e},$ $0< \underline \eta/2 \leq \rho_\e , \rho_\e ^{\rm a} \leq 3\bar \eta/2,$ we deduce from \eqref{p-p-8} that
\begin{align*}
 \vr_\e^{6} (T, X_\e(T,x,y)) & \leq C\bigl(\log \rho_\e  - \log \rho_\e ^{\rm a} \bigr)^{6} (T, X_\e(T,x,y))\\
 & \leq  C\int_{0}^{T} \Big( | R_\e \cdot \nabla \log \rho_\e ^{\rm a} |^{6} +  |\fF_\e|^{6}  +  |\e [(\eta \frak{w})_{y}]_\e |^{6}\Big) (t, X_\e(t,x,y))\,\dt,
\end{align*}
from which, we infer
\ba\label{thT-7}
 \| \vr_\e (T)\|_{L^{\infty}(\OO)}^{6} & \leq  C\int_{0}^{T} \big( \| R_\e \|_{L^{\infty}(\OO)}^{6} \|(\rho_\e ^{\rm a})^{-1} \nabla \rho_\e ^{\rm a} \|_{L^{\infty}(\OO)}^{6} +  \|\fF_\e\|_{L^{\infty}(\OO)}^{6} +  \e^{6} \| (\eta \frak{w})_{y} \|_{L^{\infty}(\OO)}^{6} \big)\dt\\
 & \leq  C\int_{0}^{T } \big(e^{-\a t}\bigl( \| R_\e \|_{L^{\infty}(\OO)}^{6}   +  \e^{6}\bigr)
 +  \|\fF_\e\|_{L^{\infty}(\OO)}^{6}  \big)\,\dt.
\ea

It follows from Sobolev embedding theorem, Proposition \ref{prop-energy-basic}  and Lemma \ref{prop-nablaR-6} that
 \ba\label{EstR}
\int_{0}^{T} e^{-\a t} \|R_\e\|_{L^{\infty}(\OO)}^{6}\,\dt  & \leq C\int_{0}^{T} e^{-\a t} \bigl(\|R_\e\|_{L^{2}(\OO)}^{6}
 + \|\nabla R_\e\|_{L^{6}(\OO)}^{6}\bigr)\,\dt\\
& \leq C\e^{3} + C  \e E^{2}_\e(T).
\ea

Concerning the term related to the effective viscous flux, we get, by using Sobolev embedding inequality, that
\be\label{F-embed-Linfty}
\|\fF_\e\|_{L^{\infty}(\OO)} \leq C \big( \|\fF_\e\|_{L^{6}(\OO)} + \|\nabla \fF_\e\|_{L^{3}(\OO)}\big).
\ee
The estimate related to $\|\fF_\e\|_{L^{6}(\OO)}$ is given in \eqref{vflux-0}.  While for $\|\nabla \fF_\e\|_{L^{3}(\OO)}$, it follows from a similar derivation of \eqref{vflux-4} that
\begin{align*}
 \|\nabla \fF_\e\|_{L^{3}(\OO)}  & \leq C\Bigl( \|  \fD_{t}R_\e  \|_{L^{3}}   +  \|R_\e\| _{L^{3}}
   \|\nabla u^{\rm a}_\e\|_{L^{\infty}}  +   \|\vr_\e \|_{L^{3}}   \bigl\|\d_{t} u^{\rm a}_\e + u_\e^{\rm a}\cdot \nabla u_\e^{\rm a}\bigr\|_{L^{\infty}}  +  \|G_\e\|_{L^{3}} \Bigr)  \\
 & \leq C \|  \fD_{t}R_\e  \|_{L^{3}(\OO)}  + Ce^{-\a t} \bigl(\|R_\e\| _{L^{3}(\OO)}   +  \|\vr_\e \|_{L^{3}(\OO)} + \e^{\frac 23}\bigr),
\end{align*}
{where we used the estimate of $G_{\e}$ in \eqref{def-G12-est}.} By using the Gagliardo-Nirenberg interpolation inequality and \eqref{thm2-2}, we find
\ba
\int_{0}^{T} \|  \fD_{t}R_\e  \|_{L^{3}(\OO)}^{6}\,\dt &  \leq  C \int_{0}^{T}\|  \fD_{t}R_\e  \|_{L^{2}(\OO)}^{4}  \| \fD_{t}R_\e  \|_{H^{1}(\OO)}^{2}\,\dt \\
& \leq C \sup_{0 < t < T} \|  \fD_{t}R_\e(t)  \|_{L^{2}}^{4}   \int_{0}^{T} \| \fD_{t}R_\e  \|_{H^{1}(\OO)}^{2}\,\dt  \leq C E^{3}_{\e}(T).
\nonumber\ea
Similarly, we have
\begin{align*}
\int_{0}^{T}   e^{-\a t}  \|R_\e\| _{L^{3}(\OO)}^{6}\,\dt  & \leq C \int_{0}^{T}   e^{-\a t}  \|R_\e\| _{L^{2}(\OO)}^{4}  \| R_\e\|_{H^{1}(\OO)}^{2}\,\dt \\
 & \leq C \sup_{0 < t < T} \|  R_\e(t)  \|_{L^{2}(\OO)}^{4}   \int_{0}^{T} e^{-\a t}  \|  R_\e  \|_{H^{1}(\OO)}^{2}\,\dt \leq C\e^{3}.
\end{align*}
Finally, in view of \eqref{a-priori-r}, one has
\ba
\|\vr_\e (t)\|_{L^{3}(\OO)} \leq \|\vr_\e (t)\|_{L^{2}(\OO)}^{\frac 23}\|\vr_\e (t)\|_{L^{\infty}(\OO)}^{\frac 13} \leq C \e^{\frac 13} \th_\e^{\frac 13}(t),
\nn \ea
so that
\ba
\int_{0}^{T}   e^{-\a t}  \|\vr_\e (t)\| _{L^{3}(\OO)}^{6}\,\dt  \leq  C \e^{2}\int_{0}^{T}   e^{-\a t}   \th_{\e}^{2}(t)\,\dt    \leq  C \e^{2} \th_\e^{2}(T).
\nn \ea
As a result,
\ba
\int_{0}^{T} \|\nabla \fF_\e(t)\|_{L^{3}(\OO)}^6\,\dt   \leq  C \big( E_\e^{3}(T)+\e^{2} \th_\e^{2}(T)+\e^{3}\big),
\nn \ea
which together with \eqref{vflux-0} and \eqref{F-embed-Linfty} ensures that
 \ba\label{thT-9}
\int_{0}^{T}\|\fF_\e(t)\|_{L^{\infty}(\OO)}^6\,\dt  & \leq   C\int_{0}^{T}\bigl(\|\fF_\e\|_{L^{6}(\OO)}^6 + \|\nabla \fF_\e\|_{L^{3}(\OO)}^6\bigr)\,\dt
\\
&\leq  C \big( E_\e^{3}(T)+\e^{2} \th_\e^{2}(T)+\e^{3}\big).
\ea

By inserting the estimates \eqref{EstR} and \eqref{thT-9} into \eqref{thT-7}, we arrive at
 \ba
\th_\e^{6}(T)\leq  C \big( E_\e^{3}(T)+\e^{2} \th_\e^{2}(T)+\e^{3}\big) \leq C  \e^{3} + C E_\e^{3}(T) + \frac{\th_\e^{6}(T)}{2},
\nn \ea
which leads to \eqref{thT-0}.
\end{proof}

Now we are in a position to complete the proof of Theorem \ref{thm2}.

\begin{proof}[Proof of Theorem \ref{thm2}]
We first deduce from \eqref{E(T)-1} and \eqref{nablaR-4-0} that
\ba
 E_\e(T) \leq C \e +  C \e^{\frac 12} E_\e(T)\quad \mbox{for}\ T\leq T_\e^\star,
\nn \ea
which implies  that for $\e \leq \e_{1}$ small such that that $C\e_{1}^{\frac12} = \frac{1}{2},$ there holds
\be\label{E(T)-f}
 E_\e(T) \leq C \e \quad \mbox{for}\ T\leq T_\e^\star.
\ee
This together with Lemma \ref{prop-thT} ensures that
\be\label{E-th-T-1}
  \th_\e^{2}(T) \leq  C \e \quad \mbox{for}\ T\leq T_\e^\star.
\ee
In particular, if we take $\e \leq \e_{2}$ to be small that $C\e_{2} = \frac{1}{6} \min\left\{1,\underline\eta\right\}^2,$
\eqref{E-th-T-1} contradicts with the definition of $T_\e^\star$ given by \eqref{a-priori-r}. As a consequence,
we deduce that $T_\e^\star=T_\e^\ast$.

 It remains  to show that the life-span  $T_\e^\ast= \infty$. Indeed we have shown that the estimate \eqref{upper-lower-rho-0}  holds
 for all $t < T_\e^\ast$ and $\e\leq \e_0\eqdefa \min\{ \e_{1}, \e_{2}\}$, then regularity criteria for smooth solutions of compressible Navier-Stokes equations
 (see for instance \cite{HLX11, SWZ11, SZ11}) ensures that $T_\e^\ast= \infty.$
 This completes the proof of Theorem \ref{thm2}.

 \end{proof}

\section{1D compressible Navier-Stokes equations with a parameter}\label{sec:1dNS}

 In this section, we investigate the 1D compressible Navier-Stokes equations \eqref{CNS-limit} with a parameter $y.$
 We assume that $(\eta,w)$ is a global smooth solution of \eqref{CNS-limit} determined at the beginning of Section \ref{sec2}.
For simplicity,  we shall always neglect  $y$ variable and
 denote $D_t \eqdefa \d_t +  w  \d_x$ to be the material derivative in the rest of this section.

\subsection{Conservation of mass, momentum, and energy}

Integrating \eqref{CNS-limit} in $x$ over $\TT$  leads to the conservations of the mass and of the momentum:
\beq\label{csv-m-m}
\frac{\rm d}{\dt}\int_{\TT}  \eta (t,x,y)\,\dx = 0,\andf \frac{\rm d}{\dt}\int_{\TT} ( \eta  w )(t,x,y)\,\dx = 0,
\eeq
which together with \eqref{ini-1} ensures that
\ba\label{csv-m-m-1}
\int_{\TT}  \eta (t,x,y)\,\dx = 1,\quad \int_{\TT} ( \eta  w )(t,x,y)\,\dx = 0, \quad \forall\, t\in\R^{+}, \  y\in \R.
\ea

\medskip

Recall the conservation of energy:
\begin{lem}\label{prop-basic-energy}
{\sl  There holds
\ba\label{csv-eng}
\frac{\rm d}{\dt}&E_0(t) + \nu\int_{\TT}  w _x^2\,\dx = 0 \with E_0(t) \eqdefa  \int_{\TT} \bigl(\frac{1}{2} \eta  w ^2 + P( \eta )\bigr) \,\dx, \andf\\
& P(\eta) \eqdefa  \frac{a}{\g-1} \eta^\g \quad \mbox{if $\g >1$} \andf  P(\eta) \eqdefa a (\eta \log \eta + 1) \quad  \mbox{if $\g = 1$}.
\ea
}
\end{lem}

It follows from  \eqref{ass-ini-1} that
\be\label{ini-2}
\begin{split}
 \frak{E}_{00}(y) & \eqdefa  \int_{\TT} \bigl(\frac{1}{2} \varsigma_0 w _0^2  + (\varsigma_0-1)^2\bigr)(x,y) \,\dx  \in  (L^1 \cap  L^{\infty})(\R)\with \bar{\frak{E}}_{00} \eqdefa \sup_{y\in \R} \frak{E}_{00}(y) ,\\
 E_{00}(y) &\eqdefa  \int_{\TT} \bigl(\frac{1}{2} \varsigma_0 w _0^2 + P(\varsigma_0)\bigr)(x,y) \,\dx  \in   L^{\infty}(\R)
 \with \bar E_{00} \eqdefa \sup_{y\in \R} E_{00}(y).
 \end{split}
\ee
Then for all $t\in \R^{+},  \ y\in \R$, there holds
\ba\label{csv-eng-1}
E_0(t) + \nu \int_0^t \int_{\TT}  w _x^2\,\dx\,\dt' \leq  E_{00}(y).
\ea

\subsection{Upper bound for the density function}

We observe from \eqref{sl-1d-0} that the density function $\eta$ admits an upper bound depending on time.
 We shall derive here  a time uniform upper bound for $ \eta.$
 We  first recall from  \cite{SZ02} that for $ u \in L^1(\TT)$ and for all $x\in\TT$
\be\label{I-def}
I(u)(x) \eqdefa \int_0^x u(x')\,\dx', \quad \tilde I(u) \eqdefa I(u) - \langle I(u) \rangle \andf
  \langle u \rangle \eqdefa \int_{\TT} u (x)\,\dx.
\ee
It  is easy to observe that $I(u) \in C(\TT) \cap W^{1,1}(\TT)$ with
\be\label{I-pt1}
\|I(u)\|_{W^{1,1}_\h\cap L^\infty_\h} \leq \|u\|_{L^1_\h}.
\ee
Moreover, for each $u\in W^{1,1}(\TT) \cap C(\TT)$, standard density argument implies
\ba\label{I-pt2}
 I(u_x)(x) = \int_0^x u_x \,\dx' = u(x) - u(0) \andf \tilde I(u_x)(x)  = u(x) - \langle u \rangle.
\ea

\begin{prop}\label{prop-upperbd-density}
{\sl For all $(t,x,y) \in\R^{+} \times \TT \times \R$, there holds
\be\label{upperb-vtr}
 \eta (t,x,y) \leq \max\left\{ \varsigma_0, \bar \varsigma_1  \right\} \exp\bigl(4 \bar E_{00}^{\frac 12}\bigr)\eqdefa \bar{ \eta },
\ee
where $\bar\varsigma_0$ is given in \eqref{ini-data-1d-ass} and $\bar \varsigma_1$ will be determined by \eqref{def-bar-vr1} below.
}
\end{prop}

\begin{proof}

We rewrite the momentum equation in \eqref{CNS-limit} as
\ba
( \eta w)_t + ( \eta  w ^2)_x - \nu  w _{xx} + p( \eta )_x = 0.
\nn \ea
Applying the operator $\tilde I$ to the above equation  gives
\be\label{tI-mm-1}
\tilde I\big(( \eta w)_t\big) +  \eta  w ^2 - \nu  w _{x} + p( \eta ) - \langle  \eta  w ^2 - \nu  w _{x} + p( \eta )  \rangle = 0.
\ee
 We first compute
\begin{align*}
\tilde I\big(( \eta w)_t\big) = \d_t \tilde I( \eta w)  =  D_t \tilde I( \eta  w )  -   w  \d_x \tilde I( \eta  w ) =  D_t \tilde I( \eta  w )  -   \eta  w ^2.
\end{align*}
While it follows from the transport equation of \eqref{CNS-limit} that
\ba\label{tI-mm-3}
 -  w _x = \frac{- \eta  w _x}{ \eta } = \frac{ \eta _t +  w   \eta _x}{ \eta }  = D_t \log  \eta .
\nn \ea
By inserting the above equalities into \eqref{tI-mm-1} and
 the fact that $\langle w_x \rangle = 0$, we obtain
\be\label{tI-mm-4}
D_t \tilde I( \eta  w ) + \nu  D_t \log  \eta + p( \eta ) - \langle  \eta  w ^2 +  p( \eta )  \rangle = 0.
\ee

Let ${X}(t,x,y)$ be the trajectory of $w,$ which is determined by
\be\label{flow-defw}
\frac{\rm d}{\dt} X_w(t,x,y) = w(t, X(t,x,y)); \quad X_w(0,x,y) = x.
\ee
Let $\frak{y}(t,x,y) = \log  \eta(t,X_w(t,x,y),y) $. Then in view of \eqref{tI-mm-4}, we write
\be\label{tI-mm-6}
 \nu  \frac{\rm d}{\dt} \frak{y} = g(\frak{y}) + \frac{\rm d}{\dt}\frak{b}
\ee
with
$$
g(\frak{y}) = - p(\eta) + \langle  \eta  w ^2 +  p( \eta )  \rangle, \quad  \frak{b} = -  \tilde I( \eta  w ).
$$

Before proceeding, we recall the following lemma:

\begin{lem}[Lemma 1.3 of  \cite{SZ03}]\label{lem-uperb}
{\sl Suppose $g\in C(\R)$ and $\frak{y}, b\in W^{1,1}(0,T)$ for all $T>0$. Suppose that $\frak{y}$ verifies
$$
\frac{\rm d}{\dt} \frak{y}  = g(\frak{y}) + \frac{\rm d}{\dt} \frak{b} \quad \mbox{on}  \ \R^{+}, \quad \frak{y}(0) = \frak{y}_0.
$$
If $g(+\infty) = -\infty$ and there exist non-negative constants $N_0,N_1\geq 0$ so that
for any $0\leq t_1<t_2 <\infty,$
$$
\frak{b}(t_2) - \frak{b}(t_1) \leq N_0 + N_1(t_2 - t_1).
$$
Then one has
$$
\frak{y}(t) \leq \bar{\frak{y}} <\infty, \quad \forall \, t\in \R^{+} \with
\bar{\frak{y}} \eqdefa  \max\{ \frak{y}_0, \frak{y}_1 \} + N_0,
$$
where $\frak{y}_1$ is such that $g(\frak{y}) \leq -N_1$ for all $\frak{y}\geq \frak{y}_1$.
}
\end{lem}

Now we would like to apply Lemma \ref{lem-uperb} to derive an upper bound of $\eta$ via \eqref{tI-mm-6}. Firstly, it follows from
 \eqref{csv-eng-1} and \eqref{ini-2} that
$$
g(\frak{y}) = - p(e^{\frak{y}}) + \int_{\TT}  \left(\eta  w ^2 +  p( \eta )\right)  \,\dx
\leq -p(e^{\frak{y}}) + \g \bar E_{00} \to -\infty, \quad \mbox{as} \ \frak{y}\to +\infty.
$$
While in view of \eqref{I-def}, we have for all $t\in \R^{+}$,
\be\label{I(rv)}
|\ti ( \eta  w )|_{L^\infty_\h} \leq 2\| \eta  w \|_{L^1_\h} \leq 2 \| \eta \|_{L^1_\h}^{\frac 12}\| \eta  w ^2\|_{L^1_\h}^{\frac 12}\leq 2 \bar E_{00}^{\frac 12},
\ee
which implies
$$
{ |\tilde I( \eta  w )(t_{2}) -\tilde I( \eta  w )(t_{1}) |\leq 4 E_{00}^{\frac 12}, \ \forall \, t_1,  t_2 \in [0,\infty).}
$$
On the other hand, we observe that
\be\label{def-bar-vr1}
- p(\bar\varsigma_1) +  \g \bar E_{00} = 0 \Rightarrow
\bar \varsigma_1\eqdefa \Big(\frac{ \g \bar E_{00} }{a}\Big)^{\frac{1}{\g}}.
\ee
Hence we get, by applying Lemma \ref{lem-uperb}, that
\be\label{upperb-vtr0}
 \frak{y} \leq \max\{ \log  \bar \varsigma_0, \log \bar \varsigma_1  \} + 4 E_{00}^{\frac 12},
\ee
where $\bar \varsigma_0$ is given in \eqref{ini-data-1d-ass}. This leads to \eqref{upperb-vtr}.
\end{proof}

\subsection{Decay estimates of $L^2$ norms}

With the upper bound of $ \eta$ obtained  in \eqref{upperb-vtr}, similarly to Lemma 2.1 of \cite{ZZ20}, we have

\begin{lem}\label{prop-kinetic-tx}
{\sl Under the assumptions of \eqref{ini-1}, one has
\ba
\int_{\TT}  \eta  w ^2 \, \dx \leq \bar{ \eta }^2 \int_{\TT}  w _x^2 \, \dx.
\nn \ea}
\end{lem}

\begin{proof}
Under the assumptions of \eqref{ini-1}, one has \eqref{csv-m-m-1} for all $t>0.$ Then  we deduce from
\eqref{upperb-vtr} and Poincar\'e's inequality that
\ba
\int_{\TT}  \eta  w ^2 \, \dx  & = \int_{\TT}  \eta  w ^2  \, \dx  - \langle  \eta  w  \rangle^2\\
& = \frac12\int_{\TT} \int_{\TT}  \eta (x)  \eta (x') | w (x) -  w (x')|^2\,\dx \,\dx'\\
& \leq\frac{\bar{ \eta }^2}2 \int_{\TT} \int_{\TT}  | w (x) -  w (x')|^2\,\dx \,\dx'\\
& = \bar{ \eta }^2  \int_{\TT} |  w  - \langle  w  \rangle|^2 \,\dx \leq \bar{ \eta }^2 \int_{\TT}   w _x^2 \,\dx.
\nn\ea
\end{proof}

By combining Lemma \ref{prop-kinetic-tx}  with the  energy estimate \eqref{csv-eng-1}, we achieve
\be\label{tvr-tv2-tx}
\int_0^\infty \int_{\TT}  \eta  w ^2 \,\dx\,\dt \leq \bar{ \eta }^2 \int_0^\infty  \int_{\TT}   w _x^2 \,\dx\,\dt \leq \bar{ \eta }^2 \nu^{-1}  E_{00}.
\ee

\begin{prop}\label{prop-decay-exp-L2}
{\sl There exist positive  constants $\a>0$ and  $C$ which depend on $(a, \g, \nu, \bar \varsigma_0,  \bar E_{00})$ such that
\ba\label{decay-exp-L2-4}
\| \eta ^{\frac{1}{2}} w(t) \|_{L^2_\h} + \|(\eta -1)(t)\|_{L^2_\h} \leq C E_{00}^{\frac{1}{2}}(y) e^{-\a t}, \quad \forall \, t\in \R^{+}, \ y\in \R.
\ea
}
\end{prop}

\begin{proof} By
combining \eqref{csv-m-m} with \eqref{csv-eng}, we obtain
\ba\label{csv-L2-1}
\frac{\rm d}{\dt} \int_{\TT} \Bigl(\frac{1}{2}  \eta  w ^2 + \big(P( \eta ) - P(1) - P'(1)( \eta - 1)\big)\Bigr) \,\dx
+ \nu\int_{\TT}  w _x^2\,\dx = 0.
\ea
While we get, by using Taylor's formula, that
\be\label{P-1-0}
P( \eta ) - P(1) - P'(1)( \eta - 1) = P''(\hat \eta) ( \eta -1)^2 = \frac{p'(\hat \eta)}{\hat \eta} ( \eta -1)^2 = a \g \hat \eta^{\g-2} ( \eta -1)^2,
\nn\ee
for some $\hat\eta$ between $ \eta $ and $1$. Due to the upper bound $0\leq  \eta \leq \bar{ \eta } <\infty$, we have
$$0\leq \hat \eta \leq \max\{1,\bar{ \eta }\} \leq 1 + \bar{ \eta }.$$
So that if $\g\in [1,2]$,  we have
\ba\label{P-1-1}
P( \eta ) - P(1) - P'(1)( \eta - 1) \geq  a \g (1+\bar{ \eta })^{\g-2} ( \eta -1)^2.
\ea
While for  $\g\in [1,2]$, it is easy to observe that
\ba\label{P-1-2}
P( \eta ) - P(1) - P'(1)( \eta - 1)  \leq C(\g) a ( \eta - 1)^2,
\ea
where the constant $C$ is solely determined by $\g$. It follows from \eqref{P-1-1} and \eqref{P-1-2} that $P( \eta ) - P(1) - P'(1)( \eta - 1)$ behaves like $( \eta -1)^2$.

Notice that $I( \eta - 1)|_{x=0}=I( \eta - 1)|_{x=1}=0.$  By multiplying the momentum equation of \eqref{CNS-limit} by $ I( \eta - 1)$
and using integration by parts, one has
\beq\label{csv-L2-2}
\int_{\TT} p( \eta ) ( \eta -1) \,\dx   = \int_{\TT} ( \eta  w )_t I( \eta -1)\,\dx  - \int_{\TT} ( \eta  w ^2) ( \eta -1) \,\dx + \nu \int_{\TT}  w _x ( \eta -1) \,\dx.
\eeq
We first compute
\be\label{csv-L2-3}
\int_{\TT} ( \eta  w )_t I( \eta -1)\,\dx   =  \frac{\rm d}{\dt} \int_{\TT} ( \eta  w ) I( \eta -1)\,\dx - \int_{\TT} ( \eta  w ) \d_tI( \eta -1) \,\dx.
\nn\ee
By virtue of \eqref{csv-m-m-1} and Lemma \ref{prop-kinetic-tx}, one has
\ba
 \bigl| \int_{\TT} ( \eta  w )\d_t I( \eta -1) \,\dx \bigr| & = \bigl| \int_{\TT} ( \eta  w ) I(( \eta  w )_x)\,\dx \bigr|\\
  & = \bigl| \int_{\TT} ( \eta  w )( \eta  w (t,x) -  \eta  w (t,0)) \,\dx \bigr| \\
  & = \bigl| \int_{\TT} ( \eta  w )^2 \,\dx \bigr|  \leq \bar{ \eta }\int_{\TT}  \eta  w ^2 \,\dx \leq \bar{ \eta }^3 \int_{\TT}  w _x^2 \,\dx.
\nn\ea
Similarly, we have
\be\label{csv-L2-5}
 \bigl| \int_{\TT} ( \eta  w ^2) ( \eta -1) \,\dx \bigr|  \leq (1+\bar{ \eta }) \int_{\TT}  \eta  w ^2 \,\dx \leq (1+\bar{ \eta }) \bar{ \eta }^2 \int_{\TT}  w _x^2 \,\dx.
\ee
While applying Young's inequality yields
\be\label{csv-L2-6}
\nu \bigl| \int_{\TT}  w _x ( \eta -1) \,\dx  \bigr|  \leq \de \int_{\TT} | \eta -1|^2 \,\dx \ + \de^{-1} \nu^2 \int_{\TT}  w _x^2 \,\dx.
\nn\ee

By inserting the above estimates into \eqref{csv-L2-3}, we find
\ba\label{csv-L2-7}
\int_{\TT} p( \eta ) ( \eta -1) \,\dx &-  \frac{\rm d}{\dt} \int_{\TT} ( \eta  w ) I( \eta -1)\,\dx \\
 &\leq \de \int_{\TT} | \eta -1|^2 \,\dx + \big(\bar{ \eta }^2+\bar{ \eta }^3 + \de^{-1} \nu^2\big)\int_{\TT}  w _x^2 \,\dx.
\ea

Notice that as long as $\g \geq 1$, one has
\be\label{csv-L2-8}
 \big(p(r) - p(1)\big) (r-1)  - a (r-1)^2  = a r (r^{\g-1}-1)(r-1) \geq 0, \quad \forall\, r\geq 0.
\nn\ee
As a result, it comes out
\ba\label{csv-L2-9}
\int_{\TT} p( \eta ) ( \eta -1) \,\dx = \int_{\TT} \big(p( \eta ) - p(1)\big) ( \eta -1) \,\dx \geq a \int_{\TT} ( \eta -1)^2 \,\dx.
\nn \ea
Then choosing $\de = \frac{a}{2}$ in \eqref{csv-L2-7} yields
\ba\label{csv-L2-10}
\frac{a}{2}\int_{\TT} ( \eta -1)^2 \,\dx -  \frac{\rm d}{\dt} \int_{\TT} ( \eta  w ) I( \eta -1)\,\dx  \leq  \big(\bar{ \eta }^2+\bar{ \eta }^3 + 2 a^{-1} \nu^2\big)\int_{\TT}  w _x^2 \,\dx.
\ea

Multiplying \eqref{csv-L2-1} by a constant $A_1$ and summing up the resulting inequality with \eqref{csv-L2-10}, we get
\ba\label{csv-L2-11}
& \frac{\rm d}{\dt}\int_{\TT}  \Bigl(\frac{A_1}{2}  \eta  w ^2 + A_1 \big(P( \eta ) - P(1) - P'(1)( \eta - 1)\big)  - ( \eta  w ) I( \eta -1) \Bigr)\,\dx  \\
 &\quad + A_1 \nu \int_{\TT}  w _x^2\,\dx + \frac{a}{2}\int_{\TT} ( \eta -1)^2 \,\dx   \leq  \big(\bar{ \eta }^2+\bar{ \eta }^3 + 2 a^{-1} \nu^2\big)\int_{\TT}  w _x^2 \,\dx.
\nonumber\ea
Taking $A_1$ large enough so that
\ba\label{csv-L2-13}
 A_1 \geq 4, \quad A_1 a \g (1+\bar{ \eta })^{\g-2} \geq 2, \quad A_1 \nu \geq \bar{ \eta }^2+\bar{ \eta }^3 + 2 a^{-1} \nu^2 + 1,
\ea
we obtain
\ba\label{decay-exp-L2-1}
\frac{\rm d}{\dt}\int_{\TT} \Bigl(\frac{A_1}{2}  \eta  w ^2 + A_1 \big(P( \eta ) - P(1) - P'(1)( \eta - 1)\big)  &- ( \eta  w ) I( \eta -1)\Bigr) \,\dx  \\
 &+  \int_{\TT}  w _x^2\,\dx + \frac{a}{2}\int_{\TT} ( \eta -1)^2 \,\dx   \leq  0.
\ea
Furthermore, due to
\ba
 \int_{\TT} \bigl|( \eta  w ) I( \eta -1)\bigr| \,\dx \leq \int_{\TT} | \eta  w | \,\dx \int_{\TT} | \eta -1| \,\dx \leq \int_{\TT}  \eta  w ^2 \,\dx +  \int_{\TT} | \eta -1|^2 \,\dx,
\nn \ea
and taking into account \eqref{P-1-1} and \eqref{P-1-2}, we have
\ba\label{decay-exp-L2-2}
\int_{\TT} \Bigl( \frac{A_1}{2}  \eta  w ^2 + A_1 \big(P( \eta ) - P(1) &- P'(1)( \eta - 1)\big)  - ( \eta  w ) I( \eta -1)\Bigr) \,\dx \\
 &\geq  \int_{\TT}  \Bigl(\frac{A_1}{4}  \eta  w ^2 + \frac{A_1}{2}a \g (1+\bar{ \eta })^{\g-2} ( \eta - 1)^2 \Bigr)  \,\dx,
\ea
and
\ba\label{decay-exp-L2-3}
\int_{\TT}  \Bigl(\frac{A_1}{2}  \eta  w ^2 + A_1 \big(P( \eta ) - P(1)& - P'(1)( \eta - 1)\big)  - ( \eta  w ) I( \eta -1)\Bigr) \,\dx \\
 &\leq  \int_{\TT} \Bigl(\big(\frac{A_1}{2} + 1 \big)  \eta  w ^2 + \big(A_1 a C(\g) +1 \big) ( \eta - 1)^2\Bigr)   \,\dx.
\ea
 Then \eqref{decay-exp-L2-4} follows from \eqref{decay-exp-L2-1}--\eqref{decay-exp-L2-3} and Lemma \ref{prop-kinetic-tx}.
\end{proof}

\subsection{Lower bound of the density function}

\begin{prop}\label{prop-lowerbd-density}
{\sl There exists a positive constant $\underline  \eta $ which depends on  $(a, \g, \nu, \bar \varsigma_0 , \underline \varsigma_0, \bar E_{00})$  so that
\be\label{lowerb-vtr}
 \eta (t,x,y) \geq \underline  \eta , \quad \forall \, (t,x,y)\in \R^{+}\times \TT\times \R.
\ee
}
\end{prop}

\begin{proof}
Let
$
\frak{y}_1 \eqdefa  \log \frac{1}{ \eta } = - \log \eta .
$
Then it follows from \eqref{tI-mm-4} that
\ba\label{density-lower1}
 D_t\big(\nu \frak{y}_1 - \tilde I( \eta  w )\big) =  p( \eta ) - \langle  \eta  w ^2 +  p( \eta )  \rangle,
\nn \ea
from which, we deduce that for each $t\in\R^{+}$,
\ba\label{density-lower2}
 \max_{x\in \TT} \big(\nu \frak{y}_1 - \tilde I( \eta  w )\big)(t,\cdot) &\leq   \max_{x\in \TT} \big(\nu \frak{y}_1 - \tilde I( \eta  w )\big)(0,\cdot)  + \int_0^t   \max_{x\in \TT}\big(p( \eta )  + \langle  \eta  w ^2 +  p( \eta )  \rangle\big) \dt'\\
 & \leq - \nu \log\underline \varsigma_0 + E_{00}^{\frac{1}{2}} + (a \bar{ \eta }^\g + \g E_{00}) t.
 \nonumber
\ea
This implies
\ba
 \frak{y}_1 (t) \leq  -\log\underline \varsigma_0 + 2 \nu^{-1}E_{00}^{\frac{1}{2}} +  \nu^{-1} (a \bar{ \eta }^\g + \g E_{00}) t,
\nn \ea
which is equivalent to
\ba\label{density-lower4}
 \eta (t) &\geq  \underline \varsigma_0 \exp\big(- 2 \nu^{-1}\bar E_{00}^{\frac{1}{2}} -  \nu^{-1} (a \bar{ \eta }^\g + \g \bar E_{00}) t \big) \eqdefa
 \underline{ \eta }_1(t).
\ea

To obtain a lower bound of the density function for large time, we
define
$$
\frak{y}_2 = \exp(\l(\nu\frak{y}_1 -  \tilde I( \eta  w ))), \quad \l >0
$$
which solves
\ba\label{density-lower5}
 D_t \frak{y}_2 &  =  \l \exp(\l(\nu\frak{y}_1 -  \tilde I( \eta  w ))) \big(p( \eta ) - \langle  \eta  w ^2 +  p( \eta )  \rangle\big)\\
 & =  - \l \frak{y}_2  \langle  \eta  w ^2 +  p( \eta )  \rangle + \l   \eta ^{-\l\nu} a  \eta ^\g \exp( -  \l \tilde I( \eta  w )).
\nonumber\ea
Taking $\l$ so that $\l \nu = \g$ gives rise to
\ba\label{density-lower6}
 D_t \frak{y}_2 + \nu^{-1} \g \frak{y}_2  \langle  \eta  w ^2 +  p( \eta )  \rangle  =   a \nu^{-1} \g  \exp( - \l \tilde I( \eta  w ) ).
\ea
Yet we observe from   Proposition \ref{prop-decay-exp-L2} that
\ba\label{density-lower6.5}
\langle p( \eta ) \rangle  &  =   \int_{\TT}  p( \eta )  \, \dx
 \geq  p(1) - \int_{\TT}  |p( \eta ) - p(1)| \, \dx \\
& \geq p(1) - p'(\bar{ \eta }) \int_{\TT}  | \eta - 1| \, \dx \geq a - a \g \bar{ \eta }^{\g - 1} \big(C e^{-\a t}\big)^{\frac{1}{2}}.
\ea
Let $T_1$ be such that
\ba\label{T1-def}
 \g \bar{ \eta }^{\g - 1} \big(C e^{-\a T_1}\big)^{\frac{1}{2}}  = \frac{1}{2}.
\ea
Then for all $t\geq T_1$, there holds
$$
\langle  \eta  w ^2 +  p( \eta )  \rangle  \geq \langle p( \eta )  \rangle  \geq \frac{a}{2},
$$
so that we deduce from \eqref{density-lower6} that  for $t\geq T_1$,
\ba\label{density-lower7}
 \frak{y}_2 (t) \leq e^{- \tilde a t}\frak{y}_2(T_1) + \int_{T_1}^t e^{- \tilde a (t-t')}  a \nu^{-1} \g  \exp( - \l \tilde I( \eta  w ) )(t')\,\dt',
\ea
where $\tilde a\eqdefa \frac{a \nu^{-1} \g}{2}>0$. Together with  \eqref{I(rv)}, \eqref{density-lower7} ensures that
for  $t\geq T_1$,
\ba
 \frak{y}_2 (t) \leq e^{- \tilde a t}\frak{y}_2 (T_1) + a\left(\tilde a \nu\right)^{-1} \g   e^{2\l E_{00}^{\frac 12}},
\nn \ea
from
which, we infer
\ba\label{density-lower10}
 \eta ^{-\lambda\g} (t)   \leq e^{- \tilde a t} \eta ^{-\lambda\g} (T_1) e^{2 \l E_{00}^{\frac 12}}  +
  a\left(\tilde a \nu\right)^{-1} \g  e^{4 \l E_{00}^{\frac 12}},
\nn \ea
 that is,
 for  $t\geq T_1$  there holds
\ba\label{density-lower12}
 \eta (t)   \geq  \underline{ \eta }_2(T_1) \eqdefa \Big( e^{- \tilde a T_1}\underline{ \eta }_1^{-\lambda\g} (T_1) e^{2 \l \bar E_{00}^{\frac 12}}
 + a \left(\tilde a\nu\right)^{-1} \g  e^{4 \l \bar E_{00}^{\frac 12}}\Big)^{-\frac{1}{\lambda\g}}.
\ea

Combining \eqref{density-lower4} and \eqref{density-lower12}, we deduce that \eqref{lowerb-vtr} holds with
\ba\label{density-lower13}
\underline{ \eta } \eqdefa \min\left\{\underline{ \eta }_1(T_1),\underline{ \eta }_2(T_1) \right\}.
\ea
This completes the proof of the proposition.
\end{proof}

\subsection{Decay estimate of $\| w _x(t)\|_{L^2_\h}$}

It follows from  \eqref{ass-ini-1} that
\be\label{ini-3}
  E_{10} (y) \eqdefa  \|\varsigma_0-1\|_{H^{1}_\h}^2 + \| w _0\|_{H^{1}_\h}^2 \in ( L^{1} \cap L^{\infty})(\R)
\andf \bar E_{10} \eqdefa  \sup_{y\in \R}E_{10} (y)  <\infty.
\ee

We start with the following two lemmas:
\begin{lem}\label{prop-est-rv4}
{\sl There holds for all $t\in \R^{+}$
\be\label{est-rv4}
\frac{\rm d}{\dt}\int_{\TT}  \eta  w ^4 \,\dx + 6 \nu \int_{\TT}  | w  |^2 | w _x|^2\,\dx \leq 24a \nu^{-1}\bar{ \eta }^{2\g-1} \int_{\TT}  w _x^2\,\dx.
\ee}
\end{lem}

\begin{proof}
Multiplying $4  w ^3$ to the momentum equation of \eqref{CNS-limit} and integrating the resulting equation  over $\TT$ gives
\ba\label{est-rv4-1}
\frac{\rm d}{\dt}\int_{\TT}  \eta  w ^4 \,\dx + 12 \nu \int_{\TT}  | w  |^2 | w _x|^2\,\dx   & = 12 \int_{\TT} p( \eta )  w ^2  w _x\,\dx  \\
& \leq 6\nu \int_{\TT}  | w  |^2 | w _x|^2\,\dx + 24a \nu^{-1} \int_{\TT}   \eta ^{2\g}  w ^2\,\dx \\
& \leq 6\nu \int_{\TT}  | w  |^2 | w _x|^2\,\dx + 24 a\nu^{-1} \bar{ \eta }^{2\g-1} \int_{\TT}  w _x^2\,\dx,
\nonumber\ea
where we used Lemma \ref{prop-kinetic-tx} in the last step. And \eqref{est-rv4} follows.
\end{proof}

\begin{lem}\label{prop-est-vx}
{\sl There exist  positive  constants $B_1$ and $B_2$ solely depending on $(a, \g, \nu , \bar\varsigma_0, \bar E_{00})$, so that for all $t\in \R^{+}$
\ba\label{est-vx}
  \frac{\rm d}{\dt} \int_{\TT}\Bigl( \frac\nu2  w _x^2  + \frac{a^2}{2\nu} \big( \eta ^{2\g} & - 1 - 2\g ( \eta -1)\big) -  \big(p( \eta ) - p(1)\big)  w _{x}\Bigr) \,\dx + \frac{1}{2} \int_{\TT}  \eta  w _t^2 \,\dx \\
  & \qquad\quad \leq 5 \bar{ \eta }\int_{\TT} | w  |^2 |  w _x |^2 \,\dx  + B_{1} \int_{\TT}  w _x^2 \,\dx + B_2 \int_{\TT} ( \eta - 1)^2 \,\dx.
\ea}
\end{lem}

\begin{proof}
Multiplying $ w _t$ to the momentum equation of \eqref{CNS-limit} and integrating the resulting equation over $\TT$ gives
\be\label{est-vx-1}
\int_{\TT} ( \eta  w )_t  w _t \,\dx + \int_{\TT} ( \eta  w ^2)_x  w _t \,\dx + \frac\nu2 \frac{\rm d}{\dt} \int_{\TT}  w _x^2 \,\dx  + \int_{\TT} p( \eta )_x  w _t \,\dx  = 0.
\ee
We now handle term by term above. Firstly, it follows from the continuity equation of \eqref{CNS-limit} that
\begin{align*}\label{est-vx-2}
\int_{\TT} ( \eta  w )_t  w _t \,\dx
& = \int_{\TT}  \eta  w _t^2 \,\dx -  \int_{\TT} ( \eta  w ^2)_x   w _t \,\dx + \int_{\TT}  \eta  w    w _x   w _t \,\dx.
\end{align*}
This implies
\ba\label{est-vx-3}
\int_{\TT} ( \eta  w )_t  w _t \,\dx + \int_{\TT} ( \eta  w ^2)_x  w _t \,\dx  =& \int_{\TT}  \eta  w _t^2 \,\dx + \int_{\TT}  \eta  w    w _x   w _t \,\dx\\
\geq & \frac{3}{4} \int_{\TT}  \eta  w _t^2 \,\dx - 4 \bar{ \eta }\int_{\TT} | w  |^2 |  w _x |^2 \,\dx.
\ea

It is rather complicated to estimate the term related to the pressure in \eqref{est-vx-1}. We compute
\beq\label{est-vx-4}
\int_{\TT} p( \eta )_x  w _t \,\dx
 = - \frac{\rm d}{\dt} \int_{\TT} \big(p( \eta ) - p(1)\big)  w _{x} \,\dx + \int_{\TT} p( \eta )_t  w _{x} \,\dx.
\eeq
For the last term of \eqref{est-vx-4}, we decompose it as
\begin{align*}\label{est-vx-5}
 \int_{\TT} p( \eta )_t  w _{x} \,\dx & =  \nu^{-1}\int_{\TT} p( \eta )_t \big(\nu w _{x} - p( \eta )\big) \,\dx + \nu^{-1}\int_{\TT} p( \eta )_t p( \eta ) \,\dx \\
  & =  \nu^{-1}\int_{\TT} p( \eta ) w   \big(\nu w _{xx} - p( \eta )_x\big) \,\dx  -  a (\g-1)\nu^{-1}\int_{\TT} \eta ^\g  w _x \big(\nu w _{x} - p( \eta )\big) \,\dx \\
  & \quad +  \frac{a^2}{2\nu}\frac{\rm d}{\dt}\int_{\TT} \big( \eta ^{2\g}  - 1 - 2\g ( \eta -1)\big)\,\dx.
\end{align*}
Observing that
\ba
 \nu^{-1}\int_{\TT} p( \eta ) w   \big(\nu w _{xx} - p( \eta )_x\big) \,\dx= \nu^{-1}\int_{\TT} p( \eta ) w    \eta \big( w _t +  w   w _x)\big) \,\dx.
 \nn \ea
Applying Young's inequality yields
\ba
\nu^{-1}\bigl|\int_{\TT} p( \eta ) w    \eta  w _t \,\dx \bigr|
& \leq \frac{1}{4}\int_{\TT}  \eta  w _t^2 \,\dx + 4 \nu^{-2} a \bbeta^{2\g+1} \int_{\TT}  w _x^2 \,\dx,\\
\nu^{-1}\bigl| \int_{\TT} p( \eta ) w    \eta  w   w _x \,\dx  \bigr|
& \leq  a^2\nu^{-2} \bbeta^{2\g+2} \int_{\TT}  w _x^2 \,\dx + \bbeta \int_{\TT}  | w |^2 | w _x|^2 \,\dx,
\nonumber
\ea
where we used Lemma \ref{prop-kinetic-tx} in the second inequality.

We finally compute
\ba
 a (\g-1)\bigl| \int_{\TT} \eta ^\g  w _x^2  \,\dx \bigr|  \leq a (\g-1) \bbeta^\g \int_{\TT}  w _x^2  \,\dx,
\nn \ea
and
\begin{align*}
\nu^{-1}a (\g-1)\bigl|  \int_{\TT}  \eta ^\g  w _x  p( \eta ) \,\dx\bigr|
& =  \nu^{-1}  a (\g-1)  \bigl| \int_{\TT} ( \eta ^{2\g} -1)  w _x  \,\dx \bigr| \\
& \leq  \nu^{-1}  a (\g-1) \Big(  \int_{\TT} ( \eta ^{2\g} -1)^2  \,\dx +  \int_{\TT}   w _x^2  \,\dx\Big) \\
& \leq \nu^{-1} a (\g-1)  \Big( \bigl(2\g \bbeta^{2\g-1}\bigr)^2 \int_{\TT} ( \eta -1)^2  \,\dx +  \int_{\TT}   w _x^2  \,\dx\Big).
\end{align*}

By inserting the above estimates into \eqref{est-vx-4}, we achieve
\beq\label{est-vx-6}
\begin{split}
\int_{\TT} p( \eta )_x  w _t \,\dx\geq  & \frac{\rm d}{\dt} \int_{\TT} \Bigl( \frac{a^2}{2\nu} \big( \eta ^{2\g}  - 1 - 2\g ( \eta -1)\big)
-  \big(p( \eta ) - p(1)\big)  w _{x}\Bigr) \,\dx\\
&-\Bigl(\bar{ \eta }\int_{\TT} | w  |^2 |  w _x |^2 \,\dx+\frac14\int_{\TT}\eta w_t^2\,\dx
+B_{1} \int_{\TT}  w _x^2 \,\dx + B_2 \int_{\TT} ( \eta - 1)^2 \,\dx\Bigr),
\end{split}
\eeq
where
\ba
&B_1\eqdefa 4 \nu^{-2} a \bbeta^{2\g+1} + \nu^{-2} a^2 \bbeta^{2\g+2}  + a (\g-1) \bigl(\bbeta^\g  + \nu^{-1}\bigr)
 \andf B_2\eqdefa 4a \g^2(\g-1) \nu^{-1 }  \bbeta^{2(2\g-1)}.
\nn \ea

By substituting \eqref{est-vx-3} and \eqref{est-vx-6} into \eqref{est-vx-1}, we achieve \eqref{est-vx}.
\end{proof}

  By multiplying \eqref{est-rv4} by a large enough positive constant $A_2,$ which
   depends on $(\nu, \bbeta),$ and summing up the resulting inequality
  with \eqref{est-vx}, we obtain

\begin{corollary}
{\sl Let $A_2$ and $B_3$ be determined by
\ba\label{A2-B3}
 6 A_2 \nu  = 5\bbeta + 1, \quad B_3 = B_1 + 24 A_2 \nu^{-1} a \bar{ \eta }^{\g+1}.
\nn \ea There holds
\ba\label{est-vx-rv4}
 & \frac{\rm d}{\dt}   \int_{\TT} \Bigl(A_2  \eta  w ^4 + \nu  w _x^2  + \frac{a^2}{2\nu} \big( \eta ^{2\g}  - 1 - 2\g ( \eta -1)\big) -  \big(p( \eta ) - p(1)\big)  w _{x} \Bigr)\,\dx \\
  & \qquad\quad + \frac{1}{2} \int_{\TT}  \eta  w _t^2 \,\dx + \int_{\TT} | w  |^2 |  w _x |^2 \,\dx \leq  B_{3} \int_{\TT}  w _x^2 \,\dx + B_2 \int_{\TT} ( \eta - 1)^2 \,\dx.
\ea
 }
\end{corollary}

\begin{prop}\label{prop-decay-exp-H1}
{\sl There exist positive constants $\a$ and $C$ solely depending on $(a, \g, \nu, \bar \varsigma_0, \underline \varsigma_{0}, \bar E_{00})$ such that
\ba\label{decay-exp-H1-2}
\| w(t)\|_{H^1_\h} + \| w(t) \|_{C^{0,\frac{1}{2}}_\h}  \leq C E_{10}^{\frac{1}{2}}(y) e^{-\a t}, \quad \forall \, t\in \R^{+}, \ y\in \R.
\ea
}
\end{prop}

\begin{proof}
Thanks to  \eqref{decay-exp-L2-1} and \eqref{est-vx-rv4}, we deduce that
there exists a  large enough  positive constant $A_3$  depending only on  $(a, \g, \nu , \bar\varsigma_0, \bar E_{00})$ so that
\ba\label{decay-exp-H1-1}
& \frac{\rm d}{\dt} F_2(t)
  +   \int_{\TT} \bigl( w _x^2 +  ( \eta -1)^2 +  | w  |^2 |  w _x |^2  + \frac{1}{2} \eta  w _t^2\bigr) \,\dx  \leq  0,
\ea
where
\ba\label{decay-exp-H1-3}
F_{2}(t)   \eqdefa \int_{\TT} \Bigl(&\frac{A_1 A_3}{2}  \eta  w ^2 + A_1 A_3 \big(P( \eta ) - P(1) - P'(1)( \eta - 1)\big) - A_3 ( \eta  w ) I( \eta -1) \\
&+  A_2  \eta  w ^4 + \nu  w _x^2   + \frac{a^2}{2\nu} \big( \eta ^{2\g}  - 1 - 2\g ( \eta -1)\big) -  \big(p( \eta ) - p(1)\big)  w _{x} \Bigr) \,\dx.
\ea
 Observing that
\begin{align*}\label{decay-exp-H1-4}
\bigl|\int_{\TT} \big(p( \eta ) - p(1)\big)  w _{x}\,\dx \bigr| & \leq (2 \nu)^{-1}\int_{\TT} |p( \eta ) - p(1)|^2 \,\dx  +  \frac{\nu}{2}\int_{\TT} | w _{x}|^2\,\dx \\
&\leq (2 \nu)^{-1}p'(\bbeta)^2 \int_{\TT} | \eta - 1|^2 \,\dx  +  \frac{\nu}{2}\int_{\TT} | w _{x}|^2\,\dx.
\end{align*}
By choosing $A_3$ sufficiently large, we have
\ba\label{decay-exp-H1-5}
 \int_{\TT}  \left(\eta  w ^2 + ( \eta -1)^2 +   \eta  w ^4 +   w _x^2\right)
  \,\dx \leq  F_{2}  \leq C \int_{\TT} \left( \eta  w ^2 + ( \eta -1)^2 +   \eta  w ^4 +  w _x^2\right) \,\dx,
\ea
where $C$  solely depends on $(a, \g, \nu , \bar\varsigma_0, \bar E_{00})$. Then \eqref{decay-exp-H1-2} follows from \eqref{decay-exp-H1-1},
\eqref{decay-exp-H1-5} and \eqref{decay-exp-L2-4}.
\end{proof}

We remark that up to now, the strictly lower bound of $\eta$ in \eqref{lowerb-vtr} is not really needed.
Indeed let us  recall

 \begin{lem}[Lemma 3.2 in \cite{F-book}]\label{lem-Poincare}
 {\sl Let $\Omega\subset \R^d$ be a bounded domain with $d\geq 1.$ Let $\rho$ be a non-negative function satisfying
 $$
 \int_\Omega \rho \,\dx \geq M >0 \andf \int_\Omega \rho^q  \,\dx \leq  E_0 <\infty,
 $$
 for some $q>1.$ Then for each $u\in H^1(\Omega)$, there holds
 $$
 \|u\|_{L^2(\Omega)}^2 \leq C(M, E_0) \left(\|\nabla u\|_{L^2(\Omega)}^2 + \Big(\int_\Omega \rho |u|\,\dx \Big)^2\right).
 $$
}
 \end{lem}

\begin{remark}
By integrating \eqref{decay-exp-H1-1} over $[0,t],$ we obtain for all $t\in \R^{+}$
\ba\label{decay-exp-H1-6}
\int_{\TT}  \bigl(\eta  w ^2 + ( \eta -1)^2 +   \eta  w ^4 +   w _x^2\bigr)(t) \,\dx
&+ \int_0^t \int_{\TT} \bigl(( \eta -1)^2  + | w  |^2 |  w _x |^2  + \frac{1}{2}  \eta  w _t^2\bigr) \,\dx\,\dt'   \\
& \leq C \int_{\TT}  \bigl(\eta _0  w _0^2 + ( \eta _0 -1)^2 +   \eta _0  w _0^4 +  w _{0,x}^2\bigr) \,\dx < \infty,
\nonumber
\ea
from which and Lemma \ref{lem-Poincare}, we infer
$$
\int_{\TT} |w|^{2} \, \dx\leq \int_{\TT} \bigl(\eta  w ^2 +   w _x^2\bigr) \,\dx \in L^{\infty}(\R_{+}).
$$
Then along the same lines of proof of \eqref{decay-exp-H1-2} , yet without using the positive lower bound of $\eta$,
we have
$$
\bigl\|\bigl( \eta ^{\frac{1}{2}} w,  \eta -1, \eta ^{\frac{1}{2}} w ^2, w _x\bigr)\bigr\|_{L^2_\h} + \| w \|_{L^\infty_\h}  \leq C E_{10}^{\frac{1}{2}}(y) e^{-\a t}, \quad \forall \, t\in \R^{+},  \ y\in \R,
$$
where $\a>0$ solely depending on $(a, \g, \nu, \bar \varsigma_0)$ and $C>0$ solely depending on $(a, \g, \nu, \bar \varsigma_0, \bar E_{00})$.

\end{remark}

\subsection{Decay estimate of $\| \eta _x(t)\|_{L^2_x}$}
From this subsection on,
 we need to use the lower bound of $ \eta $ obtained in Proposition \ref{prop-lowerbd-density}.

\begin{prop}\label{prop-decay-exp-rx}
{\sl
There exist positive constants $C$ and $\a$  depending  solely on $(a, \g, \nu, \bar \varsigma_0, \bar E_{00}, \underline{\varsigma}_0)$  so that
\be\label{decay-exp-rx}
\| \eta_x(t)\|_{L^2_\h} \leq C E_{10}^{\frac{1}{2}} e^{-\a t}, \quad \forall\, t\in \R^{+}.
\ee
}
\end{prop}

\begin{proof} Let  $\zeta \eqdefa  \eta ^{-1}.$ By  multiplying $- \eta ^{-2}$ to the density equation
of \eqref{CNS-limit}, we get
\ba\label{decay-exp-rx-1}
\zeta_t +  w  \zeta_x - \zeta  w _x = 0.
\nn \ea
Applying $\d_x$ to the above equation gives
\be\label{decay-exp-rx-2}
\zeta_{tx} +  w  \zeta_{xx} - \zeta  w _{xx} = 0 \Longrightarrow  \zeta  w _{xx} = \zeta_{tx} +  w  \zeta_{xx}.
\ee
While we rewrite the momentum equation of \eqref{CNS-limit} as
 \ba\label{decay-exp-rx-4}
 w _t +  w   w _x - \nu (\zeta_{tx} +  w  \zeta_{xx}) - \zeta^{-1} p'( \eta )\zeta_x = 0.
\nn \ea
Multiplying the above equation by $\eta$ yeilds
\be\label{decay-exp-rx-6}
 \eta \big(( w  - \nu \zeta_x)_t +  w  ( w  - \nu \zeta_x)_x \big) -  \eta ^2 p'( \eta )\zeta_x = 0.
\ee
By multiplying $( w  - \nu \zeta_x)$ to \eqref{decay-exp-rx-6} and integrating the resulting equation over $\TT,$ we find
\be\label{decay-exp-rx-7}
\frac{1}{2}\frac{\rm d}{\dt }\int_{\TT}  \eta ( w  - \nu \zeta_x)^2\,\dx  - \int_{\TT}  \eta ^2 p'( \eta )\zeta_x ( w  -\nu\zeta_x)\,\dx  = 0.
\ee
We compute
\begin{align*}\label{decay-exp-rx-8}
- \int_{\TT}  \eta ^2 p'( \eta )\zeta_x ( w -\nu\zeta_x)\,\dx
& =\nu a \g  \int_{\TT}  \eta ^{\g+1} \zeta_x^2 \,\dx +  \int_{\TT}  p( \eta )_x  w  \,\dx.
\end{align*}
In view of  Proposition \ref{prop-lowerbd-density}, we have $ \eta \geq \underline{ \eta }>0,$  so that
\ba\label{decay-exp-rx-9}
 \nu a \g  \int_{\TT} \eta ^{\g+1} \zeta_x^2 \,\dx \geq \nu a \g \underline{ \eta }^{\g+1} \int_{\TT} \zeta_x^2 \,\dx.
\nn \ea
We observe that
\begin{align*}
\bigl|\int_{\TT}  p( \eta )_x  w  \,\dx\bigl| & = \bigl|\int_{\TT}  p( \eta )  w _x \,\dx\bigl|  = \bigl|\int_{\TT}  (p( \eta )-p(1))  w _x \,\dx\bigl| \\
& \leq \int_{\TT}  (p( \eta )-p(1))^2\,\dx + \int_{\TT}  w _x^2 \,\dx \\
& \leq  a^2 \g^2 \bbeta^{2\g-2} \int_{\TT}  ( \eta -1)^2\,\dx + \int_{\TT}  w _x^2 \,\dx.
\end{align*}

By substituting the above estimates into \eqref{decay-exp-rx-7}, we achieve
\be\label{decay-exp-rx-0}
\frac{1}{2}\frac{\rm d}{\dt }\int_{\TT}  \eta ( w  - \nu \zeta_x)^2\,\dx  + \nu a \g \underline{ \eta }^{\g+1} \int_{\TT} \zeta_x^2 \,\dx \leq a^2 \g^2 \bbeta^{2\g-2} \int_{\TT}  ( \eta -1)^2\,\dx + \int_{\TT}  w _x^2 \,\dx.
\ee

Let $A_4$ be a  large enough positive constant  so that
\be\label{A4}
A_4 \geq 4 + \nu \andf \frac{a}{2}A_4  \geq  \a^2 \g^2 \bbeta^{2\g-2} + 1.
\ee
Then by multiplying \eqref{decay-exp-L2-1} by $A_4$ and summing up the resulting inequality with  \eqref{decay-exp-rx-0}, we obtain
\ba\label{decay-exp-rx-11}
 \frac{\rm d}{\dt} F_3(t) +  \int_{\TT} \bigl( w _x^2 + ( \eta -1)^2  +& \nu a \g \underline{ \eta }^{\g+1} \zeta_x^2\bigr) \,\dx   \leq  0
\with\\
F_3(t) \eqdefa \int_{\TT} \Bigl(\frac{A_1A_4}{2}  \eta  w ^2 +& A_1 A_4 \big(P( \eta ) - P(1) - P'(1)( \eta - 1)\big)\\
  &- A_4( \eta  w ) I( \eta -1) + \frac{1}{2} \eta ( w  - \nu \zeta_x)^2\Bigr)\,\dx.
\ea
Notice that
\ba
2 a^2 + (a-b)^2  = a^2 + 2 (a-\frac{b}{2})^2 + \frac{b^2}{2} \geq  a^2 + \frac{b^2}{2},\nn
\ea
by choosing $A_4$ sufficiently large, we find
\ba\label{decay-exp-rx-13}
 \int_{\TT}  \bigl(\eta  w ^2 + ( \eta - 1)^2  + \frac{\nu \underline{ \eta }}{4} \zeta_x^2\bigr)\,\dx
 \leq F_3(t) \leq C \int_{\TT}  \bigl(\eta  w ^2 + ( \eta - 1)^2 +  \frac{\nu}{2} \zeta_x^2\bigr)\,\dx,
\ea
where $C$ depends solely on $(a,\g, \nu , \bar\varsigma_0, \bar E_{00})$.

It follows from \eqref{decay-exp-rx-11}--\eqref{decay-exp-rx-13} that
\ba\label{decay-exp-rx-14}
 \|\zeta_x(t)\|_{L^2_\h} \leq CE_{10}^{\frac12} \e^{-\a t}, \quad \forall\, t\in \R^{+},
\nn \ea
with $\a$ and $C$ satisfying the assumptions in the proposition.
 Then \eqref{decay-exp-rx} follows from the fact that
$$
\| \eta _x(t)\|_{L^2_\h} = \| \eta ^{2} \zeta_x(t) \|_{L^2_\h} \leq \bbeta^2 \|\zeta_x(t)\|_{L^2_\h}.
$$
This completes the proof of Proposition \ref{prop-decay-exp-rx}.
\end{proof}

\subsection{Decay estimates of $H^2$ norms}
We first deduce from  \eqref{ass-ini-1} that
\be\label{ini-4}
  E_{20} (y) \eqdefa \|\varsigma_0-1\|_{H^{2}_\h}^2 + \| w _0\|_{H^{2}_\h}^2 \in (L^1 \cap  L^{\infty})(\R)\andf
 \bar E_{20} \eqdefa  \sup_{y\in \R}E_{20} (y)  <\infty.
\ee

\begin{lem}\label{prop-vxx-L2}
{\sl We have
\be\label{vxx-L2-1}
\int_0^\infty \int_{\TT}  w _{xx}^2\,\dx\,\dt \leq C E_{10},
\ee
where $C$ solely depends on $(a, \g, \nu, \bar \varsigma_0, \bar E_{00}, \underline{\varsigma}_0)$.}
\end{lem}

\begin{proof} Indeed by multiplying $ \eta ^{-1}  w _{xx}$ to the momentum equation  of \eqref{CNS-limit} and integrating
 the resulting equation over $\TT,$ we find
\ba\label{vxx-L2}
\frac{\rm d}{\dt} \int_{\TT}  w _x^2 \, \dx + \frac{\nu}{\bbeta} \int_{\TT}  w _{xx}^2\,\dx \leq \frac{4 \bbeta}{\nu}\Big(\int_{\TT}  w ^2 w _x^2 \,\dx + a^2 \g^2 \int_{\TT}  \eta ^{2\g - 4}  \eta _x^2\,\dx \Big).
\ea
Yet by virtue of \eqref{decay-exp-H1-2} and \eqref{decay-exp-rx}, we have
\ba
& \int_{\TT}  w ^2 w _x^2(t) \,\dx \leq \| w(t) \|_{L^\infty_\h}^2 \| w _x(t)\|_{L^2_\h}^2 \leq C  E_{10} e^{-\a  t}, \\
&\int_{\TT}  \eta ^{2\g - 4}  \eta _x^2(t)\,\dx \leq \underline{ \eta }^{2\g - 4} \| \eta _x(t)\|_{L^2_\h}^2 \leq C E_{10}  e^{-\a  t},
\nn\ea
where $C$ solely depends on $(a, \g, \nu, \bar \varsigma_0, \bar E_{00}, \underline{\varsigma}_0)$.
Then integrating \eqref{vxx-L2} over $\R^+$ leads to \eqref{vxx-L2-1}.
\end{proof}

\begin{prop}\label{prop-Dtv-L2}
{\sl Let   $D_t\eqdefa\d_t +  w  \d_x$ be the material derivative. Then there exist positive constants  $C$ and $\a$  depending
 on $(a, \g, \nu, \bar \varsigma_0, \underline{\varsigma}_0,\bar E_{10})$ so that
\be\label{Dtv-L2-1}
\int_{\TT}  \eta |D_t w |^2 \, \dx\leq C  E_{20} (y) e^{-\a  t} \andf \int_0^\infty \int_{\TT} |(D_t w )_x|^2\,\dx\,\dt \leq C  E_{20}(y).
\ee
}
\end{prop}

\begin{proof}
Applying the material derivative $D_t$ to the momentum equation of \eqref{CNS-limit} gives
\ba\label{Dtv-L3}
 \eta D_t^2  w  + D_t  \eta D_t  w  - \nu D_t  w _{xx} + D_t p( \eta )_x = 0.
\nn \ea
We compute
\ba
D_t  \eta D_t  w   = -  \eta  w _x D_t  w  = - w _x (\nu  w _{xx} - p( \eta )_x) = - \frac{\nu}{2}( w _x^2)_x +  w _x p( \eta )_x,
\nn \ea
and
\ba
- \nu D_t  w _{xx} = -\nu ( w _{txx} +  w   w _{xxx})   = -\nu (D_t w )_{xx} + \frac{3\nu}{2}( w _x^2)_x.
\nn \ea
As a result, it comes out
\ba\label{Dtv-L7}
 \eta D_t^2  w  -\nu (D_t w )_{xx} +\nu ( w _x^2)_x  + \big( w  p( \eta )_x \big)_x + p( \eta )_{tx} = 0.
\ea
Multiplying \eqref{Dtv-L7} by $D_t  w $ and integrating the resulting equation over $\TT$ yields
\ba\label{Dtv-L8}
\frac{1}{2}& \frac{\rm d}{\dt} \int_{\TT}   \eta |D_t w |^2 \, \dx + \nu \int_{\TT} |(D_t w )_x|^2\,\dx \\
 & = \nu\int_{\TT}  w _x^2 (D_t w )_x\,\dx + \int_{\TT}  w  p( \eta )_x (D_t w )_x\,\dx + \int_{\TT} p( \eta )_t (D_t w )_x\,\dx\\
& \leq \frac{\nu}{2}\int_{\TT} |(D_t w )_x|^2\,\dx + 8 \nu^{-1} \int_{\TT}  w _x^4\,\dx + C \int_{\TT}  (\eta _x^2 +  w _x^2)\,\dx,
\ea
where $C$ solely depend on $(a, \g, \nu, \bar \varsigma_0, \bar E_{00}, \underline{\varsigma}_0)$, and we used the uniform boundedness of $\| \eta \|_{L^\infty}$ and $\| w \|_{L^\infty}$.

While by applying  Sobolev embedding theorem and H\"older inequality, one has
\ba\label{Dtv-L9}
 \| w _x\|_{L^4_\h}^4 \leq   \| w _x\|_{L^2_\h}^2  \| w _x\|_{L^\infty_\h}^2 \leq C \| w _x\|_{L^2_\h}^2 \| w _x\|_{H^1_\h}^2 \leq C e^{- 2\a  t} \big(\| w _x\|_{L^2_\h}^2 + \| w _{xx}\|_{L^2_{\h}}^2\big).
\nn \ea
Then we deduce from   \eqref{Dtv-L8} that there exists $C$ solely depending on $(a, \g, \nu, \bar \varsigma_0, \bar E_{00}, \underline{\varsigma}_0,\bar E_{10})$ so that
\ba\label{Dtv-L2}
\frac{\rm d}{\dt} \int_{\TT}  \eta |D_t w |^2 \, \dx + \nu \int_{\TT} |(D_t w )_x|^2\,\dx \leq C \int_{\TT} \bigl( w _x^2 +  w _{xx}^2 +  \eta _x^2\bigr)\,\dx.
\ea

By multiplying \eqref{vxx-L2} by a sufficiently large constant $A_5$ and summing up the resulting  inequality with \eqref{Dtv-L2}, we get
\ba\label{vxx-Dtv-L2}
\frac{\rm d}{\dt} \int_{\TT} \bigl(A_5  w _x^2 +  \eta |D_t w |^2\bigr) \, \dx +  \int_{\TT}  w _{xx}^2\,\dx  + \nu \int_{\TT} |(D_t w )_x|^2\,\dx  \leq C \int_{\TT}  \bigl(w _x^2 +  \eta _x^2\bigr)\,\dx.
\ea
Observing that
\ba\label{vxx-Dtv-L2-2}
 \int_{\TT}   \eta |D_t w |^2 \, \dx  \leq  C \int_{\TT}   \bigl(\eta  w _t^2 +  \eta  w _x^2\bigr) \, \dx.
\nn \ea
Then by virtue of  \eqref{decay-exp-H1-1} and \eqref{decay-exp-rx-11}, we can find a large enough constant $A_6$ such that the quantity
\ba\label{def-F4}
F_4(t)\eqdefa A_6 F_3(t) + \int_{\TT} \bigl( A_5  w _x^2 +  \eta |D_t w |^2 \bigr)\, \dx
\ea
satisfies
\begin{align*}\label{vxx-Dtv-L2-4}
\int_{\TT}  \bigl(\eta  w ^2 + ( \eta - 1)^2 + &  \eta _x^2  +  \eta  w _x^2 +  w _x^2 +  \eta |D_t w |^2\bigr)\,\dx  \leq F_4(t) \\
&  \leq C \int_{\TT} \bigl( \eta  w ^2 + ( \eta - 1)^2 +   \eta _x^2 +  \eta  w _x^2 +  \eta |D_t w |^2\bigr)\,\dx,
\end{align*}
and
\ba\label{vxx-Dtv-L2-5}
\frac{\rm d}{\dt}F_4(t) + \int_{\TT} \bigl( w _x^2 + ( \eta -1)^2  +  | w  |^2 |  w _x |^2 +  \eta  w _t^2 +  \eta _x^2 + w_{xx}^2\bigr)\,\dx \leq 0.
\nn \ea
Here $F_3(t)$ is defined in  \eqref{decay-exp-rx-11} and $C$ solely depends on $(a, \g, \nu, \bar \varsigma_0, \bar E_{00}, \underline{\varsigma}_0,\bar E_{10})$. Then there exists $\a>0$ solely depending on $(a, \g, \nu, \bar \varsigma_0, \bar E_{00}, \underline{\varsigma}_0)$ such that
\ba\label{def-F4-est}
F_4(t)\leq C E_{20} e^{-\a t}.
\nn \ea
And \eqref{Dtv-L2-1} follows.
\end{proof}

\subsection{Proof of Proposition \ref{S2prop1}} \label{sec-decay-all}

With the estimates obtained in the previous sections, we shall prove  Proposition \ref{S2prop1}
by induction method and along  the same line as that of  Propositions \ref{prop-upperbd-density},
\ref{prop-decay-exp-L2}, \ref{prop-decay-exp-H1}, \ref{prop-decay-exp-rx}
and \ref{prop-Dtv-L2}. Since it involves only technicalities, we postpone the proof in Appendix \ref{appa}.

\section{Decay estimates of  $(\eta_{y},w_{y})$}\label{sec:1dNS-y}

In this section, we investigate the decay in time estimates of $(\eta _y,  w _y)$.  We first get, by applying $\d_y$ to \eqref{CNS-limit}, that
\be\label{CNS-1d-dy}
\left\{\begin{aligned}
& \eta _{yt}  + ( \eta  w )_{yx}= 0,\\
&( \eta  w )_{yt} + ( \eta  w ^2)_{yx} - \nu w _{yxx} + p( \eta )_{yx} = 0.
\end{aligned}\right.
\ee
Integrating \eqref{CNS-1d-dy} with respect to $x$ over $\TT$ gives
\ba\label{csv-m-m-dy}
\frac{\rm d}{\dt}\int_{\TT}  \eta _y \,\dx = 0 \andf \frac{\rm d}{\dt}\int_{\TT} ( \eta  w )_y \,\dx = 0.
\nn \ea
It follows from  \eqref{ini-1} that
\ba\label{ini-1-dy}
\int_{\TT} \varsigma_{0y}\,\dx = 0 \andf \int_{\TT} (\varsigma_0  w _0)_y \,\dx = 0.
\nn \ea
This implies
\ba\label{csv-m-m-1-dy}
\int_{\TT}  \eta _y \,\dx = 0 \andf \int_{\TT} ( \eta  w )_y\,\dx = 0, \quad \forall \, t\in\R_{+}.
\ea

\subsection{Decay estimates of $L^2$ norms}

Without loss of generality, we may assume that
\be\label{pressure-2}
  p'(1) = 1.
\ee
Note that this assumption \eqref{pressure-2} can always hold  after a suitable normalization.
In view of  \eqref{ass-ini-1}, one has
\be\label{ini-6}
 E_{00}^{(1)} \eqdefa   \|(\d_y \varsigma_{0},\d_y  w _{0})\|_{L^2_\h}^2 \in (L^{1} \cap L^{\infty})(\R) \andf \bar E_{00}^{(1)} \eqdefa  \sup_{y\in \R} E_{00}^{(1)} <\infty.
\ee

Throughout this subsection, $A, \a$ and $C$ are positive numbers solely depending on $(a, \g, \nu, \bar \varsigma_0, \bar E_{10}, \underline{\varsigma}_0),$ which may differ from line to line.

\begin{lem}\label{energy-basic-y}
{\sl For all $t\in \R^{+}$,
one has
\ba\label{energy-0-0}
\int_{\TT} \bigl( \eta  w _y^2 +  \eta _y^2\bigr) \,\dx + \nu \int_0^t \int_{\TT}| w _{yx}|^2 \,\dx  \leq C E_{00}^{(1)} .
\ea}
\end{lem}

\begin{proof}
Taking $L^2(\TT)$ inner product of $\eqref{CNS-1d-dy}_2$ with $ w _y$ gives
\ba\label{energy-0-1}
\int_{\TT}( \eta  w )_{yt}  w _y \,\dx + \int_{\TT}( \eta  w ^2)_{yx}  w _y \,\dx  + \nu \int_{\TT}| w _{yx}|^2 \,\dx + \int_{\TT} p( \eta )_{yx}  w _y \,\dx = 0.
\ea
Next we handle term by term above.
For the first term in \eqref{energy-0-1}, we have
\begin{align*}
\int_{\TT}( \eta  w )_{yt}  w _y \,\dx & =  \int_{\TT}( \eta  w _y +  \eta _y  w )_{t}  w _{y} \,\dx\\
& = \frac{1}{2}\frac{\rm d}{\dt}  \int_{\TT}  \eta  w _y^2  \,\dx + \frac{1}{2} \int_{\TT}  \eta _t  w _y^2  \,\dx+ \int_{\TT}( \eta  w )_{y} ( w   w _{y})_x \,\dx +  \int_{\TT} w _t   \eta _y  w _y \,\dx.
\end{align*}
For the second  term in \eqref{energy-0-1}, we have
\ba\label{energy-0-5}
\int_{\TT}( \eta  w ^2)_{yx}  w _y \,\dx  = - \int_{\TT}( \eta  w ^2)_{y}  w _{yx} \,\dx  = - \int_{\TT}(2  \eta  w   w _y +  w ^2  \eta _y)  w _{yx} \,\dx .
\nn \ea
And
for the last term in \eqref{energy-0-1}, one has
\ba\label{energy-0-6}
\int_{\TT} p( \eta )_{yx}  w _y \,\dx = - \int_{\TT} p( \eta )_{y}  w _{yx} \,\dx = - \int_{\TT} p'( \eta )  \eta _{y}  w _{yx} \,\dx.
\nn \ea

By substituting the above equalities into \eqref{energy-0-1}, we achieve
\ba\label{energy-0-7}
\frac{1}{2}& \frac{\rm d}{\dt}  \int_{\TT}  \eta  w _y^2  \,\dx + \nu \int_{\TT}| w _{yx}|^2 \,\dx   \\
 & \quad =  - \frac{1}{2} \int_{\TT}  \eta _t  w _y^2  \,\dx - \int_{\TT}( w   w _x  \eta _y  w _y +  \eta  w _x  w _y^2 ) \,\dx\\
 & \qquad -  \int_{\TT} w _t   \eta _y  w _y \,\dx + \int_{\TT}   \eta  w   w _y   w _{yx} \,\dx  + \int_{\TT} p'( \eta )  \eta _{y}  w _{yx} \,\dx.
\ea

On the other hand, by
taking $L^2(\TT)$ inner product of  $\eqref{CNS-1d-dy}_1$ with $ \eta _y,$ we find
\ba\label{energy-0-9}
\frac{1}{2} \frac{\rm d}{\dt} \int_{\TT}  \eta _y^2  \,\dx & =  - \int_{\TT}( \eta  w )_{yx}  \eta _y \,\dx \\
 & = -  \int_{\TT}  w   \eta _{yx}   \eta _y  \,\dx - \int_{\TT} (  w _x  \eta _{y}^2  +  \eta _{x}   w _y  \eta _y) \,\dx - \int_{\TT}   \eta   w _{yx}  \eta _y  \,\dx.
\ea

Summing up \eqref{energy-0-7} and \eqref{energy-0-9} gives rise to
\begin{align*}
&\frac{1}{2} \frac{\rm d}{\dt}  \int_{\TT}  \bigl(\eta  w _y^2 +  \eta _y^2\bigr) \,\dx + \nu \int_{\TT}| w _{yx}|^2 \,\dx   \\
& \quad =  -\int_{\TT}\big( \frac{1}{2}   \eta _t + \eta  w _x   \big)   w _y^2  \,\dx  -
 \int_{\TT}\big( w   w _x + w _t  + \eta _x  \big)   \eta _y  w _y  \,\dx  \\
& \qquad+ \frac{1}{2} \int_{\TT}  w _x  \eta _{y}^2 \,\dx +  \int_{\TT}{( p'( \eta ) -  \eta )}  \eta _y  w _{yx}  \,\dx + \int_{\TT}   \eta  w   w _y   w _{yx} \,\dx,
\end{align*}
from which,  \eqref{pressure-2}, \eqref{decay-exp-all-1} and  \eqref{decay-exp-all-2}, we deduce
\ba\label{energy-0}
 \frac{\rm d}{\dt}  \int_{\TT}  \bigl(\eta  w _y^2 +  \eta _y^2\bigr) \,\dx + \nu \int_{\TT} w _{yx}^2 \,\dx  \leq C e^{-\a t}
  \int_{\TT} \bigl( \eta  w _y^2 +  \eta _y^2\bigr) \,\dx.
\ea
Applying Gronwall's inequality leads to
\eqref{energy-0-0}.
\end{proof}

\begin{prop}\label{prop-decay-L2-y}
{\sl  We have
\ba\label{decay-L2-y-2}
\int_{\TT} \bigl( \eta  w _y^2 +  \eta _y^2\bigr)(t) \,\dx    \leq C  E_{00}^{(1)}(y) e^{-\a t} \andf
 \int_0^\infty \int_{\TT} \bigl(w _{yx}^2 + \eta _y^2\bigr) \,\dx \,\dt  \leq C E_{00}^{(1)}(y).
\ea}
\end{prop}

\begin{proof} Recall \eqref{I-pt2}, we get, by
multiplying the momentum equation of \eqref{CNS-1d-dy}  by $I( \eta _y)$ and integrating the resulting equality over $\TT,$ that
\ba\label{decay-L2-y-3}
\int_{\TT} p( \eta )_{y}  \eta _y \,\dx = \int_{\TT} ( \eta  w )_{yt} I( \eta _y)\,\dx  - \int_{\TT}  ( \eta  w ^2)_{y}  \eta _y\,\dx   + \nu\int_{\TT} w _{yx}  \eta _y \,\dx .
\ea
It is easy to observe that
\begin{align*}
-  \int_{\TT}  &( \eta  w ^2)_{y}  \eta _y \,\dx   = - \int_{\TT}  w ^2  \eta _y^2 \,\dx  - \int_{\TT}  2  \eta  w    w _{y}  \eta _y \,\dx,\\
&\nu\int_{\TT}  w _{yx}  \eta _y \,\dx \leq  C\int_{\TT}  w _{yx}^2 \,\dx + \de \int_{\TT}  \eta _y^2\,\dx,
\end{align*}
and
\ba\label{decay-L2-y-6}
 \int_{\TT} ( \eta  w )_{yt} I( \eta _y)\,\dx  = \frac{\rm d}{\dt}\int_{\TT} ( \eta  w )_{y} I( \eta _y)\,\dx - \int_{\TT} ( \eta  w )_{y} I( \eta _{yt})\,\dx.
\nn \ea
By virtue of \eqref{csv-m-m-1-dy}, one has
\begin{align*}
- \int_{\TT} ( \eta  w )_{y}  I( \eta _{yt})\,\dx & = \int_{\TT} ( \eta  w )_{y} I\left(( \eta  w )_{yx}\right)\,\dx  \\
&  =  \int_{\TT} ( \eta  w )_{y} (( \eta  w )_{y} - ( \eta  w )(t,0,y))\,\dx  =   \int_{\TT} ( \eta  w )_{y}^2 \,\dx.
\end{align*}

By inserting the above estimates into \eqref{decay-L2-y-3}, we obtain
\begin{align*}
\int_{\TT} p'( \eta )  \eta _y^2 \,\dx\leq \frac{\rm d}{\dt}\int_{\TT} ( \eta  w )_{y} I( \eta _y)\,\dx+  \int_{\TT}  \eta ^2  w _y^2  \,\dx
+ C\int_{\TT}  w _{yx}^2 \,\dx + \de \int_{\TT}  \eta _y^2\,\dx.
 \end{align*}
 Choosing $\de = \frac{p'(\underline{\eta})}{2}$ in the above inequality gives rise to
\ba\label{decay-L2-y-8}
\frac{p'(\underline{\eta})}{2}\int_{\TT}  \eta _y^2 \,\dx  - \frac{\rm d}{\dt}\int_{\TT} ( \eta  w )_{y} I( \eta _y)\,\dx \leq  C \int_{\TT}  w _{yx}^2 \,\dx +  \int_{\TT}  \eta ^2  w _y^2  \,\dx.
\ea
Notice from  Lemma \ref{prop-kinetic-tx} and \eqref{csv-m-m-1-dy} that
\ba\label{decay-L2-y-9}
 \int_{\TT}  \eta  w _y^2 - \langle  \eta  w _y \rangle^2 \,\dx =   \int_{\TT}  \eta ( w _y - \langle  \eta  w _y \rangle)^2 \,\dx  \leq \bbeta^2 \int_{\TT}  w _{yx}^2 \,\dx,
\nn \ea from which and
 \eqref{csv-m-m-1-dy},  we infer
\ba\label{decay-L2-y-10}
 \int_{\TT}  \eta  w _y^2
 & \leq \bbeta^2 \int_{\TT}  w _{yx}^2 \,\dx +  \langle  \eta _y  w  \rangle^2 \\
 & \leq \bbeta^2 \int_{\TT}  w _{yx}^2 \,\dx +  \int_{\TT}  \eta _y^2  w ^2 \,\dx.
\ea
Hence thanks to \eqref{decay-exp-H1-2}, we deduce from \eqref{decay-L2-y-8} that
\ba\label{decay-L2-y-8a}
\frac{p'(\underline{\eta})}{2}\int_{\TT}  \eta _y^2 \,\dx  - \frac{\rm d}{\dt}\int_{\TT} ( \eta  w )_{y} I( \eta _y)\,\dx \leq  C \int_{\TT}  w _{yx}^2 \,\dx + Ce^{-\alpha t} \int_{\TT}  \eta_y^2 \,\dx.
\ea

Let $A$ be a sufficiently large constant, we denote
\be\label{F1-1-def}
F_1^{(1)}\eqdefa \int_{\TT} \bigl(A ( \eta  w _y^2 +  \eta _y^2) - ( \eta  w )_y I( \eta _y) \bigr)\,\dx.
\ee
Thanks to Lemma \ref{energy-basic-y}, we get, by multiplying \eqref{energy-0} by $A$ and summing up the resulting inequality with \eqref{decay-L2-y-8a}, that
\ba\label{decay-L2-y-1}
 \frac{\rm d}{\dt} F_1^{(1)}(t) + \int_{\TT}\bigl( w _{yx}^2 + \frac{ p'(\underline{\eta})}{2}  \eta _y^2 \bigr)\,\dx  \leq C E_{00}^{(1)} e^{-\a t}.
\ea
Due to
\ba \label{decay-L2-y-12}
\bigl|\int_{\TT}( \eta  w )_y I( \eta _y) \,\dx\bigr|  \leq \| \eta _y\|_{L^1_\h} \bigl(\| w   \eta _y\|_{L^1_\h} + \| \eta  w _y\|_{L^1_\h}\bigr)
\leq (1 + \| w \|_{L^\infty}) \int_{\TT}  ( \eta  w _y^2 +  \eta _y^2)\,\dx,
\nn \ea
we deduce from \eqref{decay-L2-y-10} that
\ba\label{F1-1-pt}
\int_{\TT}  \bigl(\eta  w _y^2 +  \eta _y^2\bigr) \,\dx  \leq F_1^{(1)}(t) \leq C \int_{\TT}
\bigl(\eta  w _y^2 +  \eta _y^2\bigr) \,\dx \leq C \int_{\TT}  \bigl(w _{yx}^2 +  \eta _y^2\bigr) \,\dx,
\nn \ea
which together with \eqref{decay-L2-y-1} ensures \eqref{decay-L2-y-2}.
\end{proof}

\subsection{Decay  estimates of $H^1$ norms}

It follows from  \eqref{ass-ini-1} that
\be\label{ini-7}
E_{10}^{(1)} \eqdefa  \|\d_y \varsigma_{0} \|_{H^1_\h}^2 + \|\d_y  w _{0}\|_{H^1_\h}^2 \in (L^{1} \cap L^{\infty}) (\R) \andf
 \bar E_{10}^{(1)} \eqdefa  \sup_{y\in \R} E_{10}^{(1)} <\infty.
\ee

Throughout this subsection,  $A, \a, C$ are positive constants solely depending on $(a, \g, \nu, \bar \varsigma_0, \bar E_{10}, \underline{\varsigma}_0, \bar E_{00}^{(1)}),$ which may differ from line to line.

\begin{lem}\label{prop-H1-w-y}
{\sl For all $t\in \R^{+}$, there holds
\ba\label{decay-H1-y-0}
 \frac{\rm d}{\dt} \int_{\TT}  w_{yx}^2\,\dx + \nu \int_{\TT} \eta^{-1} w_{yxx}^2 \,\dx  \leq  C e^{-\a t}
  \int_{\TT}\bigl( w_{yx}^2 + \eta_y^2\bigr) \,\dx +  C \int_{\TT} \eta_{yx}^2 \,\dx.
\ea
}
\end{lem}

\begin{proof}
We rewrite the equation $\eqref{CNS-1d-dy}_2$ as
\beq\label{mm-y-1}
\eta (w_{yt} + w w_{yx} + w_y w_x) + \eta_y (w_t + w w_x) - \nu w _{yxx} + p( \eta )_{yx} = 0.
\eeq
Multiplying the above equation by $\eta^{-1} w_{yxx}$ and integrating the resulting equation over $\TT$ yields
\ba\label{decay-H1-y-1}
\frac{1}{2}& \frac{\rm d}{\dt} \int_{\TT}  w_{yx}^2\,\dx + \nu \int_{\TT} \eta^{-1} w_{yxx}^2 \,\dx \\
&\quad  =   \int_{\TT} \big(\eta (w w_{yx} + w_y w_x) + \eta_y (w_t + w w_x) + p(\eta)_{yx}\big) \eta^{-1} w_{yxx} \,\dx.
\ea
Applying Young's inequality and using \eqref{decay-exp-all-2} gives
\begin{align*}
& \int_{\TT} \big(\eta (w w_{yx} + w_y w_x) + \eta_y (w_t + w w_x) + p(\eta)_{yx}\big) \eta^{-1} w_{yxx} \,\dx \\
& \quad \leq \frac{\nu}{2}  \int_{\TT} \eta^{-1} w_{yxx}^2 \,\dx +  C e^{-\a t} \int_{\TT}\left( w_{yx}^2  + w_y^2 + \eta_y^2 \right)\,\dx +  C \int_{\TT} |p(\eta)_{yx}|^2 \,\dx.
\end{align*}
Notice that
\ba\label{decay-H1-y-3}
\int_{\TT} |p(\eta)_{yx}|^2 \,\dx = \int_{\TT} |p'(\eta)\eta_{yx} + p''(\eta) \eta_y \eta_x|^2 \,\dx.
\nn \ea
As a consequence, thanks to \eqref{decay-L2-y-10}, we thus deduce \eqref{decay-H1-y-0} from \eqref{decay-H1-y-1}.
\end{proof}

\begin{lem}\label{prop-H1-eta-y}
{\sl  For all $t\in \R^{+}$, there holds
\ba\label{decay-H1-eta-y-1}
 \frac{\rm d}{\dt} \int_{\TT}  \eta(w_y - \nu (\eta^{-1})_{yx})^2 \,\dx + \nu \int_{\TT} p'(\eta) \eta^{-2} \eta_{yx}^2 \,\dx
  \leq  C  \int_{\TT}\bigl( w_{yx}^2 + \eta_y^2\bigr) \,\dx.
\ea
}
\end{lem}

\begin{proof}
We rewrite the equation $\eqref{CNS-1d-dy}_2$ as
\beq\label{mm-y-2}
D_t w_y + w_x w_y + \eta^{-1} \eta_y D_t w  - \nu \eta^{-1} w _{yxx} + \eta^{-1} p( \eta )_{yx} = 0.
\eeq
Let $\zeta \eqdefa \eta^{-1}$. Then  we deduce from equation $\eqref{CNS-1d-dy}_1$ that
\ba\label{mass-y-2}
\zeta w_{yxx} =  (\zeta_{tx} + w \zeta_{xx})_y  - \zeta_y w_{xx} = D_t \zeta_{yx} + w_y \zeta_{xx}  - \zeta_y w_{xx} .
\nn \ea
Therefore, we obtain
\ba\label{decay-H1-eta-y-2}
D_t (w_y - \nu \zeta_{yx})  + w_x w_y + \eta^{-1} \eta_y D_t w  - \nu(w_y \zeta_{xx}  - \zeta_y w_{xx}) +  \eta^{-1} p( \eta )_{yx} = 0.
\nn \ea
By taking $L^2$ inner product of the above equation with  $\eta (w_y - \nu \zeta_{yx}),$ we find
\ba\label{decay-H1-eta-y-4}
 \frac{1}{2}&\frac{\rm d}{\dt} \int_{\TT} \eta (w_y - \nu \zeta_{yx})^2\,\dx   + \int_{\TT}  p( \eta )_{yx}(w_y - \nu \zeta_{yx})  \,\dx \\
 &  +   \int_{\TT}  \big( \eta w_x w_y +  \eta_y D_t w  - \nu \eta (w_y \zeta_{xx}  - \zeta_y w_{xx}) \big) (w_y - \nu \zeta_{yx})  \,\dx = 0.
\ea
%Observes that
%$$
%\zeta_y  = - \eta^{-2} \eta_y, \quad \zeta_{yx}  = 2 \eta^{-3} \eta_x \eta_y  - \eta^{-2} \eta_{yx} , \quad \zeta_{xx} = (-\eta^{-2}\eta_{x})_x = 2 %\eta^{-3} \eta_x^2 - \eta^{-2} \eta_{xx},
%$$ and $$
% p( \eta )_{yx} = p'(\eta)\eta_{yx} + p''(\eta) \eta_y \eta_x. $$
Notice that
\begin{align*}
\int_{\TT}  p( \eta )_{yx}(w_y - \nu \zeta_{yx})  \,\dx  &  = \int_{\TT}  (p'(\eta)\eta_{yx} + p''(\eta) \eta_y \eta_x) (w_y - \nu (2 \eta^{-3} \eta_x \eta_y  - \eta^{-2} \eta_{yx}))  \,\dx  \\
& = \nu \int_{\TT} p'(\eta) \eta^{-2} \eta_{yx}^2 \,\dx +\int_{\TT}  p'(\eta)\eta_{yx}  (w_y - 2 \nu  \eta^{-3} \eta_x \eta_y )  \,\dx \\
 & \quad + \int_{\TT}   p''(\eta) \eta_y \eta_x \bigl(w_y - 2 \nu  \eta^{-3} \eta_x \eta_y  +\nu \eta^{-2} \eta_{yx}\bigr))  \,\dx.
\end{align*}
Hence thanks to \eqref{decay-exp-all-3}, we deduce from \eqref{decay-H1-eta-y-4} that
\ba\label{decay-H1-eta-y-6}
&\frac{1}{2}\frac{\rm d}{\dt} \int_{\TT}  \eta (w_y - \nu \zeta_{yx})^2\,\dx   + \nu \int_{\TT} p'(\eta) \eta^{-2} \eta_{yx}^2 \,\dx \leq  \de  \int_{\TT} \eta_{yx}^2 \,\dx  + C  \int_{\TT}  \bigl(w_y^2 + \eta_y^2\bigr)   \,\dx .
\ea
Choosing $\de>0$ small so that
$
\de \leq \frac{\nu}{2}  p'(\underline \eta) {\underline \eta}^{-2},
$
and using \eqref{decay-L2-y-10},  we conclude the proof of \eqref{decay-H1-eta-y-1}.
\end{proof}

\begin{prop}\label{prop-decay-H1-y}
{\sl  For all $t\in \R^{+}$, there holds
\ba\label{decay-H1-eta-w-y-1}
 \int_{\TT} \left(w_{yx}^2 + \eta_{yx}^2\right)(t) \,\dx \leq C E_{10}^{(1)}(y) e^{-\a t}, \quad  \int_0^\infty \int_{\TT} w _{yxx}^2  \,\dx \,\dt  \leq C E_{10}^{(1)}(y),
\ea
and
\be\label{decay-H1-eta-w-y-t-2}
 \int_0^\infty\int_{\TT} \eta w_{yt}^2\,\dx\,\dt  \leq C E_{10}^{(1)}(y).
\ee }
\end{prop}

\begin{proof} Thanks to Lemmas  \ref{prop-H1-w-y} and \ref{prop-H1-eta-y}, we deduce \eqref{decay-H1-eta-w-y-1}
by similar arguments using in the proof of Proposition \ref{prop-decay-exp-rx}.

While we get, by taking $L^2(\TT)$ inner product of \eqref{mm-y-1} with $w_{yt},$ that
\ba\label{decay-H1-eta-w-y-t-3}
\int_{\TT} \eta  w_{yt}^2 \,\dx  + \nu  \frac{\rm d}{\dt} \int_{\TT}  w_{yx}^2 \,\dx  = - \int_{\TT}\big(w w_{yx} + w_y w_x + \eta_y (w_t + w w_x) + p( \eta )_{yx} \big)  w_{yt}\,\dx,
\nn \ea
which implies
\ba\label{decay-H1-eta-w-y-t-1}
 \nu\frac{\rm d}{\dt} \int_{\TT}  w_{yx}^2\,\dx + \frac{1}{2} \int_{\TT} \eta w_{yt}^2\,\dx \leq C \int_{\TT}\bigl( w_{yx}^2 + w_y^2 + \eta_y^2 + \eta_{yx}^2\bigr)\,\dx.
\ea
By integrating \eqref{decay-H1-eta-w-y-t-1} over $[0,t]$ and using \eqref{decay-H1-eta-w-y-1},
we obtain \eqref{decay-H1-eta-w-y-t-2}.
\end{proof}

\subsection{Decay estimate of $\|w_y\|_{H^{2}}$}

In view of \eqref{ass-ini-1}, we have
\be\label{ini-8}
 E_{20}^{(1)} \eqdefa  \|\d_y \varsigma_{0}\|_{H^2_\h}^2+ \|\d_y  w _{0}\|_{H^2_\h}^2 \in (L^{1}\cap L^{\infty}) (\R)
 \andf \bar E_{20}^{(1)} \eqdefa  \sup_{y\in \R} E_{20}^{(1)} <\infty.
\ee

Throughout this subsection,  $A, \a, C$ are positive numbers solely depending on $(a, \g, \nu, \bar \varsigma_0, \bar E_{20}, \underline{\varsigma}_0, \bar E_{10}^{(1)}),$ which  may change  from line to line.

\begin{prop}\label{prop-Dtv-y-L2}
{\sl Let  $D_t\eqdefa\d_t +  w  \d_x$ be  the material derivative. Then for all $t \in \R^{+}$, there holds
\be\label{Dtv-L2-y-1}
\int_{\TT}  \eta |D_t w_y |^2 \, \dx\leq C  E_{20}^{(1)} (y) e^{-\a t} \andf \int_0^\infty \int_{\TT} |(D_t w_y )_x|^2\,\dx\,\dt \leq C  E_{20}^{(1)} (y) ,
\ee and
\be\label{vxx-L2-y-1}
\int_{\TT}  \bigl(|\eta_{yt}|^2  + |w_{yt}|^2 + |w_{yxx}|^2\bigr) \, \dx\leq C  E_{20}^{(1)} (y)e^{-\a t}.
\ee
}
\end{prop}

\begin{proof}
Applying $D_t$ to $\eta\times$\eqref{mm-y-2} gives
\ba\label{Dtv-L2-y-3}
\eta D_t^2 w_y + D_t \eta D_t w_y  + D_t (\eta w_x w_y) + D_t (\eta_y D_t w ) - \nu  D_t w _{yxx} +  D_t p( \eta )_{yx} = 0.
\ea
It is easy to observe that
\begin{align*}%\label{Dtv-L2-y-4}
D_t  \eta D_t  w_y  & = -  \eta  w _x D_t  w_y  =  w_x (\eta w_x w_y +  \eta_y D_t w  - \nu  w _{yxx} +  p( \eta )_{yx} ),
\\
 D_t (\eta w_x w_y) &= D_t (\eta w_x) w_y + \eta w_x D_t w_y,
\\
D_t (\eta_y D_t w )  & = -(\eta_y w_x + \eta_x w_y + \eta w_{yx}) D_t w + \eta_y D_t^2 w,
\\
- \nu D_t  w _{yxx} &
  = -\nu (D_t w_y )_{xx} + \nu( 2 w_x w_{yxx} + w_{xx} w_{yx}),
\end{align*}
and
\ba
D_t p( \eta )_{yx} = p( \eta )_{tyx} +  w  p( \eta )_{yxx}.
\nn \ea

Then by taking $L^2(\TT)$ inner product of \eqref{Dtv-L2-y-3} with $D_t w_y$  and using the fact
$$
w_x p( \eta )_{yx} + w  p( \eta )_{yxx} = (w p(\eta)_{yx})_x,
$$ we find
\begin{align*}
\frac{1}{2}& \frac{\rm d}{\dt} \int_{\TT}  \eta |D_t w_y |^2 \, \dx + \nu \int_{\TT} |(D_t w_y )_x|^2\,\dx \\
&  \leq  \frac{\nu}{2} \int_{\TT} |(D_t w_y )_x|^2\,\dx + C \int_{\TT} \bigl( w_y^2 + w_{yx}^2 +  w _{yxx}^2 +  \eta_y^2 +\eta_{yx}^2\bigr)\,\dx.
\end{align*}
This implies
  \ba\label{Dtv-L2-y}
\frac{\rm d}{\dt} \int_{\TT}  \eta |D_t w_y |^2 \, \dx
+ \nu \int_{\TT} |(D_t w_y )_x|^2\,\dx \leq C \int_{\TT}  \bigl(w_y^2 + w_{yx}^2 +  w _{yxx}^2 +  \eta_y^2 +\eta_{yx}^2\bigr)\,\dx,
\ea
from which,
  Propositions \ref{prop-decay-L2-y} and \ref{prop-decay-H1-y},  Lemmas \ref{prop-H1-w-y} and \ref{prop-H1-eta-y},
   we deduce \eqref{Dtv-L2-y-1} by the similar argument used in the proof of Proposition \ref{prop-decay-exp-rx}.

Thanks to \eqref{Dtv-L2-y-1}, we conclude \eqref{vxx-L2-y-1} by using the mass and momentum equations \eqref{CNS-1d-dy}.
\end{proof}

\subsection{Proof of Proposition \ref{S2prop2}} \label{sec-decay-y-all}

  \begin{proof}[Proof of Proposition \ref{S2prop2}] By summarizing Propositions \ref{prop-decay-L2-y}, \ref{prop-decay-H1-y} and \ref{prop-Dtv-y-L2},
we deduce that there exist two positive constants $C$ and $\a$ solely depends on $(a, \g, \nu, \bar \varsigma_0, \bar E_{20}, \underline{\varsigma}_0, \bar E_{10}^{(1)})$ such that
\ba\label{decay-exp-all-y-1}
\| \eta_y \|_{H^1_\h} + \| w_{y} \|_{H^2_\h} + \| \eta _{yt}\|_{L^2_\h} +  \|w_{yt}\|_{L^2_\h} \leq C \left( E_{20}^{(1)}(y)\right)^{\frac 12}  e^{-\a t}, \quad \forall \, t\in \R^{+}.
\ea

In what follows, we shall
follow the same strategy as that of the  proof  of Proposition \ref{S2prop1}.
We first get, by  applying $\d_x$  to \eqref{CNS-1d-dy}, that
\be\label{CNS-1d-dy-dx}
\left\{\begin{aligned}
& \eta _{yxt}  + ( \eta  w )_{yxx}= 0,\\
&( \eta  w )_{yxt} + ( \eta  w ^2)_{yxx} - \nu w _{yxxx} + p( \eta )_{yxx} = 0.
\end{aligned}\right.
\ee
We can also rewrite $\eqref{CNS-1d-dy-dx}_{2}$ as
\beq\label{CNS-1d-dy-dx-1}
\begin{split}
D_{t} w_{yx} &+  (w_{y}w_{xx} + 2 w_{x}w_{yx}) + \eta^{-1} \eta_{y} (D_{t} w_{x} + w_{x}^{2}) \\
&+ \eta^{-1} \eta_{x} (D_{t} w_{y} + w_{x}w_{y}) + \eta^{-1} \eta_{yx} D_{t}w
- \nu \eta^{-1} w_{yxxx} + \eta^{-1} p(\eta)_{yxx} = 0.
\end{split}
\eeq

We split the proof of the remaining estimates in \eqref{decay-exp-all-y-2} into the following steps:

\medskip

\noindent{\bf Step 1.} {Decay estimates for $\eta_{yxx}.$}

\medskip

Recalling that $\zeta \eqdefa \eta^{-1},$
we deduce from $\eqref{CNS-1d-dy-dx}_{1}$ that
\be\label{CNS-1d-dy-dx-2}
 \zeta w_{yxxx} = D_{t}\zeta_{yxx} - \zeta_{yx} w_{xx} - \zeta_{x} w_{yxx} -  \zeta_{y} w_{xxx}  + \zeta_{xxx} w_{y} + \zeta_{yxx} w_{x} + \zeta_{xx} w_{yx}.
\ee
Plugging \eqref{CNS-1d-dy-dx-2} into \eqref{CNS-1d-dy-dx-1} gives rise to
\begin{align*}%\label{CNS-1d-dy-dx-3}
 &D_{t} (w_{yx} - \nu \zeta_{yxx}) +  w_{x}(w_{yx} - \nu\zeta_{yxx}) +  (w_{y}w_{xx} +  w_{x}w_{yx}) \\
& \quad +  \eta^{-1} \eta_{y} (D_{t} w_{x} + w_{x}^{2}) + \eta^{-1} \eta_{x} (D_{t} w_{y} + w_{x}w_{y}) + \eta^{-1} \eta_{yx} D_{t}w  \\
& \quad - \nu \bigl(- \zeta_{yx} w_{xx} - \zeta_{x} w_{yxx} -  \zeta_{y} w_{xxx}  + \zeta_{xxx} w_{y} + \zeta_{xx} w_{yx}\bigr)+ \eta^{-1} p(\eta)_{yxx} = 0.
\end{align*}
By taking $L^2(\TT)$ inner product of the above equation with  $\eta (w_{yx} - \nu \zeta_{yxx}) $ and using the decay estimates we have derived in the previous sections, we obtain
\ba\label{CNS-1d-dy-dx-4}
&\frac{1}{2}\frac{\rm d}{\dt} \int_{\TT}  \eta  (w_{yx} - \nu \zeta_{yxx})^{2}  \,\dx + \int_{\TT}  (w_{yx} - \nu \zeta_{yxx}) p(\eta)_{yxx} \,\dx\\
& \quad \leq   C \left(E_{20}(y)\right)^{\frac 12} e^{-\a t} \int_{\TT}  \eta (w_{yx} - \nu \zeta_{yxx})^{2}  \,\dx + \de \int_{\TT} \eta  (w_{yx} - \nu \zeta_{yxx})^{2}  \,\dx  \\
& \qquad + C \de^{-1}  \left( E_{30}(y) + E_{20}^{(1)}(y) \right) e^{-\a t}.
\ea
To handle the term related to the pressure, we write
\ba\label{CNS-1d-dy-dx-5}
p(\eta)_{yxx} & = p'(\eta)\eta_{yxx} +p''(\eta) \eta_{y}\eta_{xx} +  p''(\eta) 2 \eta_{x}\eta_{yx} + p'''(\eta)\eta_{y} \eta_{x}^{2}\\
&= p'(\eta) (-\zeta^{-2} \zeta_{yxx} +  2 \zeta^{-3} \zeta_{y} \zeta_{xx}  +  4 \zeta^{-3} \zeta_{x}\zeta_{yx}  - 6 \zeta^{-4} \zeta_{y}\zeta_{x}^{2}) \\
& \quad + p''(\eta) \eta_{y}\eta_{xx} +  p''(\eta) 2 \eta_{x}\eta_{yx} + p'''(\eta)\eta_{y} \eta_{x}^{2}\\
& = p'(\eta) \eta^{2} \nu^{-1} (w_{yx} - \nu\zeta_{yxx}) - p'(\eta) \eta^{2} \nu^{-1} w_{yx}  + p'(\eta) \eta^{3}  (2 \zeta_{y} \zeta_{xx}  +  4 \zeta_{x}\zeta_{yx}  - 6 \eta^{2} \zeta_{y}\zeta_{x}^{2}) \\
& \quad + p''(\eta) \eta_{y}\eta_{xx} +  p''(\eta) 2 \eta_{x}\eta_{yx} + p'''(\eta)\eta_{y} \eta_{x}^{2}.
 \ea
Notice that
$$
p'(\eta)\eta \geq p'(\underline \eta) \underline \eta >0,
$$
 by choosing $\de$ sufficiently  small, we deduce from \eqref{CNS-1d-dy-dx-4} and \eqref{CNS-1d-dy-dx-5} that
\begin{align*}%\label{CNS-1d-dy-dx-6}
&\frac{\rm d}{\dt} \int_{\TT} \eta  (w_{yx} - \nu \zeta_{yxx})^{2}\,\dx
 + \nu^{-1} p'(\underline \eta) \underline \eta  \int_{\TT} \eta  (w_{yx} - \nu \zeta_{yxx})^{2} \,\dx\\
& \leq C e^{-\a t}   \int_{\TT}  \eta (w_{yx} - \nu \zeta_{yxx})^{2}  \,\dx + C \left( E_{30}(y) + E_{20}^{(1)}(y) \right) e^{-\a t}.
 \end{align*}
Applying Gronwall's inequality gives
\ba\label{CNS-1d-dy-dx-7}
\int_{\TT} \eta  (w_{yx} - \nu \zeta_{yxx})^{2} \,\dx \leq  C   \left( E_{30}(y) + E_{20}^{(1)}(y) \right) e^{-\a t},
 \nn \ea
from which and \eqref{decay-exp-all-y-1},  we infer
 \ba\label{CNS-1d-dy-dx-8}
\int_{\TT}  \zeta_{yxx}^{2}\,\dx  \leq C  \left( E_{30}(y) + E_{20}^{(1)}(y) \right) e^{-\a t}.
 \ea
 This leads to
  \ba\label{CNS-1d-dy-dx-9}
\int_{\TT}  \eta_{yxx}^{2}\,\dx  \leq C  \left( E_{30}(y) + E_{20}^{(1)}(y) \right) e^{-\a t}.
 \ea

\medskip

\noindent{\bf Step 2.} {Decay estimates for $D_{t} w_{yx}.$}

\medskip

We first rewrite $\eqref{CNS-1d-dy-dx}_{2}$ as
\beq\label{Dt-wys}\begin{split}
&\eta D_{t} w_{yx} - \nu  w_{yxxx} + \frak{g} + p(\eta)_{yxx} = 0,\\
\frak{g} \eqdefa \eta (w_{y}w_{xx} &+ 2 w_{x}w_{yx}) + \eta_{y} (D_{t} w_{x} + w_{x}^{2}) + \eta_{x} (D_{t} w_{y} + w_{x}w_{y}) +  \eta_{yx} D_{t}w.
\end{split}
\eeq
By taking $L^2(\TT)$ inner product of \eqref{Dt-wys} with $\eta^{-1} w_{yxxx},$ we obtain
\ba\label{Dt-wys-2}
\int_{\TT}  D_{t} w_{yx} w_{yxxx}\,\dx-\nu \int_{\TT} \eta^{-1}  |w_{yxxx}|^{2}\,\dx+\int_{\TT} \bigl(\frak{g}
 + p(\eta)_{yxx} \bigr) \eta^{-1} w_{yxxx}\,\dx=0.
 \ea
It is easy to observe that
\begin{align*}
\int_{\TT} D_{t} w_{yx}  w_{yxxx}\,\dx & = \int_{\TT}  \d_{t} w_{yx} w_{yxxx}\,\dx + \int_{\TT}  w w_{yxx} w_{yxxx}\,\dx\\
 & = -\frac{1}{2} \frac{\rm d}{\dt} \int_{\TT}  |w_{yxx}|^{2} - \frac{1}{2}\int_{\TT} w_{x} (w_{yxx})^{2}.
\end{align*}
While it follows from \eqref{CNS-1d-Dtwn-26}, \eqref{decay-exp-all-y-1} and \eqref{CNS-1d-dy-dx-9} that
\ba
\int_{\TT} \bigl(\frak{g} + p(\eta)_{yxx} \bigr) \eta^{-1} w_{yxxx}\,\dx \leq  C \de^{-1}  E(y) e^{-\a t} + \de \int_{\TT}   |w_{yxxx}|^{2}\,\dx.
\nn \ea

By inserting the above estimates into \eqref{Dt-wys-2} and choosing $\de$ suitably small, we achieve
\ba\label{Dt-wys-1}
\frac{\rm d}{\dt} \int_{\TT}  |w_{yxx}|^{2}\,\dx + \nu \bar\eta^{-1} \int_{\TT}   |w_{yxxx}|^{2}\,\dx \leq C  E(y) e^{-\a t}.
\ea

While by taking $L^2(\TT)$ inner product of \eqref{Dt-wys} with $ D_{t} w_{yx},$ we find
\ba\label{Dt-wys-7}
\int_{\TT} \eta  |D_{t}w_{yx}|^{2}\,\dx-\nu\int_{\TT}D_t w_{yxxx}D_{t} w_{yx}\,\dx +\int_{\TT}\bigl( D_t\frak{g} +D_t p(\eta)_{yxx}\bigr)D_{t} w_{yx}\,\dx = 0.
\ea
Observing that
\begin{align*}%\label{Dt-wys-7}
- \nu\int_{\TT}   w_{yxxx}  D_{t} w_{yx}\,\dx & = - \nu \int_{\TT}  w_{yxxx}  \d_{t} w_{yx} \,\dx - \nu  \int_{\TT}  w_{yxxx}  w w_{yxx}\,\dx \\
& = \frac\nu2\frac{\rm d}{\dt} \int_{\TT}  |w_{yxx}|^{2}\,\dx + \frac{\nu}{2} \int_{\TT}     (w_{yxx})^{2}  w_{x}\,\dx\\
& \geq \frac\nu2\frac{\rm d}{\dt} \int_{\TT}  |w_{yxx}|^{2} \,\dx-  C  E(y) e^{-\a t}.
\end{align*}
By using the estimates \eqref{CNS-1d-Dtwn-26}, \eqref{decay-exp-all-y-1} and \eqref{CNS-1d-dy-dx-9},  we find
\ba
\int_{\TT} \bigl(\frak{g} + p(\eta)_{yxx} \bigr) D_{t} w_{yx}\,\dx \leq  C \de^{-1}  E(y) e^{-\a t} + \de \int_{\TT}   | D_{t} w_{yx}|^{2}\,\dx.
\nn \ea

By substituting  the above estimates into \eqref{Dt-wys-7}
 and taking $\de$ to be suitably small, we achieve
\ba\label{Dt-wys-5}
 \nu\frac{\rm d}{\dt} \int_{\TT} |w_{yxx}|^{2}\,\dx + \int_{\TT} \eta  |D_{t}w_{yx}|^{2}\,\dx \leq C  E(y) e^{-\a t}.
\ea

 On the other hand,  we get, by applying  $D_{t}$ to \eqref{Dt-wys}  and then taking $L^2(\TT)$ inner product of the resulting
equation  with $D_{t} w_{yx},$ that
\ba\label{Dt-wys-6}
\int_{\TT}\bigl(D_t(\eta D_{t} w_{yx})- \nu D_t w_{yxxx} + D_t\frak{g} +D_t p(\eta)_{yxx}\bigr)D_{t} w_{yx}\,\dx = 0.
\ea
Next we handle term by term above. We first observe that
\begin{align*}%\label{Dt-wys-11}
 \int_{\TT}D_{t}( \eta D_{t} w_{yx} ) D_{t} w_{yx}\,\dx  & = \int_{\TT} (D_{t} \eta) D_{t} w_{yx} D_{t} w_{yx}\,\dx
  + \int_{\TT}\eta (D_{t} D_{t} w_{yx})D_{t} w_{yx}\,\dx \\
& = \frac{1}{2} \frac{\rm d}{\dt} \int_{\TT}\eta |D_{t} w_{yx}|^{2}\,\dx  + \frac12\int_{\TT} (D_{t} \eta) |D_{t} w_{yx}|^{2}\,\dx \\
& \geq  \frac{1}{2}\frac{\rm d}{\dt} \int_{\TT}\eta |D_{t} w_{yx}|^{2}\,\dx  - C \int_{\TT} |D_{t} w_{yx}|^{2}.
\end{align*}

Notice that
\ba\label{Dt-wys-12}
D_{t} w_{yxxx} = (D_{t} w_{yx})_{xx} - w_{xx} w_{yxx} - 2 w_{x} w_{yxxx},
\nn \ea
we have
\begin{align*}%\label{Dt-wys-13}
- \nu \int_{\TT}  D_{t} w_{yxxx} D_{t} w_{yx}\,\dx  & = -\nu  \int_{\TT}(D_{t} w_{yx})_{xx} D_{t} w_{yx}\,\dx
+  \nu \int_{\TT} \bigl(w_{xx} w_{yxx} + 2 w_{x} w_{yxxx}\bigr)D_{t} w_{yx}\,\dx \\
& \geq \nu \int_{\TT} |(D_{t} w_{yx})_{x}|^{2}\,\dx  - C \int_{\TT} \bigl(|D_{t} w_{yx}|^{2} + |w_{yxxx}|^{2}\bigr)\,\dx.
\end{align*}

We get, by applying the estimates \eqref{CNS-1d-Dtwn-26}, \eqref{decay-exp-all-y-1} and \eqref{CNS-1d-dy-dx-9}, that
\ba\label{Dt-wys-14}
\int_{\TT} (D_{t} \bigl(\frak{g}-\eta_{x} D_{t} w_{y} \bigr) D_{t} w_{yx}\,\dx \leq  C  E(y) e^{-\a t} + C  \int_{\TT}   | D_{t} w_{yx}|^{2} \,\dx.
\nn \ea
In view of \eqref{mm-y-2}, we write
\ba\label{Dt-wys-15}
D_{t} (\eta_{x} D_{t} w_{y}) = D_{t} \eta_{x} D_{t} w_{y} - \eta_{x} D_{t} \big( w_x w_y + \eta^{-1} \eta_y D_t w  - \nu \eta^{-1} w _{yxx} + \eta^{-1} p( \eta )_{yx}   \big).
\nn \ea
It  suffices to deal with the highest order derivative term $\eta_{x}\eta^{-1} D_{t} w_{yxx}$ above. Other terms can be estimated similarly, even easier.
 Observing that
\ba\label{Dt-wys-16}
D_{t} w_{yxx} = (D_{t} w_{yx})_{x} - w_{x} w_{yxx},
\nn \ea
we get
\ba\label{Dt-wys-17}
\int_{\TT} \eta_{x}\eta^{-1} (D_{t} w_{yxx}) D_{t} w_{yx}\,\dx \leq C  E(y) e^{-\a t} + C  \int_{\TT}   | D_{t} w_{yx}|^{2}\,\dx
 + \frac{\nu}{4} \int_{\TT} |(D_{t} w_{yx} )_{x} |^{2}\,\dx.
\nn \ea

For the term related to the pressure in \eqref{Dt-wys-6}, we have
\begin{align*}\label{Dt-wys-18}
&\int_{\TT} (D_{t} p(\eta)_{yxx}) D_{t} w_{yx}\,\dx = \int_{\TT} \bigl(\d_{t} p(\eta)_{yxx} + w p(\eta)_{yxxx}\bigr) D_{t} w_{yx}\,\dx \\
& =  - \int_{\TT} \bigl((\d_{t} p(\eta)_{yx} + w p(\eta)_{yxx}) (D_{t} w_{yx} )_{x}  -  w_{x} p(\eta)_{yxx} D_{t} w_{yx}\bigr)\,\dx\\
& \leq C \int_{\TT} |\d_{t} p(\eta)_{yx} + w p(\eta)_{yxx}|^{2}\,\dx + \frac{\nu}{4} \int_{\TT} |(D_{t} w_{yx} )_{x} |^{2}\,\dx
+  \int_{\TT} \bigl(| w_{x} p(\eta)_{yxx}|^{2}  +    | D_{t} w_{yx}|^{2}\bigr)\,\dx\\
& \leq  C  E(y) e^{-\a t} + C  \int_{\TT}   | D_{t} w_{yx}|^{2}\,\dx + \frac{\nu}{4} \int_{\TT} |(D_{t} w_{yx} )_{x} |^{2}\,\dx.
\end{align*}

By substituting the above estimate into \eqref{Dt-wys-6}, we arrive at
\ba\label{Dt-wys-9}
\frac{\rm d}{\dt} \int_{\TT} \eta |D_{t} w_{yx}|^{2}\,\dx + \nu \int_{\TT} |(D_{t} w_{yx})_{x}|^{2} \,\dx\leq C  E(y) e^{-\a t} + C \int_{\TT}   \bigl(|w_{yxxx}|^{2} + |D_{t}w_{yx}|^{2}\bigr)\,\dx.
\ea

By virtue of \eqref{Dt-wys-1}, \eqref{Dt-wys-5} and \eqref{Dt-wys-9}, we can find a large enough positive constant $A>0$ so that
\ba\label{Dt-wys-19}
\frac{\rm d}{\dt} &\int_{\TT}\bigl( A \bigl(|w_{yxx}|^2+|w_{yxx}|^{2}\bigr) + \eta |D_{t} w_{yx}|^{2}\bigr)\,\dx\\
 &+ \int_{\TT} \bigl(|w_{yxxx}|^{2} + |D_{t}w_{yx}|^{2} +  \nu |(D_{t} w_{yx})_{x}|^{2}\bigr)\,\dx \leq C  E(y) e^{-\a t}.
\ea
Together with \eqref{decay-exp-all-y-1}, we deduce from \eqref{Dt-wys-19} that
\ba\label{Dt-wys-20}
 \|D_{t} w_{yx}\|_{L^{2}_{h}}^{2} \leq C  E(y) e^{-\a t} .
\ea
Together with \eqref{decay-exp-all-y-1}, this implies
\ba\label{Dt-wys-21}
 \|w_{yxxx}\|_{L^{2}_{h}}^{2} \leq C  E(y) e^{-\a t} .
\ea

By summarizing the estimates \eqref{decay-exp-all-y-1}, \eqref{CNS-1d-dy-dx-9}, \eqref{Dt-wys-20} and \eqref{Dt-wys-21}, we conclude the proof of \eqref{decay-exp-all-y-2}.
\end{proof}

\section{Decay estimates of  $(\eta_{yy}, w_{yy})$}\label{sec:1dNS-yy}

In this section, we investigate the decay in time estimates of $(\eta _y,  w _y)$. We first get,
by  applying $\d_y^{2}$ to \eqref{CNS-limit}, that
\be\label{CNS-1d-dyy}
\left\{\begin{aligned}
& \eta _{yyt}  + ( \eta  w )_{yyx}= 0,\\
&( \eta  w )_{yyt} + ( \eta  w ^2)_{yyx} - \nu w _{yyxx} + p( \eta )_{yyx} = 0.
\end{aligned}\right.
\ee
Integrating \eqref{CNS-1d-dyy} with respect to $x$ over $\TT$ gives
\ba\label{csv-m-m-dyy}
\frac{\rm d}{\dt}\int_{\TT}  \eta_{yy} \,\dx = 0 \andf \frac{\rm d}{\dt}\int_{\TT} ( \eta  w )_{yy} \,\dx = 0.
\nn \ea
This implies
\ba\label{csv-m-m-1-dyy}
\int_{\TT}  \eta_{yy} \,\dx = 0 \andf \int_{\TT} ( \eta  w )_{yy}\,\dx = 0, \quad \forall \, t\in\R_{+}.
\ea

\subsection{Decay estimates of $L^{2}$ norms}
 Introduce
\be\label{ini-6-dyy}
 E_{00}^{(2)} (y)\eqdefa   \|(\d_y^{2} \varsigma_{0},\d_y^{2}  w _{0})(y)\|_{L^2_\h}^2 \in (L^{1} \cap L^{\infty})(\R) \andf \bar E_{00}^{(2)} \eqdefa  \sup_{y\in \R} E_{00}^{(2)} <\infty.
\ee
Throughout this subsection, $A, \a$ and $C$ are positive numbers solely depending on $(a, \g, \nu, \bar \varsigma_0, \bar E_{20}, \bar E_{10}^{(1)},  \underline{\varsigma}_0),$ which may differ from line to line. We first give the basic energy estimates:
\begin{lem}\label{energy-basic-yy}
{\sl For all $t\in \R^{+}$,
one has
\ba\label{energy-0-0-yy}
\int_{\TT} \bigl( \eta  w_{yy}^2 +  \eta_{yy}^2\bigr) \,\dx + \nu \int_0^t \int_{\TT}| w _{yyx}|^2 \,\dx  \leq C \left( E_{10}^{(1)} (y) + E_{20}(y)\right).
\ea}
\end{lem}

\begin{proof}
Firstly, we rewrite $\eqref{CNS-1d-dyy}_{2}$ as
\ba\label{CNS-1d-dyy-1}
\eta D_t w_{yy}  + \eta (2 w_{y} w_{yx} + w_{yy}w_{x})+ 2 \eta_{y} (D_{t} w_{y} + w_{y} w_{x}) + \eta_{yy} D_{t} w - \nu w _{yyxx} + p( \eta )_{yyx} = 0.
\ea

By using the decay estimates we have derived in the previous sections and testing \eqref{CNS-1d-dyy-1} by $ w_{yy}$, we  obtain
\ba\label{energy-0-1-dyy}
\frac{1}{2} & \frac{\rm d}{\dt}  \int_{\TT}   \eta  w_{yy}^2  \,\dx + \nu \int_{\TT}| w _{yyx}|^2 \,\dx   \\
 & \quad \leq   \frac{\nu}{8} \int_{\TT}| w _{yyx}|^2 \,\dx  + C e^{-\a t}  \int_{\TT} \bigl( \eta  w_{yy}^2 +  \eta_{yy}^2\bigr) \,\dx + \int_{\TT} p'( \eta )  \eta_{yy}  w _{yyx} + C  E_{10}^{(1)} (y) e^{-\a t}.
 \ea

On the other hand, by taking $L^2(\TT)$ inner product of  $\eqref{CNS-1d-dyy}_1$ with $ \eta_{yy},$ we find
\ba\label{energy-0-2-dyy}
 \frac{1}{2}&\frac{\rm d}{\dt} \int_{\TT}   \eta_{yy}^2  \,\dx  =  - \int_{\TT}( \eta  w )_{yyx}  \eta_{yy} \,\dx \\
 & = -  \int_{\TT} \bigl(   \eta   w_{yyx} + \eta_{yy}  w_{x} +   \eta_{yyx}  w + 2 \eta_{yx}  w_{y} + 2 \eta_{y}  w_{yx} + \eta w_{yyx} + \eta_{x} w_{yy}\bigr) \eta_{yy}\,\dx \\
 %& =  -  \int_{\TT}  \eta   w_{yyx} \eta_{yy}\,\dx - \int_{\TT}  \frac{1}{2} (\eta_{yy}^{2})_{x} w \,\dx - \int_{\TT} (\eta_{yy}  w_{x} + 2 \eta_{yx}  w_{y} + 2 \eta_{y}  w_{yx} + \eta w_{yyx} + \eta_{x} w_{yy}) \eta_{yy}\,\dx\\
 & =   -  \int_{\TT}  \eta   w_{yyx} \eta_{yy}\,\dx - \int_{\TT} \bigl(\frac{1}{2}\eta_{yy}  w_{x} + 2 \eta_{yx}  w_{y} + 2 \eta_{y}  w_{yx} + \eta w_{yyx} + \eta_{x} w_{yy}\bigr) \eta_{yy}\,\dx.
 \ea
By summing up \eqref{energy-0-1-dyy} with \eqref{energy-0-2-dyy}, we achieve
\begin{align*}
&\frac{\rm d}{\dt}  \int_{\TT} \frac{1}{2} \bigl(\eta  w_{yy}^2 +  \eta_{yy}^2\bigr) \,\dx + \nu \int_{\TT}| w _{yyx}|^2 \,\dx   \\
& \quad \leq \frac{\nu}{4} \int_{\TT}| w _{yyx}|^2 \,\dx +   C e^{-\a t}  \int_{\TT} \bigl( \eta  w_{yy}^2 +  \eta_{yy}^2\bigr) \,\dx + \int_{\TT} (p'( \eta ) -\eta) \eta_{yy} w _{yyx}\,\dx + C  E_{10}^{(1)} (y) e^{-\a t},
\end{align*}
from which and  the decay of $p'(\eta) - \eta$ obtained in \eqref{S2eq1},  we infer
\ba\label{energy-0-3-dyy}
&\frac{\rm d}{\dt}  \int_{\TT}  \bigl(\eta  w_{yy}^2 +  \eta_{yy}^2\bigr) \,\dx + \nu \int_{\TT}| w _{yyx}|^2 \,\dx   \\
& \quad \leq  C e^{-\a t}  \int_{\TT} \bigl( \eta  w_{yy}^2 +  \eta_{yy}^2\bigr) \,\dx + C  E_{10}^{(1)} (y) e^{-\a t}.
 \ea
Applying Gronwall's inequality leads to \eqref{energy-0-0-yy}.
\end{proof}

\begin{prop}\label{prop-decay-L2-yy}
{\sl  For all $t\in \R_{+}$,
\ba\label{decay-L2-yy}
\int_{\TT} \bigl( \eta  w_{yy}^2 +  \eta_{yy}^2\bigr) (t)\,\dx  \leq C \left( E_{10}^{(1)} (y) + E_{20}(y)\right) e^{-\a t}.
\ea}
\end{prop}

\begin{proof} We first get, by multiplying the momentum equation of \eqref{CNS-1d-dyy}  by $I( \eta_{yy})$ and integrating the resulting equality over $\TT,$  that
\ba\label{decay-L2-yy-1}
\int_{\TT} p( \eta )_{yy}  \eta_{yy} \,\dx = \int_{\TT} ( \eta  w )_{yyt} I( \eta_{yy})\,\dx  - \int_{\TT}  ( \eta  w^2)_{yy}  \eta_{yy}\,\dx   +   \nu  \int_{\TT}w_{yyx}  \eta_{yy} \,\dx .
\ea
We now handle term by term above. For the term on the left-hand side of \eqref{decay-L2-yy-1}, we have
\begin{align*}%\label{decay-L2-yy-2}
\int_{\TT} p( \eta )_{yy}  \eta_{yy} \,\dx & = \int_{\TT} p'( \eta ) (\eta_{yy})^{2} \,\dx +  \int_{\TT} p''( \eta ) (\eta_{y})^{2} \eta_{yy} \,\dx\\
& \geq \int_{\TT} p'( \eta ) (\eta_{yy})^{2} \,\dx - C \de E_{10}^{(1)} (y) e^{-\a t} \int_{\TT} (\eta_{yy})^{2} \,\dx - C \de^{-1} E_{10}^{(1)} (y) e^{-\a t}.
\end{align*}
We then compute the first term on the right-hand side of \eqref{decay-L2-yy-1} as follows
\begin{align*}%\label{decay-L2-yy-3}
\int_{\TT} ( \eta  w )_{yyt} I( \eta_{yy})\,\dx  & =  \frac{\rm d}{\dt}\int_{\TT} ( \eta  w )_{yy} I( \eta_{yy})\,\dx - \int_{\TT} ( \eta  w )_{yy} I( \eta _{yyt})\,\dx\\
& =  \frac{\rm d}{\dt}\int_{\TT} ( \eta  w )_{yy} I( \eta_{yy})\,\dx + \int_{\TT} ( \eta  w )_{yy} I\left(( \eta  w )_{yyx}\right)\,\dx  \\
&  =  \frac{\rm d}{\dt}\int_{\TT} ( \eta  w )_{yy} I( \eta_{yy})\,\dx  + \int_{\TT} ( \eta  w )_{yy} (( \eta  w )_{yy} (t,x,y) - ( \eta  w )(t,0,y))\,\dx \\
&  =  \frac{\rm d}{\dt}\int_{\TT} ( \eta  w )_{yy} I( \eta_{yy})\,\dx  + \int_{\TT} ( ( \eta  w )_{yy} )^{2}\,\dx\\
& =  \frac{\rm d}{\dt}\int_{\TT} ( \eta  w )_{yy} I( \eta_{yy})\,\dx + \int_{\TT} (\eta_{yy} w + 2 \eta_{y} w_{y} + \eta w_{yy})^{2}\,\dx \\
&  \leq \frac{\rm d}{\dt}\int_{\TT} ( \eta  w )_{yy} I( \eta_{yy})\,\dx + C \left( E_{10}^{(1)} (y) + E_{20}(y)\right) e^{-\a t}+ C  \int_{\TT}  \eta w_{yy}^{2}\,\dx,
\end{align*}
where we used $\eqref{CNS-1d-dyy}_{1}$, \eqref{csv-m-m-1-dyy}, \eqref{energy-0-0-yy}, and the decay estimates we have obtained in the previous sections.

Similarly, we have
\begin{align*}%\label{decay-L2-yy-4}
- \int_{\TT}  ( \eta  w^2)_{yy}  \eta_{yy}\,\dx  & = - \int_{\TT}  ( w^{2} \eta_{yy} + 2 \eta_{y} w w_{y} + \eta (w w_{yy} + w_{y}^{2}))  \eta_{yy}\,\dx \\
& \leq C \left( E_{10}^{(1)} (y) + E_{20}(y)\right) e^{-\a t}+ C  \int_{\TT}  \eta w_{yy}^{2}\,\dx.
\end{align*}

For the last term in \eqref{decay-L2-yy-1}, we find
\ba\label{decay-L2-yy-40}
  \nu  \int_{\TT}w_{yyx}  \eta_{yy} \,\dx  \leq \de \nu  \int_{\TT}  \eta_{yy}^{2}\,\dx +  \de^{-1} \nu  \int_{\TT} w_{yyx}^{2}\,\dx.
\nn \ea

 Notice that $p'(\eta) \geq p'(\underline \eta) >0$, by substituting the above estimates into \eqref{decay-L2-yy-1}
  and taking $\de$ to be suitably small, we achieve
\ba\label{decay-L2-yy-5}
\frac{p'(\underline{\eta})}{2} \int_{\TT}  \eta_{yy}^2 \,\dx\leq \frac{\rm d}{\dt}\int_{\TT} ( \eta  w )_{yy} I( \eta_{yy})\,\dx + C \left( E_{10}^{(1)} (y) + E_{20}(y)\right) e^{-\a t}+ C  \int_{\TT}  \bigl(\eta w_{yy}^{2}+  w _{yyx}^2\bigr) \,\dx.
\ea
While it follows from a similar  proof of Lemma \ref{prop-kinetic-tx} that
\ba\label{decay-L2-yy-6}
 \int_{\TT} \bigl( \eta  w_{yy}^2 - \langle  \eta  w_{yy} \rangle^2\bigr) \,\dx  \leq \bbeta^2 \int_{\TT}  w _{yyx}^2 \,\dx.
\nn \ea
In view of \eqref{csv-m-m-1-dyy}, it holds that
$$
\langle  \eta  w_{yy} \rangle = - \langle  w \eta_{yy}  \rangle - \langle  2 \eta_{y} w_{y} \rangle.
$$
As a result, we infer
\ba\label{decay-L2-yy-7}
 \int_{\TT}  \eta  w_{yy}^2\,\dx
 & \leq \bbeta^2 \int_{\TT}  w_{yyx}^2 \,\dx +  \langle  \eta w_{yy}  \rangle^2 \\
 & \leq \bbeta^2 \int_{\TT}  w _{yyx}^2 \,\dx +  8 \int_{\TT} \bigl( \eta_{yy}^2  w^2 + \eta_{y}^{2} w_{y}^{2} \bigr) \,\dx\\
 & \leq C \left( E_{10}^{(1)} (y) + E_{20}(y)\right) e^{-\a t}  + C\int_{\TT}  w _{yyx}^2 \,\dx.
\ea
Thanks to \eqref{decay-L2-yy-5} and \eqref{decay-L2-yy-7}, we deduce  that
\ba\label{decay-L2-yy-8}
\frac{p'(\underline{\eta})}{2}\int_{\TT}  \eta_{yy}^2 \,\dx  - \frac{\rm d}{\dt}\int_{\TT} ( \eta  w )_{yy} I( \eta_{yy})\,\dx \leq  C \left( E_{10}^{(1)} (y) + E_{20}(y)\right) e^{-\a t}  + C\int_{\TT}  w _{yyx}^2 \,\dx.
\ea

Let $A$ be a sufficiently large positive constant, we denote
\be\label{F2-1-def}
F_1^{(2)}\eqdefa \int_{\TT} \bigl(A ( \eta  w_{yy}^2 +  \eta_{yy}^2) - ( \eta  w )_{yy} I( \eta_{yy}) \bigr)\,\dx \andf
F_2^{(2)}\eqdefa F_1^{(1)} + F_1^{(2)},
\ee
where $F_1^{(1)} $ is given in \eqref{F1-1-def}.

Then by virtue of \eqref{energy-0-3-dyy} and \eqref{decay-L2-yy-8},  we get
\ba\label{decay-L2-yy-9}
 \frac{\rm d}{\dt} F_1^{(2)}(t) + \int_{\TT}\bigl( w _{yyx}^2 + \frac{ p'(\underline{\eta})}{2}  \eta_{yy}^2 \bigr)\,\dx  \leq  C \left( E_{10}^{(1)} (y) + E_{20}(y)\right) e^{-\a t}  .
\ea
Notice that
\ba\label{decay-L2-yy-10}
\left| \int_{\TT} ( \eta  w )_{yy} I ( \eta_{yy})\,\dx \right|  &  \leq  \| ( \eta  w )_{yy} \|_{L^1_\h}   \| \eta_{yy} \|_{L^1_\h} \\
& \leq  \| \eta_{yy} w + 2 \eta_{y} w_{y} + \eta w_{yy} \|_{L^1_\h}   \| \eta_{yy} \|_{L^1_\h} \\
& \leq C \int_{\TT}\bigl( \eta  w_{yy}^2 +  \eta_{yy}^2 + \eta_{y}^{2} + w_{y}^{2}\bigr)\,\dx.
\ea
 Then by \eqref{decay-L2-y-1} and \eqref{decay-L2-yy-9}, we infer
\ba\label{decay-L2-yy-11}
\frac{\rm d}{\dt} F_2^{(2)}(t) + \int_{\TT}\bigl( w_{yx}^{2} + w _{yyx}^2 + \frac{ p'(\underline{\eta})}{2}  (\eta_{y}^{2} + \eta_{yy}^2 )\bigr)\,\dx  \leq  C \left( E_{10}^{(1)} (y) + E_{20}(y)\right) e^{-\a t}  .
\ea

On the other hand, in view of  \eqref{decay-L2-y-10}, \eqref{decay-L2-yy-7} and  \eqref{decay-L2-yy-10}, by choosing $A$ suitably large, we deduce that
\ba\label{decay-L2-yy-12}
\int_{\TT}  \bigl(\eta  w _y^2 +  \eta _y^2 + \eta  w_{yy}^2 +  \eta_{yy}^2 \bigr) \,\dx  \leq F_2^{(2)}(t)  & \leq C \int_{\TT}
\bigl(\eta  w _y^2 +  \eta _y^2 + \eta  w_{yy}^2 +  \eta_{yy}^2\bigr) \,\dx  \\
& \leq C  \int_{\TT} \bigl(w_{xy}^2 +  \eta_y^2 +  w_{yyx}^2 +  \eta_{yy}^2\bigr) \,\dx.
\ea
 Our desired estimate \eqref{decay-L2-yy} follows from \eqref{decay-L2-yy-11} and \eqref{decay-L2-yy-12}.
\end{proof}

\subsection{Proof of Proposition \ref{S2prop3}}

\begin{proof}[Proof of Proposition \ref{S2prop3}] We divide the proof into the following two steps:

\medskip

\noindent{\bf Step 1.} Decay estimate of $\eta_{yyx}.$

\medskip

The idea to derive the decay estimate of $\eta_{yyx}$ is similar as that  of $\eta_{yxx}$ in Section \ref{sec-decay-y-all}.
 In what follows, we just outline its derivation. In order to do so, wfirst e rewrite $\eqref{CNS-1d-dyy}_{2}$ as
\ba\label{decay-H1-yy-1}
 D_t w_{yy}  +  (2 w_{y} w_{yx} + w_{yy}w_{x}) +\eta^{-1}\bigl( 2 \eta_{y} (D_{t} w_{y} + w_{y} w_{x}) +
  \eta_{yy} D_{t} w - \nu  w _{yyxx} +  p( \eta )_{yyx}\bigr) = 0.
\ea
While it follows from  $\eqref{CNS-1d-dyy}_{1}$ that
\ba\label{decay-H1-yy-2}
\zeta w_{yyxx} = D_{t}\zeta_{yyx} - \zeta_{yy} w_{xx} - \zeta_{y} w_{yxx} -  \zeta_{y} w_{yxx}  + \zeta_{yxx} w_{y} + \zeta_{yxx} w_{y} + \zeta_{xx} w_{yy}.
\ea
Plugging \eqref{decay-H1-yy-2} into \eqref{decay-H1-yy-1} gives
\ba\label{decay-H1-yy-3}
 D_t (w_{yy} -\nu \zeta_{yyx}) +   \eta^{-1} p( \eta )_{yyx} + \frak{g}_1 = 0,
\ea
with
\begin{align*}%\label{decay-H1-yy-4}
\frak{g}_1 =& (2 w_{y} w_{yx} + w_{yy}w_{x}) + 2 \eta^{-1}\eta_{y} (D_{t} w_{y} + w_{y} w_{x}) + \eta^{-1} \eta_{yy} D_{t} w\\
 &- \nu (- \zeta_{yy} w_{xx} - \zeta_{y} w_{yxx} -  \zeta_{y} w_{yxx}  + \zeta_{yxx} w_{y} + \zeta_{yxx} w_{y} + \zeta_{xx} w_{yy}).
\end{align*}
Then we  deduce from Propositions \ref{S2prop1}, \ref{S2prop2} and \ref{prop-decay-L2-yy} that
\ba\label{decay-H1-yy-5}
\|\frak{g}_2\|_{L^2_{h}}^{2} \leq C E(y)  e^{-\a t}.
\ea

By taking $L^2(\TT)$ inner product of \eqref{decay-H1-yy-3} with $\eta (w_{yy} -\nu \zeta_{yyx})$ and using
\eqref{decay-H1-yy-5}, we find
\begin{align*}%\label{decay-H1-yy-6}
\frac 12 \frac{\rm d}{\dt}\int_{\TT}\eta(w_{yy} -\nu \zeta_{yyx})^{2}\,\dx &-\int_{\TT}p'(\eta) \zeta^{-2} \zeta_{yxx} (w_{yy} -\nu \zeta_{yyx})\,\dx\\
 &\leq C \de^{-1}E(y)  e^{-\a t} + \de \int_{\TT} \eta(w_{yy} -\nu \zeta_{yyx})^{2}\,\dx,
\end{align*}
from which, we infer
\ba\label{decay-H1-yy-7}
\frac 12 \frac{\rm d}{\dt}\int_{\TT} \eta(w_{yy} -\nu \zeta_{yyx})^{2}\,\dx
 &+  \nu^{-1}\int_{\TT}  p'(\eta) \eta^{2}(w_{yy} -\nu \zeta_{yyx})^{2}\,\dx \\
 &\leq C \de^{-1}E(y)  e^{-\a t} + \de \int_{\TT} \eta(w_{yy} -\nu \zeta_{yyx})^{2}\,\dx.
\ea
Taking $\de = (2\nu)^{-1}p'(\underline{\eta})\underline{\eta}$ in  \eqref{decay-H1-yy-7} gives
\ba\label{decay-H1-yy-8}
\frac{\rm d}{\dt} \int_{\TT} \eta(w_{yy} -\nu \zeta_{yyx})^{2}\,\dx + \nu^{-1}p'(\underline{\eta})\underline{\eta} \int_{\TT} \eta (w_{yy} -\nu \zeta_{yyx})^{2}\,\dx \leq C E(y)  e^{-\a t},
\nn \ea
from which, we infer
\ba\label{decay-H1-yy-9}
 \int_{\TT} \eta(w_{yy} -\nu \zeta_{yyx})^{2}\,\dx  \leq C E(y)  e^{-\a t},
\ea
and
\ba\label{decay-H1-yy-10}
 \int_{\TT} |\zeta_{yyx}|^{2}\,\dx \leq C E(y)  e^{-\a t} \andf  \int_{\TT} |\eta_{yyx}|^{2}\,\dx \leq C E(y)  e^{-\a t}.
\ea

\medskip

\noindent{\bf Step 2.} {Decay estimates of $D_{t} w_{yy}.$}

\medskip

The main idea to derive the decay estimates of $D_{t} w_{yy}$  is analogues to that of $D_{t} w_{yx}$
in Section {sec-decay-y-all}. We shall outline its proof below.

Observing that
\begin{align*}\label{Dt-wyy-2}
\int_{\TT}  D_{t} w_{yy}  w_{yyxx} & = \int_{\TT}  \d_{t} w_{yy} w_{yyxx}\,\dx + \int_{\TT}  w w_{yyx} w_{yyxx}\,\dx\\
 & = - \frac{1}{2}\frac{\rm d}{\dt} \int_{\TT}  |w_{yyx}|^{2} \,\dx- \frac{1}{2} \int_{\TT} w_{x} (w_{yyx})^{2}\,\dx.
\end{align*}
Then by taking $L^2(\TT)$ inner product of \eqref{CNS-1d-dyy-1} with $\eta^{-1} w_{yyxx},$ we deduce that
\ba\label{Dt-wyy-1}
\frac{\rm d}{\dt} \int_{\TT}  |w_{yyx}|^{2}\,\dx + \nu \bar\eta^{-1} \int_{\TT}   |w_{yyxx}|^{2}\,\dx \leq C  E(y) e^{-\a t}.
\ea

While due to
\begin{align*}%\label{Dt-wyy-7}
- \nu\int_{\TT}   w_{yyxx}  D_{t} w_{yy}\,\dx & = - \nu \int_{\TT}  w_{yyxx}  \d_{t} w_{yy}\,\dx
 - \nu \int_{\TT}   w_{yyxx}  w w_{yyx}\,\dx \\
& =\frac\nu2 \frac{\rm d}{\dt} \int_{\TT}   |w_{yyx}|^{2} + \frac{\nu}{2} \int_{\TT}     (w_{yyx})^{2}  w_{x}\\
& \geq \frac\nu2\frac{\rm d}{\dt} \int_{\TT} |w_{yyx}|^{2}\,\dx -  C  E(y) e^{-\a t}\int_{\TT}  |w_{yyx}|^{2}\,\dx,
\end{align*}
we get, by taking $L^2(\TT)$ inner product of \eqref{CNS-1d-dyy-1} with $ D_{t} w_{yy},$ that
\ba\label{Dt-wyy-5}
\nu\frac{\rm d}{\dt} \int_{\TT}  |w_{yyx}|^{2}\,\dx + \int_{\TT} \eta  |D_{t}w_{yy}|^{2}\,\dx \leq C  E(y) e^{-\a t} + C  E(y) e^{-\a t}\int_{\TT}  |w_{yyx}|^{2}\,\dx.
\ea

Notice that
\begin{align*}%\label{Dt-wyy-11}
 \int_{\TT}D_{t} (\eta D_{t} w_{yy} ) D_{t} w_{yy}\,\dx  & = \int_{\TT} (D_{t} \eta) D_{t} w_{yy} D_{t} w_{yy}\,\dx
  + \int_{\TT}\eta (D_{t}^2 w_{yy})D_{t} w_{yy} \,\dx\\
& =  \frac{1}{2}\frac{\rm d}{\dt} \int_{\TT}\eta |D_{t} w_{yy}|^{2}  +\frac12 \int_{\TT} (D_{t} \eta) |D_{t} w_{yy}|^{2} \\
& \geq  \frac{1}{2}\frac{\rm d}{\dt} \int_{\TT}\eta |D_{t} w_{yy}|^{2}  - C \int_{\TT} |D_{t} w_{yy}|^{2}\,\dx
\end{align*}
and
\begin{align*}%\label{Dt-wyy-12}\label{Dt-wyy-13}
- \nu \int_{\TT}  (D_{t} w_{yyxx}) D_{t} w_{yy}\,\dx  & = -\nu\int_{\TT}  (D_{t} w_{yy})_{xx} D_{t} w_{yy}\,\dx
 + \int_{\TT} \nu  (w_{xx} w_{yyx} + 2 w_{x} w_{yyxx})D_{t} w_{yy}\,\dx \\
& \geq \nu \int_{\TT} |(D_{t} w_{yy})_{x}|^{2}\,\dx  - C \int_{\TT} \bigl(|D_{t} w_{yy}|^{2} + |w_{yyxx}|^{2}\bigr)\,\dx,
\end{align*}
we get, by applying $D_{t}$ to \eqref{CNS-1d-dyy-1} and then taking $L^2(\TT)$ inner product of the resulting equation with  $D_{t} w_{yy},$
that
\ba\label{Dt-wyy-9}
\frac{\rm d}{\dt} \int_{\TT} \eta |D_{t} w_{yy}|^{2}\,\dx + \nu \int_{\TT} |(D_{t} w_{yy})_{x}|^{2}\,\dx \leq C  E(y) e^{-\a t} + C \int_{\TT}   \bigl(|w_{yyxx}|^{2} + |D_{t}w_{yy}|^{2} + |w_{yyx}|^{2}\bigr)\,\dx.
\ea

By virtue of \eqref{decay-L2-yy-9}--\eqref{decay-L2-yy-12}, we deduce from \eqref{Dt-wyy-1}, \eqref{Dt-wyy-5} and \eqref{Dt-wyy-9} that
\ba\label{Dt-wyy-20}
 \|D_{t} w_{yy}\|_{L^{2}_{h}}^{2} \leq C  E(y) e^{-\a t},
\ea
from which and \eqref{CNS-1d-dyy}, we infer
\ba\label{Dt-wyy-21}
 \|w_{yyxx}\|_{L^{2}_{h}}^{2} \leq C  E(y) e^{-\a t} .
\ea
By summarizing the estimates \eqref{decay-L2-yy}, \eqref{decay-H1-yy-10}, \eqref{Dt-wyy-20} and
 \eqref{Dt-wyy-21}, we conclude the proof of \eqref{decay-exp-all-yy-2}.
\end{proof}

\section{Decay estimates of $\frak{w}$}\label{Sect6}

With the decay estimates for $(\eta, w)$ obtained in the previous sections,
 we are going to derive the same exponential decay estimates for $\frak{w}$.
Notice that  the equation for $\frak{w},$ \eqref{tw}, is of standard parabolic type,  we shall only
 present the main decay in time estimates of $\frak{w}$ and skip the derivation of the related estimates
 concerning $(\d_y\frak{w},\d_y^2\frak{w}).$

 The main results state as follows

\begin{prop}\label{prop-tw-L2}
{\sl For all $t\in\R^{+}$, one has
\begin{subequations} \label{Streq19}
\begin{gather}
\label{tw-L2}
\int_{\TT} |\frak{w}|^{2}\,\dx  \leq  \underline \eta^{-1} e^{- \mu\bar \eta^{-2} t } \int_{\TT}\eta_{0}  |\frak{w}_{0}|^{2}\,\dx,\\
\label{tw-H1}
\int_{\TT}\bigl( |\frak{w}|^{2} + |\frak{w}_{x}|^2\bigr)\,\dx  \leq  CE(y) e^{- \a t },\\
\label{tw-H2}
\int_{\TT}\left( |\frak{w}|^{2} + |\frak{w}_{x} |  + | D_{t} \frak{w} |^{2} + |\d_{xx} \frak{w}|^{2}\right)\,\dx  \leq  C E(y)e^{- \a t }.
\end{gather}
\end{subequations}}
\end{prop}

\begin{proof} 1) We first get, by
taking  $L^{2}(\TT)$ inner product of \eqref{tw} with $\frak{w}$ and using integrating by parts, that
\ba\label{tw-L2-0}
\frac{1}{2}\frac{\rm d}{\dt}\int_{\TT} \eta |\frak{w}|^{2} \,\dx +   \mu\int_{\TT} |\frak{w}_{x}|^{2}\,\dx = 0.
\ea

 While observing from the density equation of \eqref{CNS-limit} and \eqref{tw} that
\ba\label{tw-new-1}
(\eta  \frak{w})_{t} +  (\eta w \frak{w})_{x} - \mu \d_x^2 \frak{w}  = 0.
\nn \ea
Integrating the above equation over $\TT$ gives
\ba\label{tw-L2-1}
\frac{\rm d}{\dt} \int_{\TT}(\eta  \frak{w})\,\dx = 0,
\nn \ea
which together with \eqref{ini-1} ensures that
\be\label{tw-L2-2}
\int_{\TT}(\eta  \frak{w})\,\dx = \int_{\TT} \eta_{0} \frak{w}_{0} \,\dx= 0.
\ee
Then we get, by using a similar  proof of Lemma \ref{prop-kinetic-tx}, that
\be\label{tw-L2-3}
\int_{\TT}\eta  |\frak{w}|^{2}\,\dx \leq \bar \eta^{2} \int_{\TT} | \frak{w}_{x} |^{2}\,\dx.
\ee
By virtue of \eqref{tw-L2-3}, we deduce from \eqref{tw-L2-0} that
\ba
\int_{\TT}\eta  |\frak{w}|^{2}\,\dx \leq  e^{\mu\bar \eta^{-2} t } \int_{\TT}\eta_{0}  |\frak{w}_{0}|^{2}\,\dx,
\nn \ea which together with \eqref{lowerb-vtr} ensures \eqref{tw-L2}.

\medskip

\noindent 2) By taking $L^2$ inner product of  \eqref{tw} with $D_{t} \frak{w}$  and using  integrating by
parts, we obtain
\ba\label{tw-H1-0}
\frac{1}{2}\int_{\TT} \eta |D_{t}\frak{w}|^{2} \,\dx + \frac{\mu}{2}\frac{\rm d}{\dt}\int_{\TT}  |\frak{w}_{x}|^{2}  \,\dx  \leq  \frac{\mu}{2} \|w_{x}\|_{L^{\infty}(\OO)} \int_{\TT} |\frak{w}_{x}|^{2} \,\dx.
\ea

By multiplying \eqref{tw-L2-0} by $\frak{A}_{1}\eqdefa 1 + \|w_{x}\|_{L^{\infty}([0,\infty)\times\OO)}$ and summing up the
resulting inequality with \eqref{tw-H1-0}, we get
\ba
\int_{\TT}\eta |D_{t}\frak{w}|^{2} \,\dx + \frac{\rm d}{\dt}  \int_{\TT}\bigl({\frak{A}_{1}}\eta | \frak{w}|^{2} +  {\mu} |\frak{w}_{x}|^{2}\bigr)  \,\dx  +  \frak{A}_{1}\int_{\TT} \mu |\frak{w}_{x}|^{2}\,\dx \leq 0.
\nn \ea
Due to $D_t\frak{w}=\frac\mu\eta\d_x^2 \frak{w},$ \eqref{tw-H1} follows.

\medskip

\noindent 3)
Applying $D_{t}$ to \eqref{tw} gives
\ba\label{tw-Dtw}
D_{t} (\eta D_{t} \frak{w}) - \mu  D_{t} \d_x^2 \frak{w}  = 0.
\nn \ea
We observe that
\ba\label{tw-Dtw-1}
D_{t} (\eta D_{t} \frak{w})  = (-\eta \d_{x} w) D_{t}  \frak{w} + \eta D_{t}^{2} \frak{w},
\nn \ea
and
\ba
D_{t} \d_x^2 \frak{w}
 = \d_{x}^{2} (D_{t} \frak{w}) - \d_{x} ( \d_{x} w \d_{x} \frak{w}).
\nn \ea
As a result, it comes out
\ba\label{tw-Dtw-3}
 \eta D_{t}^{2} \frak{w}  - \mu \d_{x}^{2} (D_{t} \frak{w})  =  (\eta \d_{x} w) D_{t}  \frak{w}  - \mu \d_{x} ( \d_{x} w \d_{x} \frak{w}).
 \ea
By taking $L^2(\TT)$ inner product of  \eqref{tw-Dtw-3} with  $D_{t} \frak{w}$, we find
\ba\label{tw-Dtw-4}
\frac{1}{2} &\frac{\rm d}{\dt} \int_{\TT} \eta |D_{t} \frak{w}|^{2} \,\dx + \mu\int_{\TT}  |  \d_{x} (D_{t} \frak{w})|^{2}\,\dx \\
& = \int_{\TT}  (\eta \d_{x} w) |D_{t}  \frak{w}|^{2}\,\dx  + \int_{\TT}  \mu (\d_{x} w \d_{x} \frak{w}) \d_{x} D_{t} \frak{w}\,\dx\\
& \leq \| \d_{x} w \|_{L^{\infty}(\OO)} \int_{\TT} \eta |D_{t}  \frak{w}|^{2}\,\dx  +  \frac{\mu}{2} \|\d_{x} w \|_{L^{\infty}(\OO)}^{2} \int_{\TT}  |\d_{x} \frak{w}|^{2}\,\dx  + \frac{\mu}{2}\int_{\TT} |  \d_{x} (D_{t} \frak{w})|^{2}\,\dx.
\ea
Multiplying \eqref{tw-H1-0} by $\frak{A}_{2} \eqdefa 2 + 2 \|\d_{x} w \|_{L^{\infty}([0,\infty)\times\OO)}^{2}$   and summing up the resulting inequality with \eqref{tw-Dtw-4} yields
\ba\label{tw-Dtw-5}
 \frac{\rm d}{\dt}  \int_{\TT}\bigl( \frak{A}_{1}\frak{A}_{2} \eta | \frak{w}|^{2} +  \mu A_{2} |\frak{w}_{x}|^{2}   +  \eta |D_{t} \frak{w}|^{2}\bigr)  \,\dx  +  \int_{\TT}  \bigl(\mu |\frak{w}_{x}|^{2} + \eta |D_{t}\frak{w}|^{2} +  \mu |  \d_{x} (D_{t} \frak{w})|^{2}\bigr)\,\dx  \leq 0.
\ea
Then by using a similar proof of Lemma \ref{prop-kinetic-tx} and Gronwall's inequality, we find
\be\label{tw-Dtw-6}
\int_{\TT} \bigl(|\frak{w}|^{2} + |\frak{w}_{x} |  + | D_{t} \frak{w} |^{2}\bigr)\,\dx   \leq  C E(y) e^{- \a t }.
\ee
Observing that $\mu\d_{xx} \frak{w} =  \eta D_{t} \frak{w}, $ we conclude the proof of \eqref{tw-H2}.
This completes the proof of Proposition \ref{prop-tw-L2}.
\end{proof}

Let us now outline  the proof of Proposition \ref{S2prop4}.

\begin{proof}[Proof of Proposition \ref{S2prop4}]
 Along the same line to proof of Proposition \ref{prop-tw-L2} and through the induction method as what we used in
 the Appendix \ref{appa}, we deduce that
\ba\label{tw-Hk-0}
\| \frak{w}(t)\|_{H^5_\h} + \|\frak{w}_{t}(t)\|_{H^3_\h} + \| \frak{w}_{tt}(t) \|_{H^1_\h}  \leq CE(y) e^{-\a t}.
\ea
The decay estimates related to $y$-derivatives of $\frak{w}$ in \eqref{S2eq4} can be derived along the same line. We omit  the details here.
\end{proof}

\section{Energy estimates for the perturbed equations}\label{sec-error}

The purpose of this section is to present the proof of Propositions \ref{prop-energy-basic} and \ref{energy}.
For simplicity, we shall neglect the subscript $\e$ in the rest of this section.

\subsection{Basic energy estimate}\label{Sect7.1}

In this subsection, we shall derive a basic energy estimate for all $t < T^\star.$ We first deduce from Proposition \ref{prop-rela}
that

\begin{lem}\label{prop-rela-w}
{\sl Let $(\rho, u)$ and $(\rho^{\rm a}, u^{\rm a})$ be respectively given by \eqref{zeta-W-def} and \eqref{sl-1d-00}. Then one has
\ba\label{ineq-entropy-1}
 \calE_1\bigl((\rho,u) | (\rho^{\rm a}, u^{\rm a} )\bigr)(t)
 + \int_0^t  \int_\Omega \bigl(\mu \left| \nabla (u-u^{\rm a})  \right|^2 + \mu' |\dive (u-u^{\rm a}) |^2\bigr)\,\dx\,\dy \,\dt' =\int_0^t \calR (t') \,\dt',
\ea
where for $G$ given by  \eqref{def-G12},
\ba\label{R-def-1}
\calR (t) & \eqdefa \int_\Omega \Bigl((\rho^{\rm a})^{-1} (\rho -\rho^{\rm a}) ( \mu \Delta u^{\rm a} + \mu' \nabla \dive u^{\rm a}
)+ \rho  (u-u^{\rm a} )\cdot \nabla u^{\rm a}\\
&\qquad+\rho (\rho^{\rm a})^{-1} G\Bigr)\cdot (u^{\rm a} - u)\,\dx\,\dy + \e\int_{\OO}(\rho^{\rm a}-\rho)  (\rho^{\rm a})^{-1} p'(\rho^{\rm a})  [(\eta \frak{w})_{y}]_\e \,\dx \, \dy  \\
& \quad   - \int_\Omega \dive u^{\rm a}  \bigl(p(\rho ) - p(\rho^{\rm a}) - p'(\rho^{\rm a})(\rho - \rho^{\rm a})\bigr)\,\dx\,\dy.
\ea
}
\end{lem}

\begin{proof} Since  $(\rho, u)$ and $(\rho^{\rm a}, u^{\rm a})$  have the same initial data, we get, by
applying  Proposition \ref{prop-rela-w}, that  \eqref{ineq-entropy-1} holds with
\ba\label{R-def2}
\calR(t) \eqdefa \int_\Omega \Bigl(&\rho \frak{D}_t u^{\rm a} \cdot (u^{\rm a} - u)+ \mu \nabla u^{\rm a}:\nabla (u^{\rm a} - u)
+ \mu' \dive u^{\rm a} \,\dive (u^{\rm a} - u)\\
&  +  (\rho^{\rm a} - \rho) \d_t P'(\eta) + (\rho^{\rm a} u^{\rm a}-\rho u) \cdot \nabla P'(\rho^{\rm a})- \dive u^{\rm a} (p(\rho ) - p(\rho^{\rm a}))\Bigr)\,\dx\,\dy.
\ea
It follows from the $u^{\rm a}$ equation of \eqref{CNS-1d-W} that
\begin{align*}
\calR_{1}(t)  & \eqdefa \int_\Omega\rho \bigl( \d_t u^{\rm a} + u^{\rm a}\cdot \nabla u^{\rm a} +   (u-u^{\rm a})\cdot \nabla u^{\rm a}\bigr)
\cdot ( u^{\rm a} - u)\,\dx\,\dy\\
& =  \int_\Omega \rho\Bigl((\rho^{\rm a})^{-1} \bigl( \mu \Delta u^{\rm a} + \mu' \nabla \dive u^{\rm a}
 - \nabla p( \rho^{\rm a} )  + G\bigr)+   (u-u^{\rm a})\cdot \nabla u^{\rm a}\Bigr) ( u^{\rm a} - u)\,\dx\,\dy.
\end{align*}
By using integration by parts and  the fact that $P''(s) = s^{-1} p'(s)$, we find
\ba
\calR_{1}(t)
& =  \int_\Omega \bigl((\rho^{\rm a})^{-1} (\rho -\rho^{\rm a}) \bigl( \mu \Delta u^{\rm a}  + \mu' \nabla \dive u^{\rm a} \bigr)
 -  \rho  \nabla P'(\rho^{\rm a})+\rho (\rho^{\rm a})^{-1} G\bigr)\cdot (  u^{\rm a}- u)\,\dx\,\dy \\
 &\quad - \int_\Omega\bigl( \mu \nabla  u^{\rm a} :\nabla ( u^{\rm a} - u)  + \mu' \dive  u^{\rm a}
\,\dive ( u^{\rm a} - u)\bigr)\,\dx\,\dy    \\
& \quad + \int_\Omega \rho  (u- u^{\rm a})\cdot \nabla  u^{\rm a} \cdot ( u^{\rm a} - u)\,\dx\,\dy.
\nn\ea
 Plugging the above equality into \eqref{R-def2} gives
\ba\label{R-def2-2}
\calR(t) & = \int_\Omega \Bigl((\rho^{\rm a})^{-1} (\rho -\rho^{\rm a}) \bigl( \mu \Delta u^{\rm a} + \mu' \nabla \dive u^{\rm a} \bigr)
+\rho (\rho^{\rm a})^{-1} G +  (u-u^{\rm a})\cdot \nabla u^{\rm a}\\
&\qquad- \rho  \nabla P'(\rho^{\rm a})\Bigr)
\cdot (u^{\rm a}- u)\,\dx\,\dy- \int_\Omega \dive u^{\rm a} (p(\rho ) - p(\rho^{\rm a}))\,\dx\,\dy\\
 &\quad+ \int_\Omega  \bigl((\rho^{\rm a} - \rho) \d_t P'(\eta) + (\rho^{\rm a} u^{\rm a} -\rho u) \cdot \nabla P'(\rho^{\rm a})\bigr) \,\dx\,\dy.
\ea
Notice that
$P''(s) = s^{-1} p'(s)$ and the renormalized equation
\be
\d_t P'(\rho^{\rm a}) + \dive(u^{\rm a} P'(\rho^{\rm a})) + (P''(\rho^{\rm a})\rho^{\rm a} - P'(\rho^{\rm a})) \dive u^{\rm a} = \e  P''(\rho^{\rm a}) [(\eta \frak{w})_{y}]_\e.
\ee
We get, by using the continuity equation $\eqref{CNS-1d-W}_{1}$, that
\ba\label{R-def2-3}
 &   - \rho  \nabla P'(\rho^{\rm a}) \cdot ( u^{\rm a} - u) + (\rho^{\rm a} - \rho) \d_t P'(\eta) + (\rho^{\rm a}  u^{\rm a} -\rho u) \cdot \nabla P'(\rho^{\rm a})\\
 &= (\rho^{\rm a} - \rho) \left( \d_t P'(\rho^{\rm a}) +  u^{\rm a} \cdot \nabla P'(\rho^{\rm a})\right) \\
 &= (\rho^{\rm a}-\rho) \left(\left( \d_t P'(\rho^{\rm a}) + \dive(u^{\rm a}P'(\rho^{\rm a})) + (P''(\rho^{\rm a})\rho^{\rm a} - P'(\rho^{\rm a})) \dive u^{\rm a} \right) -   P''(\rho^{\rm a})\rho^{\rm a} \dive u^{\rm a}\right) \\
 & =  \e (\rho^{\rm a}-\rho)  (\rho^{\rm a})^{-1} p'(\rho^{\rm a}) [(\eta \frak{w})_{y}]_\e
  - (\rho^{\rm a} - \rho) p'(\rho^{\rm a}) \dive  u^{\rm a}.
\ea

Then \eqref{R-def-1} follows by inserting \eqref{R-def2-3} into \eqref{R-def2-2}.
\end{proof}

We now present the proof of Proposition \ref{prop-energy-basic}.

\begin{proof}[Proof of Proposition \ref{prop-energy-basic}]
We first get, by applying \eqref{upper-lower-rho-0} and Taylor's expansion, that for $t < T^\star,$
\ba\label{rela-entr-lower}
\calE_1(t)\eqdefa
\calE_1\bigl((\rho,u) | (\rho^{\rm a}, u^{\rm a})\bigr)(t) & = \int_{\Omega} \Bigl(\frac{1}{2} \rho |u- u^{\rm a}|^{2} + P(\rho) - P(\rho^{\rm a}) - P'(\rho^{\rm a}) (\rho - \rho^{\rm a})\Bigr)\,\dx\,\dy\\
& \geq \int_{\OO}  \bigl({\underline \eta}/{4}  | R|^{2} + \g (2\bar \eta)^{\g-2} | \vr |^{2}\bigr)\,\dx\,\dy.
\ea
While it follows from  Lemma  \ref{prop-rela-w} that
\ba\label{energy-bas1}
 \calE_1(t)+ \int_0^t  \int_\Omega\bigl( \mu \left| \nabla R \right|^2 + \mu' |\dive R|^2\bigr)\,\dx\,\dy \,\dt' \leq\int_0^t \calR (t') \,\dt'.
\ea
According to \eqref{R-def-1}, we decompose $\calR$ as
$
\calR (t) \eqdefa  \sum_{j=1}^{5} \calR_{j}(t).
$
We first deduce from Theorem \ref{thm1} that
\begin{align*}
\calR_{1} (t) & \eqdefa \int_\Omega (\rho^{\rm a})^{-1} (\rho -\rho^{\rm a}) ( \mu \Delta u^{\rm a}
 + \mu' \nabla \dive u^{\rm a} )\cdot (u^{\rm a} - u)\,\dx\,\dy\\
& \leq  \|(\rho^{\rm a})^{-1}\|_{L^{\infty}(\OO)}  \| (\mu \Delta u^{\rm a} + \mu' \nabla \dive u^{\rm a})\|_{L^{\infty}(\OO)} \|\vr\|_{L^{2}(\OO)} \|R\|_{L^{2}(\OO)} \\
& \leq C e^{-\a t} \bigl( \|\vr\|_{L^{2}(\OO)}^{2} +  \|R\|_{L^{2}(\OO)}^{2} \bigr).
\end{align*}

For the second term in \eqref{R-def-1}, we have
 \ba
 \calR_{2} (t) \eqdefa \int_\Omega \rho (\rho^{\rm a})^{-1} G \cdot ( u^{\rm a}- u)\,\dx\,\dy
  \leq  \|\rho\|_{L^{\infty}(\OO)}   \|(\rho^{\rm a})^{-1}\|_{L^{\infty}(\OO)}  \|G\|_{L^{2}(\OO)}  \|R\|_{L^{2}(\OO)},
\nn \ea
which together with \eqref{def-G12-est}  and Theorem \ref{thm1} ensures  that
\ba
\calR_{2} (t) \leq C e^{-\a t} \e^{\frac 12} \|R\|_{L^{2}(\OO)}.
\nn \ea

For the third term in \eqref{R-def-1}, we get, by applying Theorem \ref{thm1}, that
\begin{align*}
\calR_{3} (t) & \eqdefa  \int_\Omega \rho  (u-u^{\rm a})\cdot \nabla u^{\rm a} \cdot (u^{\rm a} - u)\,\dx\,\dy\\
& \leq \|\nabla u^{\rm a}\|_{L^{\infty}(\OO)} \int_\Omega \rho  |R|^{2} \,\dx\,\dy  \leq C e^{-\a t} \calE_{1}(t).
\end{align*}
Along the same line and thanks to \eqref{pre-pot}, we have
\begin{align*}
\calR_{4} (t) & \eqdefa - \int_\Omega \dive u^{\rm a}  \bigl(p(\rho ) - p(\rho^{\rm a}) - p'(\rho^{\rm a})(\rho - \rho^{\rm a})\bigr)\,\dx\,\dy \\
&\leq (\g-1) \|  \dive u^{\rm a} \|_{L^{\infty}(\OO)}  \int_\Omega  \bigl(P(\rho ) - P(\rho^{\rm a}) - P'(\rho^{\rm a})(\rho - \rho^{\rm a})\bigr)\,\dx\,\dy  \leq C e^{-\a t} \calE_{1}(t),
\end{align*}
and
\begin{align*}
& \calR_{5}(t) \eqdefa \e \int_{\OO}(\rho^{\rm a}-\rho)  (\rho^{\rm a})^{-1} p'(\rho^{\rm a})  [(\eta \frak{w})_{y}]_\e \,\dx \, \dy \\
& \leq \e \|(\rho^{\rm a})^{-1} p'(\rho^{\rm a}) \|_{L^{\infty}(\OO)} \|[(\eta \frak{w})_{y}]_\e\|_{L^{2}(\OO)} \|\vr\|_{L^{2}(\OO)} \leq C \e^{\frac{1}{2}} e^{-\a t} \|\vr\|_{L^{2}(\OO)}.
\end{align*}

By inserting the above estimates into \eqref{energy-bas1} and using \eqref{rela-entr-lower},
 we  deduce that
\ba
\calE_{1}(t)  + \int_0^t  \int_\Omega \mu \left| \nabla R  \right|^2 \,\dx\,\dy \,\dt'\leq
 C \int_{0}^{t} e^{-\a t'}\bigl(  \e^{\frac 12}  \calE_{1}(t')^{1/2} +   \calE_{1}(t')\bigr)\,\dt'.
\nn \ea
Applying Gornwall's inequality   gives rise to
\ba
\calE_{1}(t)  + \int_0^t  \int_\Omega \mu \left| \nabla R  \right|^2 \,\dx\,\dy \,\dt'\leq  C \e,
\nn \ea
which together with \eqref{rela-entr-lower} ensures  \eqref{energy-bas}. This completes the proof of Proposition \ref{prop-energy-basic}.
\end{proof}

\subsection{Estimates of $\nabla R$}

\begin{lem}\label{prop-DtR}
{\sl For each $t < T^\star$, there holds
\ba\label{DtR-0}
\int_{\OO} |\nabla R|^{2}\,\dx\,\dy + \int_{0}^{t} \int_{\OO}  |\frak{D}_{t}R|^{2}\,\dx\,\dy
 \leq C \int_{0}^{t}\int_{\OO}  |\nabla R|^{3}\,\dx\,\dy + C \e.
\ea}
\end{lem}

\begin{proof} We first get, by taking $L^2$ inner product of the $R$ equation of \eqref{CNS-error-new} with $\frak{D}_tR$, that
\ba\label{DtR-1}
\int_{\OO} \rho |\fD_{t} R|^{2}\,\dx\,\dy - \int_{\OO} \bigl(\mu \Delta R+ \mu' \nabla\dive R\bigr)\cdot \fD_tR \,\dx\,\dy
+\int_{\OO} \nabla \big(p(\rho)  - p(\rho^{\rm a})  \big)\cdot \fD_tR \,\dx\,\dy \\
 + \int_{\OO}\Bigl(\rho R\cdot \nabla u^{\rm a} +  \vr \bigl(\d_{t}u^{\rm a}  + u^{\rm a} \cdot \nabla u^{\rm a}\bigr)-   G\Bigr) \cdot \fD_tR \,\dx\,\dy = 0.
\ea

Let us now handle term by term above. By using integration by parts, we find
\ba
- \int_{\OO}  \Delta R \cdot \fD_{t} R\,\dx\,\dy
  = \frac{1}{2}\frac{\rm d}{\dt} \int_{\OO}  |\nabla R|^{2}\,\dx\,\dy
   +  \int_{\OO} \nabla R: \bigl( \nabla u \nabla R + (u \cdot \nabla)\nabla R\bigr)\,\dx\,\dy.
 \nn \ea
Due to $u=u^{\rm a}+R,$ one has
 \begin{align*}
  \int_{\OO}\nabla R: ( \nabla u \nabla R)\,\dx\,\dy & =  \int_{\OO}  \nabla R: \bigl( \nabla u^{\rm a} \nabla R\bigr)\,\dx\,\dy
  +  \int_{\OO} \nabla R: ( \nabla R \nabla R) \,\dx\,\dy \\
 & \leq \|\nabla u^{\rm a}\|_{L^{\infty}} \int_{\OO}  |\nabla R|^{2}\,\dx\,\dy +  \int_{\OO}|\nabla R|^{3}\,\dx\,\dy,
 \end{align*}
 and
  \begin{align*}
  \int_{\OO} \nabla R: ( (u \cdot \nabla)\nabla R) \,\dx\,\dy
  & =   -  \frac{1}{2}\int_{\OO}  (\dive u^{\rm a}) |\nabla R|^{2} \,\dx\,\dy
    - \,\frac{1}{2} \int_{\OO} (\dive R) |\nabla R|^{2}\,\dx\,\dy.
\end{align*}
This together with Theorem \ref{thm1} ensures that
\ba\label{DtR-3-1}
- \int_{\OO}  \Delta R \cdot \fD_{t} R\,\dx\,\dy
  \geq \frac{1}{2}\frac{\rm d}{\dt} \int_{\OO}  |\nabla R|^{2}\,\dx\,\dy -C e^{-\alpha t} \int_{\OO}  |\nabla R|^{2}\,\dx\,\dy
  -\int_{\OO}|\nabla R|^{3}\,\dx\,\dy.
  \ea

Exactly along the same line, one has
\ba\label{DtR-3-2}
 -\int_{\OO} \nabla \dive R \cdot \fD_{t} R\,\dx\,\dy\geq &\frac{1 }{2}\frac{\rm d}{\dt} \int_{\OO}  |\dive R|^{2}\,\dx\,\dy\\
& -C e^{-\a t} \int_{\OO}  |\nabla R|^{2}\,\dx\,\dy   -\int_{\OO} |\nabla R|^{3}\,\dx\,\dy .
 \ea

It is a little trickier to deal with the pressure term in \eqref{DtR-1}. Indeed we first observe that
\ba\label{DtR-4-1}
\int_{\OO}& \nabla (p(\rho) - p(\rho^{\rm a})) \cdot \fD_{t}R\,\dx\,\dy
 = - \frac{\rm d}{\dt} \int_{\OO}  (p(\rho) - p(\rho^{\rm a})) \dive R\,\dx\,\dy\\
    &+ \int_{\OO}  \d_{t}(p(\rho) - p(\rho^{\rm a})) \dive R \,\dx\,\dy
- \int_{\OO}  (p(\rho) - p(\rho^{\rm a})) \dive(u\cdot \nabla R)\,\dx\,\dy.
\ea
It follows from the continuity equations of \eqref{CNS} and \eqref{CNS-1d-W} that
\begin{align*}
\d_{t}(p(\rho) - p(\rho^{\rm a}))
& = - \dive\big(p(\rho) u - p(\rho^{\rm a}) u^{\rm a}\big) - a (\g-1)  \big( \rho^{\g} \dive u - (\rho^{\rm a})^{\g} \dive u^{\rm a}\big)
 -\e  p'(\rho^{\rm a})  [(\eta \frak{w})_{y}]_\e
\end{align*}
so that
\begin{align*}
&\int_{\OO}  \d_{t}(p(\rho) - p(\rho^{\rm a})) \dive R\,\dx\,\dy = \int_{\OO} \big(p(\rho) u - p(\rho^{\rm a}) u^{\rm a}\big) \cdot \nabla \dive R\,\dx\,\dy \\
&- a (\g-1)\int_{\OO}  \big( \rho^{\g} \dive u - (\rho^{\rm a})^{\g} \dive  u^{\rm a}\big) \dive R\,\dx\,\dy
 -  \e \int_{\OO} p'(\rho^{\rm a})  [(\eta \frak{w})_{y}] \dive R\,\dx\,\dy.
\end{align*}
While we observe that
\begin{align*}
& - \int_{\OO}  (p(\rho) - p(\rho^{\rm a})) \dive(u\cdot \nabla R)\,\dx\,\dy  = -  \int_{\OO}  \bigl(p(\rho) u- p(\rho^{\rm a}) u^{\rm a}\bigr)\cdot \nabla \dive R \,\dx\,\dy \\
& \qquad\qquad + \int_{\OO}  p(\rho^{\rm a}) R \cdot \nabla \dive R \,\dx\,\dy
 - \int_{\OO}  (p(\rho) - p(\rho^{\rm a})) (\nabla u^{\rm a}+\nabla R): \nabla R\,\dx\,\dy.
\end{align*}
As a result, it comes out
\ba\label{DtR-4-5}
& \int_{\OO} \nabla (p(\rho) - p(\rho^{\rm a})) \cdot \fD_{t} R \,\dx\,\dy
= - \frac{\rm d}{\dt} \int_{\OO}  (p(\rho) - p(\rho^{\rm a})) \dive R\,\dx\,\dy\\
&\quad-   a (\g-1)\int_{\OO}  \big( \rho^{\g} \dive u - (\rho^{\rm a})^{\g} \dive u^{\rm a}\big) \dive R\,\dx\,\dy
 -  \e \int_{\OO} p'(\rho^{\rm a})   [(\eta \frak{w})_{y}]_\e \dive R\,\dx\,\dy\\
&\quad + \int_{\OO}  p(\rho^{\rm a}) R \cdot \nabla \dive R\,\dx\,\dy  - \int_{\OO}  (p(\rho) - p(\rho^{\rm a})) (\nabla u^{\rm a}+\nabla R): \nabla R\,\dx\,\dy.
\ea
Observe that
\begin{align*}
&\int_{\OO} \big( \rho^{\g} \dive u - (\rho^{\rm a})^{\g} \dive u^{\rm a}\big) \dive R \,\dx\,\dy\\
 &  = \int_{\OO}  \big( \rho^{\g} \dive R + (\rho^{\g}- (\rho^{\rm a})^{\g})\dive u^{\rm a} \big) \dive R \,\dx\,\dy\\
& \leq \bar\rho^{\g}\int_{\OO} |\dive R |^{2}\,\dx\,\dy  + \|\dive u^{\rm a}\|_{L^{\infty}(\OO)}
\int_{\OO}\bigl(  (\rho^{\g}- (\rho^{\rm a})^{\g})^{2} +  |\dive R|^{2}\bigr)\,\dx\,\dy,
\end{align*}
and
\ba
 \e  \int_{\OO} p'(\rho^{\rm a}) [(\eta \frak{w})_{y}]_\e \dive R\,\dx\,\dy
   \leq  \e^{\frac 12} \|p'(\rho^{\rm a})\|_{L^{\infty}(\OO)} \| (\eta \frak{w}_{y} )\|_{L^{2}(\OO)} \|\dive R\|_{L^{2}(\OO)},
\nn \ea
and
\begin{align*}
& \int_{\OO}  p(\rho^{\rm a}) R \cdot \nabla \dive R\,\dx\,\dy  - \int_{\OO}  (p(\rho) - p(\rho^{\rm a})) (\nabla u^{\rm a}+\nabla R): \nabla R\,\dx\,\dy \\
%& = - \int_{\OO}  (\nabla p(\rho^{\rm a}) ) \cdot R \cdot \nabla R - \int_{\OO}   p(\rho^{\rm a}) \nabla R :\nabla R -  \int_{\OO}  (p(\rho) - %p(\rho^{\rm a})) (\nabla W+\nabla R): \nabla R \\
& =- \int_{\OO} \Bigl( (\nabla p(\rho^{\rm a}) ) \cdot R \cdot \nabla R  +  (p(\rho) - p(\rho^{\rm a})) \nabla u^{\rm a}: \nabla R
+  p(\rho) \nabla R: \nabla R\Bigr)\,\dx\,\dy   \\
& \leq  \|\nabla p(\rho^{\rm a})\|_{L^\infty} \| R\|_{L^2}^2  + \|\nabla u^{\rm a}\|_{L^\infty}
  \| (\rho-\rho^{\rm a})\|_{L^2}^{2} +  C \|\nabla R\|_{L^2}^{2}.
\end{align*}
By substituting the above estimates into \eqref{DtR-4-5} and using Theorem \ref{thm1}, we find
\ba\label{DtR-4-5a}
 \int_{\OO} \nabla (p(\rho) - p(\rho^{\rm a})) \cdot \fD_{t} R \,\dx\,\dy
\leq & - \frac{\rm d}{\dt} \int_{\OO}  (p(\rho) - p(\rho^{\rm a})) \dive R\,\dx\,\dy\\
&+Ce^{-\alpha t}\bigl(\|\vr\|_{L^2}^2+\| R\|_{L^2}^2 +\|\nabla R\|_{L^2}^2\bigr)+ C \|\nabla R\|_{L^2}^{2}.
\ea

We further observe that
  \begin{align*}
\int_{\OO} \rho R\cdot \nabla u^{\rm a}  \cdot \fD_{t} R\,\dx\,\dy &
\leq 2 \|\nabla u^{\rm a} (t)\|_{L^{\infty}}^{2}\int_{\OO}  \rho |R|^{2}\,\dx\,\dy + \frac{1}{12}\int_{\OO} \rho (\fD_{t} R)^{2}\,\dx\,\dy,\\
\int_{\OO} \vr \bigl(\d_{t} u^{\rm a} + u^{\rm a} \cdot \nabla u^{\rm a}\bigr)   \cdot \fD_{t} R\,\dx\,\dy
& \leq 2 \bigl\| \bigl(\d_{t} u^{\rm a} + u^{\rm a} \cdot \nabla u^{\rm a}\bigr)  \bigr\|_{L^{\infty}}^{2}
\int_{\OO}  \rho^{-1} \vr^{2}\,\dx\,\dy + \frac{1}{12}\int_{\OO} \rho (\fD_{t} R)^{2}\,\dx\,\dy,\\
\int_{\OO}  G \cdot \fD_{t}R\,\dx\,\dy  &  \leq  2\int_{\OO} \rho^{-1} |G|^{2}\,\dx\,\dy
+  \frac{1}{12}\int_{\OO} \rho (\fD_{t} R)^{2}\,\dx\,\dy.
 \end{align*}
This together with \eqref{def-G12-est} and
\eqref{energy-bas}  that
\ba\label{DtR-6-1}
\int_{\OO}\Bigl(\rho R\cdot \nabla u^{\rm a} +  \vr \bigl(\d_{t}u^{\rm a}  + u^{\rm a} \cdot \nabla u^{\rm a}\bigr)-   G\Bigr)
\cdot \fD_tR \,\dx\,\dy\leq C e^{-\a t} \e + \frac{1}{4}\int_{\OO} \rho (\fD_{t} R)^{2}\,\dx\,\dy.
 \ea

By inserting the estimates \eqref{DtR-3-1}, \eqref{DtR-3-2}, \eqref{DtR-4-5a} and \eqref{DtR-6-1} into \eqref{DtR-1}
and
  and integrating the resulting inequality over $[0,t]$, we achieve
 \ba\label{DtR-7-1}
&\int_{\OO} \bigl(\mu |\nabla R|^{2} + \mu'| \dive R|^{2}\bigr) \,\dx\,\dy + \int_{0}^{t} \int_{\OO} \rho |\fD_{t}R|^{2} \,\dx\,\dy\,\dt'\\
& \leq   C \e + \int_{\OO}  (p(\rho) - p(\rho^{\rm a})) \dive R \,\dx\,\dy+ 4C\int_{0}^{t}\int_{\OO} |\nabla R|^{3} \,\dx\,\dy\,\dt' + C \int_{0}^{t}\int_{\OO}  |\nabla R|^{2}\,\dx\,\dy\,\dt'.
\ea
Notice  that
 \begin{align*}
 \int_{\OO}  (p(\rho) - p(\rho^{\rm a})) \dive R\,\dx\,\dy & \leq  \mu^{-1}\int_{\OO}  |p(\rho) - p(\rho^{\rm a})|^{2}\,\dx\,\dy + \frac{\mu}{4} \int_{\OO} |\nabla R|^{2}\,\dx\,\dy \\
 & \leq C\mu^{-1}\|\vr\|_{L^2}^{2}  + \frac{\mu}{4} \|\nabla R\|_{L^2}^{2},
 \end{align*}
from which and \eqref{energy-bas}, we deduce \eqref{DtR-0} from \eqref{DtR-7-1}.
This completes the proof of the lemma.
\end{proof}

\subsection{Estimates of $\fD_{t} R$}

\begin{lem}\label{prop-DtR-new}
{\sl For each $t < T^\star$, one has
\ba\label{DtR-0-new}
 \int_{\OO}  |D_{t}R|^{2}\,\dx\,\dy + \int_{0}^{t} \int_{\OO} |\nabla \fD_{t} R|^{2}\,\dx\,\dy\,\dt' \leq C \e
   + C \int_{0}^{t} \int_{\OO}|\nabla R|^{4}\,\dx\,\dy\,\dt'.
  \ea}
\end{lem}

\begin{proof}
By applying $\fD_{t}$ to the $R$ equation of  \eqref{CNS-error-new} and taking $L^2$ inner product of the resulting
equation with $\fD_tR,$ we find
\ba\label{DtR-1-new}
 \int_{\OO}  \Bigl(&\fD_{t} (\rho \fD_{t} R) - \mu \fD_{t} \Delta R - \mu' \fD_{t}  \nabla\dive R +  \fD_{t} \nabla \big(p(\rho)  - p(\rho^{\rm a})  \big) \\
&+  \fD_{t} \big(\rho R\cdot \nabla u^{\rm a} \big)  + \fD_{t} \big( \vr (\d_{t} u^{\rm a} + u^{\rm a} \cdot \nabla u^{\rm a})\big)\Bigr)
\cdot \fD_{t} R\,\dx\,\dy  = \int_{\OO}  \fD_{t} G\cdot \fD_{t} R\,\dx\,\dy.
\ea
Let us now handle term by term above. Notice that
\ba
\fD_{t} (\rho D_{t} R)  =\rho \fD_{t}^{2} R +  (\fD_{t} \rho) \fD_{t}R.
\nn \ea
Firstly,
\ba\label{DtR-2-2-new}
\int_{\OO} \rho \fD_{t}^{2} R  \cdot \fD_{t} R \,\dx\,\dy=  \frac{1}{2}\frac{\rm d}{\dt}\int_{\OO} \rho |\fD_{t}R|^{2}\,\dx\,\dy.
\ea
By virtue of \eqref{CNS}, we write
\ba\label{DtR-2-3-new}
\fD_{t} \rho \fD_{t}R    & = \dive u (-\rho \fD_{t}R) \\
& = \dive u \big( - \mu \Delta R- \mu' \nabla\dive R + \nabla \big(p(\rho)  - p(\rho^{\rm a})  \big) + \rho R \cdot \nabla u^{\rm a}  + \vr (\d_{t} u^{\rm a} + u^{\rm a} \cdot \nabla u^{\rm a}) -  G \big).
\ea
We shall postpone the estimate of the above terms below.

For $i=1,2,$ we write
\begin{align*}
 -\mu \fD_{t}\Delta R_{i}
 & = - \mu \Delta \d_{t}  R_{i} - \mu \dive\big( (u \cdot \nabla) \nabla R_{i}  \big) +  \mu \nabla u_{j} \cdot \d_{j}\nabla R_{i}  \\
 & =  - \mu \Delta \fD_{t}  R_{i} + \mu \dive\big(   \nabla u_{j} \cdot \d_{j} R_{i} \big) +  \mu \nabla u_{j} \cdot \d_{j}\nabla R_{i}.
  \end{align*}
 The last term above can be written as
\begin{align*}
 \mu \nabla u_{j} \cdot \d_{j}\nabla R_{i}  = \mu \d_{k} u_{j} \d_{j} \d_{k} R_{i}
 &= \mu \d_{j} \big( \d_{k} u_{j} \d_{k} R_{i}\big) - \mu \d_{k}\d_{j} u_{j}  \d_{k} R_{i}  \\
 & = \mu \d_{j} \big( \d_{k} u_{j} \d_{k} R_{i}\big) - \mu \d_{k}\big(\d_{j} u_{j}  \d_{k} R_{i}\big) + \mu   \dive u \Delta R_{i}.
 \end{align*}
 We observe that the last term above cancels with the first term on the right-side of  \eqref{DtR-2-3-new},
 so that there holds
 \begin{align*}
 & -\int_{\OO} \big( \dive u  \Delta R + \fD_{t}\Delta R\big) \cdot \fD_{t} R\,\dx\,\dy\\
  & =  - \int_{\OO}\Delta D_{t} R \cdot D_{t} R\,\dx\,\dy
   +\int_{\OO} \Big(  \d_k\big(   \d_{k} u_{j} \d_{j} R \big)  + \d_{j} \big( \d_{k} u_{j} \d_{k} R\big) -  \d_{k}\big(\d_{j} u_{j}  \d_{k} R\big) \Big)\cdot D_{t} R_i\,\dx\,\dy.
 \end{align*}
We get, by using integration by parts, that
 \ba
-\int_{\OO} \Delta \fD_{t} R \cdot \fD_{t} R \,\dx\,\dy = \int_{\OO}  |\nabla \fD_{t} R |^{2}\,\dx\,\dy,
 \nn \ea
and
  \begin{align*}
&  \mu\int_{\OO} \Big(  \d_k\big(   \d_{k} u_{j} \d_{j} R \big)  + \d_{j} \big( \d_{k} u_{j} \d_{k} R\big) -  \d_{k}\big(\d_{j} u_{j}  \d_{k} R\big) \Big)\cdot \fD_{t} R\,\dx\,\dy  \\
& \leq  12\mu \int_{\OO}  | \nabla u |^{2} |\nabla R|^{2}\,\dx\,\dy +  \frac{\mu}{8} \int_{\OO} |\nabla \fD_{t} R |^{2}\,\dx\,\dy \\
& \leq  12\mu \int_{\OO}  \bigl(\| \nabla u^{\rm a} \|_{L^\infty}^{2} + |\nabla R|^{2}\bigr) |\nabla R|^{2}\,\dx\,\dy
 +  \frac{\mu}{8} \|\nabla \fD_{t} R \|_{L^2}^{2}.
 \end{align*}
 Thus, we get, by applying Theorem \ref{thm1}, that
  \ba\label{DtR-3-6-new}
  -\mu\int_{\OO} \big( \dive u  \Delta R &+ \fD_{t}\Delta R\big) \cdot \fD_{t} R\,\dx\,\dy\\
&\geq \frac{7 }{8}\mu \|\nabla \fD_{t} R \|_{L^2}^{2}  - C e^{-\a t} \|\nabla R\|_{L^2}^{2} -  12\mu \int_{\OO} |\nabla R|^{4}\,\dx\,\dy.
 \ea
Similar argument leads to
\ba\label{DtR-4-1-new}
 - \mu'\int_{\OO} \big(\dive u (  \nabla \dive R) &+ \fD_{t}\nabla \dive R) \cdot \fD_{t} R\,\dx\,\dy\\
 & \geq \frac{7}{8} \mu' \|\dive \fD_{t} R \|_{L^2}^{2}  - C e^{-\a t} \|\nabla R\|_{L^2}^{2} -  12\mu' \int_{\OO} |\nabla R|^{4}\,\dx\,\dy.
\ea

To handle the pressure related terms, we write
\begin{align*}
& ( \dive u) \nabla(p(\rho) - p(\rho^{\rm a})) + (u \cdot \nabla )\nabla(p(\rho) - p(\rho^{\rm a})) \\
& = \d_{j} \big(u_{j} \nabla  (p(\rho) - p(\rho^{\rm a}))\big) = \nabla  \dive \big(u (p(\rho) - p(\rho^{\rm a}))\big) - \d_{j} \big( \nabla u_{j}  (p(\rho) - p(\rho^{\rm a}))\big).
 \end{align*}
In view of the continuity equations of \eqref{CNS} and \eqref{CNS-1d-W}, one has
 \begin{align*}
\d_{t}(p(\rho) - p(\rho^{\rm a}))
 =& - \dive\big( u (p(\rho)- p(\rho^{\rm a})) + p(\rho^{\rm a}) R \big)\\
  &- a (\g-1)  \big( \rho^{\g} \dive u - (\rho^{\rm a})^{\g} \dive u^{\rm a}\big) -\e p'(\rho^{\rm a})  [(\eta \frak{w})_{y}]_\e.
\end{align*}
As a consequence, we deduce that
\begin{align*}
( \dive u) \nabla(p(\rho) - p(\rho^{\rm a})) + \fD_{t} \nabla \big(p(\rho) & - p(\rho^{\rm a})  \big)
= - \d_{j} \big( \nabla u_{j}  (p(\rho) - p(\rho^{\rm a}))\big)   - \nabla \dive (p(\rho^{\rm a}) R)\\
 &- a  (\g-1)  \nabla\big( \rho^{\g} \dive u - (\rho^{\rm a})^{\g} \dive u^{\rm a}\big)  -  \e \nabla \big( p'(\rho^{\rm a})[(\eta \frak{w})_{y}]_\e \big),
 \end{align*}
 from which and Theorem \ref{thm1}, we deduce that
 \ba\label{DtR-5-5-new}
&\int_{\OO} \Big(( \dive u) \nabla(p(\rho) - p(\rho^{\rm a}))  + \fD_{t} \nabla \big(p(\rho)  - p(\rho^{\rm a})  \big)\Big) \cdot \fD_{t}R \,\dx\,\dy \\
& = \int_{\OO}  (p(\rho) - p(\rho^{\rm a}))(\nabla u^{\rm a} + \nabla R) : \nabla \fD_{t} R \,\dx\,\dy
+ \int_{\OO}  \Bigl((p(\rho^{\rm a}) \dive R + \nabla p(\rho^{\rm a}) \cdot R)  \\
& \qquad  +a (\g-1) \big( \rho^{\g} \dive R  + (\rho^{\g}- (\rho^{\rm a})^{\g} ) \dive u^{\rm a}\big)
+  \e p'(\rho^{\rm a}) [(\eta \frak{w})_{y}]\Bigr) \dive \fD_{t}R\,\dx\,\dy  \\
& \leq C e^{-\a t} \bigl(\e+\|\vr|_{L^2}^{2} + \|R\|_{L^2}^{2}\bigr)   + C \|\nabla R\|_{L^2}^{2}  + \frac{\mu}{2} \|\nabla \fD_{t} R \|_{L^2}^{2}.
 \ea

On the other hand, thanks to the continuity equation of \eqref{CNS-limit}, we infer
 \begin{align*}
&  \int_{\OO} \big( \fD_{t} (\rho R \cdot \nabla u^{\rm a})  + (\dive u) \rho R\cdot \nabla u^{\rm a}\big) \cdot \fD_{t} R\,\dx\,\dy  \\
& =  \int_{\OO} \big(  \rho (\fD_{t} R) \cdot \nabla u^{\rm a}+ \rho R \cdot D_{t} \nabla u^{\rm a}\big) \cdot \fD_{t} R\,\dx\,\dy  \\
& \leq \bigl(\| \nabla u^{\rm a}\|_{L^{\infty}}    +  \| \d_{t} \nabla u^{\rm a} \|_{L^{\infty}}
+  \|  u^{\rm a} \nabla^{2} u^{\rm a}\|_{L^{\infty}}+ \| \nabla^{2} u^{\rm a} \|_{L^{\infty}} \bigr) \\
  &\qquad\times\int_{\OO}   \rho \left(  |\fD_{t} R|^{2} +   |R|^{2}\right)\,\dx\,\dy + \| \nabla^{2} u^{\rm a} \|_{L^{\infty}}  \int_{\OO}   \rho |R|^{4}\,\dx\,\dy.
 \end{align*}
It follows from  the Gagliardo-Nirenberg interpolation inequality and \eqref{energy-bas} that
\ba
\| R \|_{L^{4}(\OO)}^{4} \leq C \| R \|_{L^{2}(\OO)}^{2} \| R\|_{H^{1}(\OO)}^{2} \leq C\e^{2} + C \e  \|\nabla R\|_{L^{2}(\OO)}^{2}.
\nn \ea
As a result, we get, by applying Theorem \ref{thm1}, that
 \ba\label{DtR-6-4-new}
  \int_{\OO} \big( \fD_{t} (\rho R \cdot \nabla u^{\rm a})  &+ (\dive u) \rho R\cdot \nabla u^{\rm a}\big) \cdot \fD_{t} R\,\dx\,\dy\\
  &\leq C  e^{-\a t}\Bigl(\e \bigl(1+ \|\nabla R\|_{L^2}^{2}\bigr) +  \int_{\OO}   \rho |D_{t} R|^{2}\,\dx\,\dy\Bigr).
 \ea

In view of the $\vr$ equation of \eqref{CNS-error-new}, we write
 \begin{align*}
 & \fD_{t} \big( \vr \bigl(\d_{t} u^{\rm a} + u^{\rm a} \cdot \nabla u^{\rm a})\bigr)
 + \dive u \big(\vr \bigl(\d_{t} u^{\rm a} + u^{\rm a} \cdot \nabla u^{\rm a}) \bigr) \\
&= \bigl(\fD_{t}\vr + \vr \dive u\bigr) \bigl(\d_{t} u^{\rm a} + u^{\rm a} \cdot \nabla u^{\rm a})\bigr)
 + \vr \fD_{t} \bigl(\d_{t} u^{\rm a} + u^{\rm a} \cdot \nabla u^{\rm a})\bigr)  \\
 & = - \big( \dive(\rho^{\rm a} R) + \e [(\eta \frak{w})_{y}]_\e\big) \bigl(\d_{t} u^{\rm a} + u^{\rm a} \cdot \nabla u^{\rm a})\bigr)
   + \vr \fD_{t} \bigl(\d_{t} u^{\rm a} + u^{\rm a} \cdot \nabla u^{\rm a})\bigr) .
  \end{align*}
 It is easy to observe that
     \begin{align*}
 &\int_{\OO} \big( \dive(\rho^{\rm a} R) + \e [(\eta \frak{w})_{y}]_\e\big)\bigl(\d_{t} u^{\rm a} + u^{\rm a} \cdot \nabla u^{\rm a})\bigr) \cdot \fD_{t} R \,\dx\,\dy\\
 &\leq \bigl\|\bigl(\d_{t} u^{\rm a} + u^{\rm a} \cdot \nabla u^{\rm a})\bigr)\bigr\|_{L^{\infty}}
 \bigl\| \rho^{\rm a} \dive R + R\cdot \nabla \rho^{\rm a} + \e [(\eta \frak{w})_{y}]_\e\bigr\|_{L^{2}}\| \fD_{t} R\|_{L^{2}}\\
 & \leq C e^{-\a t} \bigl(\|R\|_{L^{2}} + \|\nabla R\|_{L^{2}} + \e^{\frac 12}\bigr)\| \fD_{t} R\|_{L^{2}} ,
  \end{align*}
  and
 \begin{align*}
     & \int_{\OO} \vr \d_{t}\bigl(\d_{t} u^{\rm a} + u^{\rm a} \cdot \nabla u^{\rm a})\bigr)\cdot \nabla u^{\rm a})\bigr) \cdot \fD_{t} R \,\dx\,\dy\leq { \bigl\|\d_{t}\bigl(\d_{t} u^{\rm a} + u^{\rm a} \cdot \nabla u^{\rm a})\bigr)\bigr\|_{L^{\infty}}}  \| \vr\|_{L^{2}}\| \fD_{t} R\|_{L^{2}},
   \end{align*}
  and
    \begin{align*}
& \int_{\OO}    \vr u \cdot \nabla \bigl(\d_{t} u^{\rm a} + u^{\rm a} \cdot \nabla u^{\rm a})\bigr) \cdot \fD_{t} R\,\dx\,\dy \\
 &\leq \| \fD_{t} R\|_{L^{2}}\big( \bigl\|u^{\rm a}\cdot \nabla \bigl(\d_{t} u^{\rm a} + u^{\rm a} \cdot \nabla u^{\rm a})\bigr) \bigr\|_{L^{\infty}}  \| \vr\|_{L^{2}}+ \|\vr\|_{L^{\infty}} \bigl\|\bigl(\d_{t} u^{\rm a} + u^{\rm a} \cdot \nabla u^{\rm a})\bigr)\bigr\|_{L^{\infty}}
 \| R \|_{L^{2}}\big)\\
 &\leq C e^{-\a t} \e^{\frac 12} \| D_{t} R\|_{L^{2}(\OO)}.
 \end{align*}
Therefore, we obtain
  \ba\label{DtR-9-5-new}
 \int_{\OO}
\Bigl(\fD_{t} \big( \vr \bigl(\d_{t} u^{\rm a} + u^{\rm a} \cdot \nabla u^{\rm a})\bigr)
 + \dive u \big(\vr \bigl(\d_{t} u^{\rm a} &+ u^{\rm a} \cdot \nabla u^{\rm a}) \bigr)\Bigr)\cdot \fD_{t} R \,\dx\,\dy\\
 &\leq   C e^{-\a t} \e + C e^{-\a t}\| \fD_{t} R\|^{2}_{L^{2}(\OO)}.
\ea
We remark that it is exactly the term ${ \bigl\|\d_{t}\bigl(\d_{t} u^{\rm a} + u^{\rm a} \cdot \nabla u^{\rm a})\bigr)\bigr\|_{L^{\infty}}}$  that needs the highest regularity of the approximate solution. It follows from  \eqref{thm1-0} that ${\bigl\|\d_{t}\bigl(\d_{t} u^{\rm a} + u^{\rm a} \cdot \nabla u^{\rm a})\bigr)\bigr\|_{L^{\infty}}} \leq C. $

 Finally, thanks to \eqref{def-G12-est},   we compute
\ba\label{DtR-7-2-new}
  \int_{\OO} &\big( \d_{t} G + \dive (u G) \big) \cdot \fD_{t}R\,\dx\,\dy\\
   &\leq \|\d_{t} G\|_{L^{2}} \| \fD_{t} R\|_{L^{2}} + \| G\|_{L^{\infty}} (\|u^{\rm a}\|_{L^{2}} + \|R\|_{L^{2}})\| \nabla \fD_{t} R\|_{L^{2}} \\
  & \leq C \e^{\frac{1}{2}} e^{-\a t}\| \fD_{t} R\|_{L^{2}} + C \e^{\frac 12} e^{-\a t} \| \nabla \fD_{t} R\|_{L^{2}} \\
  & \leq C \e e^{-\a t} + C e^{-\a t}\| \fD_{t} R\|_{L^{2}}^{2} + \frac{1}{8}\| \nabla \fD_{t} R\|_{L^{2}}^2.
\ea

By inserting the estimates \eqref{DtR-2-2-new} and  (\ref{DtR-3-6-new}--\ref{DtR-7-2-new}) into \eqref{DtR-1-new} and then
integrating the resulting inequality over $[0,t],$ we arrive at
\ba\label{DtR-8-1-new}
 \|\fD_{t}R\|_{L^2}^{2} + \int_{0}^{t} \|\nabla \fD_{t} R(t')\|^{2}\,\dt' \leq C \e +
  C \int_{0}^{t} |\nabla R|^{4}\,\dx\,\dy
 + C \int_{0}^{t} e^{-\a t'} \|D_{t}R(t')\|^{2}\,\dt'.
\ea
Applying Gronwall's inequality leads to \eqref{DtR-0-new}.
This finishes the proof of Lemma \ref{prop-DtR-new}.
\end{proof}

\subsection{Estimates for the vorticity and proof of Proposition \ref{energy}} \label{Sect7.4}

For each vector valued function $v = (v_{1},v_{2})^{\rm T}:\R^{2} \to \R^{2}$, we define its vorticity $\curl v$ as follows:
$$
\curl v \eqdefa \nabla^{\perp} \cdot v  =  \d_{y} v_{1} - \d_{x} v_{2}, \quad \nabla^{\perp} \eqdefa \bp \d_{y} \\ - \d_{x}\ep.
$$
  In particular, we denote $\omega \eqdefa \curl R$.
  Applying $\curl$ to $\eqref{CNS-error-new}_{2}$ gives
\ba\label{vor-1}
\curl (\rho \fD_{t} R ) - \mu \Delta \o + \curl \big( \rho R\cdot  \nabla u^{\rm a} \big) + \curl \big(\vr (\d_{t} u^{\rm a}
 + u^{\rm a}  \cdot \nabla u^{\rm a} ) \big) = \curl G.
\ea
%\ba\label{vor-1}
%\rho D_{t} \curl R - \mu \Delta \curl R + \big( \d_{2} \rho D_{t} R^{1} - \d_{1} \rho D_{t} R^{2} \big) + \big( \d_{2} u\cdot \nabla R^{1} -\d_{1} u\cdot \nabla R^{2}  \big) + \d_{2} (\rho R \cdot \nabla W) = \curl G.
%\ea

\begin{lem}\label{prop-vorticity}
{\sl Let $\o\eqdefa \curl R$. Then there holds
\ba\label{vor-4}
\int_{\OO} |\nabla \o|^{2}\,\dx\,\dy    \leq C \e + C\int_{0}^{t}\int_{\OO} |\nabla R|^{4}\,\dx\,\dy,
\ea
and
\ba\label{vor-5}
\int_{0}^{t} \int_{\OO} |\nabla \o|^{2}\,\dx\,\dy\,\dt'    \leq C \e + C \int_{0}^{t} \int_{\OO} |\nabla R|^{3}\,\dx\,\dy\,\dt'.
 \ea}
\end{lem}
\begin{proof} We first get, by taking $L^2$ inner product of \eqref{vor-1} with $\o$ and using integration by parts, that
\begin{align*}
 \mu \int_{\OO} |\nabla \o|^{2} \,\dx\,\dy=\int_{\OO}\Bigl(&- G \cdot \nabla^{\perp} \o  + \rho \fD_{t} R \cdot \nabla^{\perp} \o
 \\
    & + \rho R\cdot  \nabla u^{\rm a} \cdot \nabla^{\perp} \o+ \vr \bigl(\d_{t} u^{\rm a} + u^{\rm a} \cdot \nabla u^{\rm a}\bigr)  \cdot \nabla^{\perp} \o\Bigr)\,\dx\,\dy.
  \end{align*}
Applying Young's inequality gives
\begin{align*}
\mu \int_{\OO} |\nabla \o|^{2} \,\dx\,\dy   \leq \frac{4}{\mu} \int_{\OO} \Big(& |G|^{2} +  \rho^{2} |\fD_{t} R|^{2} +  |\nabla u^{\rm a}| ^{2} \rho^{2} | R|^{2} \\
&+ \vr^{2} \bigl| \d_{t} u^{\rm a} + u^{\rm a} \cdot \nabla u^{\rm a}\bigr|^{2}  \Big)\,\dx\,\dy
 + \frac{\mu}{2} \int_{\OO} |\nabla \o|^{2}\,\dx\,\dy,
\end{align*}
 which together with Proposition  \ref{prop-energy-basic},  Lemmas \ref{prop-DtR} and  \ref{prop-DtR-new}, and Theorem
 \ref{thm1} ensures \eqref{vor-4} and \eqref{vor-5}.
\end{proof}

By summarizing Proposition \ref{prop-energy-basic}, Lemmas \ref{prop-DtR}, \ref{prop-DtR-new} and \ref{prop-vorticity}, we conclude the proof of
Proposition \ref{energy}.

 \appendix

 \section{Proof of Proposition \ref{S2prop1}}\label{appa}

 \begin{proof}[Proof of Proposition \ref{S2prop1}]
By summarizing Propositions \ref{prop-upperbd-density}, \ref{prop-decay-exp-L2}, \ref{prop-decay-exp-H1}, \ref{prop-decay-exp-rx}
and \ref{prop-Dtv-L2}, we conclude that \eqref{thm1-2} holds and  there exist two positive constants $C$ and $\a$ solely depending
 on $(a, \g, \nu, \bar \varsigma_0, \underline{\varsigma}_0, \bar E_{10})$ such that
\ba\label{decay-exp-all-1}
\|(\eta -1, w)(t)\|_{H^1_\h} + \| (\eta -1, w)(t)\|_{L^\infty_\h} + \|D_t  w(t)\|_{L^2_\h} \leq C E_{20}^{\frac{1}{2}}(y) e^{-\a t}, \quad \forall \, t\in \R^{+}, \ y\in \R,
\ea from which and \eqref{CNS-limit}, we infer
\ba\label{decay-exp-all-2}
 \| \eta _t(t)\|_{L^2_\h} +  \|(w _{xx}, w _t)(t)\|_{L^2_\h}  \leq C E_{20}^{\frac{1}{2}}(y) e^{-\a t}, \quad \forall \, t\in \R^{+}, \ y\in \R.
\ea
By virtue of Sobolev embedding, we obtain
\ba\label{decay-exp-all-3}
 \| \eta(t)\|_{C^{0,\frac 12}_\h} +  \|w(t)\|_{C^{1, \frac 12}_\h}  \leq C E_{20}^{\frac{1}{2}}(y) e^{-\a t}, \quad \forall \, t\in \R^{+}, \ y\in \R.
\ea

To derive decay estimates for higher order norms, we are going to use induction method. Let $n \in \Z_{+}, \ 1\leq n\leq 3,$  we denote
\ba\label{decay-exp-all-m00}
& E_{m+2}(t,y) \eqdefa \|(\eta -1)(t)\|_{H^{m+1}_\h}^{2} + \| w (t)\|_{H^{m+2}_\h}^{2} +  \| (\d_{t}^{m}\eta, \d_{t}^{m}w) (t)\|_{L^2_\h}^{2},\\
& D_{m+2} (t,y) \eqdefa \|(\eta -1)(t)\|_{H^{m+1}_\h}^{2} + \|w (t)\|_{H^{m+2}_\h}^{2}  + \|D_t  w(t)\|_{H^{m+1}_\h}^{2}.
\ea
and assume that for  $0 \leq m \leq n-1,$  there exists $\cF_{m}(t,y)$ so that
\begin{subequations} \label{Saeq11}
\begin{gather}
\label{cfm}
 \cF_{m}(t,y)\sim \bigl( \|(\eta -1, w)(t)\|_{H^{m+1}_\h}^{2} +\|D_t  w(t)\|_{H^m_\h}^{2} \bigr),\\
\label{decay-exp-all-m0}
\frac{\rm d}{\dt} \cF_m(t,y)+ \de  D_{m+2} (t,y) \leq C E_{m+2,0}(y) e^{-\a t},\\
\label{decay-exp-all-m}
E_{m+2}^{\frac12}(t,y) \leq C E_{m+2,0}^{\frac{1}{2}}(y) e^{-\a t},
\end{gather}
\end{subequations}
for some $\a>0, \ \de>0$ and all $t\in \R^{+}, \ y\in \R$, where $E_{m+2,0}(y) \eqdefa E_{m+2}(0,y)$.

\medskip

As shown in Section \ref{sec:1dNS}, the estimates in \eqref{decay-exp-all-m0} and \eqref{decay-exp-all-m} hold when $m=0$. We would like to prove that \eqref{decay-exp-all-m0} and \eqref{decay-exp-all-m} hold  for $m=n$. In order to it, we denote
\be\label{def-eta-w-n}
\eta^{(n)}\eqdefa \d_{x}^{n} \eta \andf  w^{(n)} \eqdefa \d_{x}^{n} w.
\ee
Then it follows from   \eqref{CNS-limit} that $(\eta^{(n)}, w^{(n)})$ verifies
\be\label{CNS-1d-n}
\left\{\begin{aligned}
& \eta^{(n)}_{t}  + w \eta^{(n)}_{x} + \eta w^{(n)}_{x}  = - f_{n},\\
& \eta(w^{(n)}_{t} + w w^{(n)}_{x}) - \nu w^{(n)}_{xx} + p'(\eta) \eta^{(n)}_{x} = - g_{n},
\end{aligned}\right.
\ee
where
\begin{align*}\label{CNS-1d-n-1}
& f_{n}\eqdefa \sum_{m=1}^{n} w^{(m)}\eta^{(n-m+1)}  + \sum_{m=1}^{n} \eta^{(m)} w^{(n-m+1)},\\
& g_{n} \eqdefa  \sum_{m=1}^{n} \eta^{(m)} w^{(n-m)}_{t}  + \sum_{m=1}^{n} \eta^{(m)} \d_{x}^{n-m} (w w_{x})  + \eta \sum_{m=1}^{n} w^{(m)} w^{(n-m+1)} + \sum_{m=1}^{n} \d_{x}^{m}(p'(\eta)) \eta^{(n-m+1)},
\end{align*}
from which, the inductive assumption \eqref{decay-exp-all-m} and Sobolev embedding theorem, we infer
\ba\label{CNS-1d-n-2}
\|(f_{n}, g_{n})(t)\|_{L^{2}_{h}} \leq C E_{n+2,0}^{\frac{1}{2}}(y) e^{-\a t}.
\ea

In what follows, we shall decompose the proof into the following steps:

\medskip

\noindent{\bf Step 1.} {Decay estimates for $\|\eta^{(n+1)}\|_{L^{2}_{h}}.$}

\medskip

Recalling $\zeta = \eta^{-1}$ and \eqref{decay-exp-rx-2}, we find
\be\label{decay-eta-n+1}
 \zeta w^{(n)}_{xx} = D_{t} \zeta^{(n+1)} +  \sum_{m=1}^{n} w^{(m)} \zeta^{(n-m+2)} - \sum_{m=1}^{n} \zeta^{(m)} w^{(n-m+2)}
  \with \zeta^{(n)} \eqdefa \d_{x}^{n}\zeta.
 \ee
While
we observe from $\eqref{CNS-1d-n}_{2}$ that
\ba\label{CNS-1d-n-3}
D_{t} w^{(n)} - \nu \eta^{-1} w^{(n)}_{xx} + \eta^{-1} p'(\eta) \eta^{(n)}_{x} = - \eta^{-1} g_{n}.
\nn \ea
By plugging \eqref{decay-eta-n+1} into the above equality  yields
\be\label{CNS-1d-n-4}
D_{t} (w^{(n)} -  \nu  \zeta^{(n+1)})  - \nu \Big(  \sum_{m=1}^{n} w^{(m)} \zeta^{(n-m+2)} - \sum_{m=1}^{n} \zeta^{(m)} w^{(n-m+2)}   \Big) + \eta^{-1} p'(\eta) \eta^{(n)}_{x} = - \eta^{-1} g_{n}.
\ee
By taking $L^2_\h$ inner product of  $\eta (w^{(n)} -  \nu  \zeta^{(n+1)}) $ with \eqref{CNS-1d-n-4}, we find
\begin{align*}
&\frac 12\frac{\rm d}{\dt} \int_{\TT}  \eta (w^{(n)} -  \nu  \zeta^{(n+1)})^{2}\,\dx   +  \int_{\TT} p'(\eta) \eta^{(n)}_{x} (w^{(n)} -  \nu  \zeta^{(n+1)})\,\dx\\
& = -\int_{\TT}    w_{x}   \eta \bigl(w^{(n)} -  \nu  \zeta^{(n+1)}\bigr)^2\,\dx  + \int_{\TT}    w_{x}  w^{(n)} \eta (w^{(n)} -  \nu  \zeta^{(n+1)})\,\dx \\
&\qquad  +\nu  \int_{\TT}      \Big(  \sum_{m=2}^{n} w^{(m)} \zeta^{(n-m+2)} - \sum_{m=1}^{n} \zeta^{(m)} w^{(n-m+2)}   \Big) \eta (w^{(n)} -  \nu  \zeta^{(n+1)})\,\dx \\
&\qquad- \int_{\TT}   g_{n} (w^{(n)} -  \nu  \zeta^{(n+1)})\,\dx,
\end{align*}
from which, and the induction assumption \eqref{decay-exp-all-m}, we infer
\ba\label{CNS-1d-n-5}
&\frac 12\frac{\rm d}{\dt} \int_{\TT}  \eta (w^{(n)} -  \nu  \zeta^{(n+1)})^{2}\,\dx    +  \int_{\TT} p'(\eta) \eta^{(n)}_{x} (w^{(n)} -  \nu  \zeta^{(n+1)})\,\dx \\
& \leq C \|w_{x}\|_{L^{\infty}_{h}}\int_{\TT} \eta (w^{(n)} -  \nu  \zeta^{(n+1)})^{2}\,\dx
  +  2\de^{-1}\int_{\TT} \eta |w_{x} w^{(n)}|^{2}\,\dx  \\
& \qquad+ 2\nu^{2}\de^{-1} \int_{\TT}    \eta   \Big|  \sum_{m=2}^{n} w^{(m)} \zeta^{(n-m+2)} - \sum_{m=1}^{n} \zeta^{(m)} w^{(n-m+2)}   \Big|^{2}\,\dx \\
& \qquad + 2\de^{-1}\int_{\TT} \eta^{-1}| g_{n} |^{2}\,\dx    + \de \int_{\TT} \eta (w^{(n)} -  \nu  \zeta^{(n+1)})^{2}\,\dx \\
&  \leq  C E_{m+2,0}^{\frac{1}{2}}(y) e^{-\a t} \int_{\TT} \eta (w^{(n)} -  \nu  \zeta^{(n+1)})^{2} \,\dx  + C \de^{-1} E_{n+2,0}(y) e^{-\a t}\\
&\qquad + \de \int_{\TT} \eta (w^{(n)} -  \nu  \zeta^{(n+1)})^{2}\,\dx,
\ea
where $\de$ is a small positive constant to be determined later on.

While we observe that
\ba%\label{CNS-1d-n-7}
\eta_{x}^{(n)} = \d_{x}^{n}(-\zeta^{-2} \zeta_{x}) = -\zeta^{-2} \zeta^{(n+1)}- \sum_{m=1}^{n} \d_{x}^{m}(\zeta^{-2}) \zeta^{(n-m+1)}.
\nn \ea
Then  we have
\ba
& \int_{\TT} p'(\eta) \eta^{(n)}_{x} (w^{(n)} -  \nu  \zeta^{(n+1)})\,\dx = \nu^{-1}
\int_{\TT} p'(\eta) \eta^{2} \bigl(w^{(n)} - \nu \zeta^{(n+1)}\bigr)^2  \,\dx\\
&\ -   \nu^{-1} \int_{\TT} p'(\eta) \eta^{2} w^{(n)}  (w^{(n)} -  \nu  \zeta^{(n+1)})\,\dx
 -  \int_{\TT} p'(\eta) \sum_{m=1}^{n} \d_{x}^{m}(\zeta^{-2}) \zeta^{(n-m+1)} (w^{(n)} -  \nu  \zeta^{(n+1)})\,\dx.
\nn\ea
By inserting the above equality into \eqref{CNS-1d-n-5} and using \eqref{thm1-2}, we deduce that
\ba\label{CNS-1d-n-9}
\frac 12 &\frac{\rm d}{\dt} \int_{\TT} \eta (w^{(n)} -  \nu  \zeta^{(n+1)})^{2}\,\dx
   +  \nu^{-1} p'(\underline \eta) \underline \eta \int_{\TT} \eta (w^{(n)} -  \nu  \zeta^{(n+1)})^{2}\,\dx\\
\leq  & C E_{n+2,0}^{\frac{1}{2}}(y) e^{-\a t} \int_{\TT} \eta (w^{(n)} -  \nu  \zeta^{(n+1)})^{2} \,\dx + C \de^{-1} E_{n+2,0}(y) e^{-\a t} \\
&+ \de \int_{\TT} \eta (w^{(n)} -  \nu  \zeta^{(n+1)})^{2}\,\dx.
\ea
Choosing $\de = \min\{1,\nu^{-1} p'(\underline \eta) \underline \eta /2\}$ in \eqref{CNS-1d-n-9}  and applying Gronwall's inequality gives
\ba\label{CNS-1d-n-10}
\int_{\TT}  \eta (w^{(n)} -  \nu  \zeta^{(n+1)})^{2}\,\dx  \leq  C   E_{n+2, 0}(y)   e^{-\a t},
 \nn \ea
which together with \eqref{decay-exp-all-m} ensures that
 \ba\label{CNS-1d-n-11}
\int_{\TT}  |\zeta^{(n+1)}|^{2} \,\dx \leq C   E_{n+2,0}(y)   e^{-\a t}.
 \nn \ea
This together together with the inductive assumption \eqref{decay-exp-all-m} leads to
  \ba\label{CNS-1d-dx-12}
\int_{\TT}  | \eta^{(n+1)}|^{2}   \leq C   E_{n+2, 0}(y)   e^{-\a t}.
 \ea
And then we deduce from \eqref{CNS-1d-n} that
 \ba\label{CNS-1d-dx-13}
\|\d_{t} \eta^{(n)}\|_{L^{2}_{h}} = \| \d_{x}^{n}(\eta w)\|_{L^{2}_{h}}  \leq C   E_{n+2, 0}(y)   e^{-\a t}.
 \ea

\medskip

 \noindent{\bf Step 2.} {Decay estimates of  $\|D_{t} w^{(n)}\|_{L^{2}_{h}}.$}

\medskip

We first get,
by  taking  $L^2_\h$ inner product of $\eta^{-1} w^{(n)}_{xx}$ with the $w^{(n)}$ equation of \eqref{CNS-1d-n}, that
\ba\label{CNS-1d-Dtwn-1}
  \frac 12\frac{\rm d}{\dt} \int_{\TT}   |w^{(n)}_{x}|^{2}\,\dx   + \nu \int_{\TT} \eta^{-1} |w^{(n)}_{xx}|^{2}\,\dx
  =   \int_{\TT}  \bigl(g_{n} - p'(\eta) \eta^{(n)}_{x} +\eta w w_{x}^{(n)}\bigr)\eta^{-1}  w^{(n)}_{xx}\,\dx.
\nn \ea
 Thanks to \eqref{thm1-2}, \eqref{decay-exp-all-m} and \eqref{CNS-1d-dx-13},  we infer
\ba\label{CNS-1d-Dtwn-2}
\frac 12 \frac{\rm d}{\dt} \int_{\TT}    |w^{(n)}_{x}|^{2}\,\dx   + \nu \bar \eta^{-1}  \int_{\TT}  |w^{(n)}_{xx}|^{2}\,\dx
 \leq  \frac{1}{2} \nu \bar \eta^{-1}  \int_{\TT}  |w^{(n)}_{xx}|^{2}\,\dx  + C   E_{n+2, 0}(y)   e^{-\a t}.
\nn \ea
This leads to
\be\label{CNS-1d-Dtwn}
\frac{\rm d}{\dt} \int_{\TT} |w^{(n)}_{x}|^{2} + \nu \bar \eta^{-1} \int_{\TT}  |w^{(n)}_{xx}|^{2} \leq C   E_{n+2, 0}(y)   e^{-\a t}.
\ee

 While we get,
by  taking  $L^2_\h$ inner product of $D_{t} w^{(n)} $  with the $w^{(n)}$ equation of \eqref{CNS-1d-n}, that
\ba\label{CNS-1d-Dtwn-4}
\int_{\TT} \eta  |D_t w^{(n)}|^{2}\,\dx  - \nu \int_{\TT}  w^{(n)}_{xx} D_t w^{(n)}\,\dx =  - \int_{\TT}  \eta^{-1} g_{n} D_{t}w^{(n)}\,\dx - \int_{\TT}   p'(\eta) \eta^{(n)}_{x} D_{t}w^{(n)}\,\dx.
\nn \ea
Observing that
\begin{align*}\label{CNS-1d-Dtwn-5}
- \nu \int_{\TT}  w^{(n)}_{xx} D_t w^{(n)}\,\dx & = - \nu \int_{\TT}  w^{(n)}_{xx} \d_t w^{(n)}\,\dx  - \nu \int_{\TT}  w^{(n)}_{xx} w w^{(n)}_{x}\,\dx \\
& = \frac{\nu}{2}\frac{\rm d}{\dt}  \int_{\TT} |w^{(n)}_{x}|^{2}\,\dx + \frac{\nu}{2}  \int_{\TT} (w^{(n)}_{x})^{2} w_{x}\,\dx,
\end{align*}
which together with \eqref{thm1-2}, \eqref{decay-exp-all-m} and \eqref{CNS-1d-dx-13} ensures that
\be\label{CNS-1d-Dtwn-3}
 \nu \frac{\rm d}{\dt} \int_{\TT} |w^{(n)}_{x}|^{2}\,\dx + \int_{\TT} \eta  |D_t w^{(n)}|^{2}\,\dx   \leq C   E_{n+2,0}(y)   e^{-\a t} .
\ee

Before proceeding,  let us admit the following lemma, the proof of which will be postponed till we finish the proof of
Proposition \ref{S2prop1}.

 \begin{lem}\label{lem-Dtwn-3} For all $t>0$, $y\in \R$, there holds
\be\label{CNS-1d-Dtwn-7}
\frac{\rm d}{\dt} \int_{\TT} \eta |D_{t} w^{(n)}_{x}|^{2}\,\dx + \nu \int_{\TT} |(D_t w^{(n)})_{x} |^{2}\,\dx
   \leq C   E_{n+2,0}(y)   e^{-\a t} + C \int_{\TT} \bigl(|D_{t} w^{(n)} |^{2} +  |w^{(n)}_{xx}|^{2}\bigr)\,\dx.
\ee
\end{lem}

\medskip

\noindent{\bf Step 3.} {End of the induction and summary of decay estimates}

\medskip

By virtue of \eqref{CNS-1d-Dtwn}, \eqref{CNS-1d-Dtwn-3} and \eqref{CNS-1d-Dtwn-7},
 we can find a  large enough constant $A_{7}$ such that
\ba\label{CNS-1d-Dtwn-25}
\frac{\rm d}{\dt} \int_{\TT} \bigl(A_{7}| w^{(n)}_{x} |^{2} + \eta |D_{t} w^{(n)}|^{2}\bigr)\,\dx
 + \int_{\TT} \bigl(|D_{t} w^{(n)} |^{2} +  |w^{(n)}_{xx}|^{2}  +  \nu  | (D_{t} w^{(n)})_{x}|^{2}\bigr)\,\dx \leq  C   E_{n+2,0}(y)   e^{-\a t},
 \nn \ea
from which,  the induction assumption \eqref{decay-exp-all-m0} for all $0\leq m\leq n-1$ and \eqref{CNS-1d-n-9}, we deduce that
 \eqref{decay-exp-all-m0} holds  with $m=n$ for some small positive number $\de.$ This implies  \eqref{decay-exp-all-m} with $m=n$.  We thus complete the induction argument and show in aprticular that
\ba\label{CNS-1d-Dtwn-26}
 \|(\eta -1)(t)\|_{H^{4}_\h}^{2} + \| w (t)\|_{H^{5}_\h}^{2} +  \sum_{m=1}^{2}\| \d_{t}^{m}\eta (t)\|_{H^{4-m}_{\h}}^{2} + \| \d_{t}^{m} w (t)\|_{H^{5-2m}_\h}^{2} \leq C E_{5,0}(y) e^{-\a t}.
\ea
This finishes the proof of Proposition \ref{S2prop1}.
\end{proof}

Now let us present the proof of Lemma \ref{lem-Dtwn-3}.

\begin{proof}[Proof of Lemma \ref{lem-Dtwn-3}]
We first rewrite $\eqref{CNS-1d-n}_{2}$ as
\beq\label{CNS-1d-Dtwn-8}
\begin{split}
 &\eta D_{t}  w^{(n)} - \nu w^{(n)}_{xx} + (p(\eta))^{(n+1)} = - \tilde g_{n}, \with \\
& \tilde g_{n}\eqdefa \sum_{m=1}^{n} \eta^{(m)} w^{(n-m)}_{t}  + \sum_{m=1}^{n} \eta^{(m)} \d_{x}^{n-m} (w w_{x})  + \eta \sum_{m=1}^{n} w^{(m)} w^{(n-m+1)}.
\end{split}
\eeq
By applying $D_{t}$ to \eqref{CNS-1d-Dtwn-8} and then taking $L^2_\h$ inner product of the resulting equation with $D_{t} w^{(n)},$
we find
\ba\label{CNS-1d-Dtwn-10}
\int_{\TT}\bigl( \eta D^2_{t}  w^{(n)} + D_{t} \eta D_{t} w^{(n)} - \nu  D_{t} w^{(n)}_{xx} + D_{t}(p(\eta))_{n}\bigr)D_{t} w^{(n)}\,\dx
 = -\int_{\TT} D_{t} \tilde g_{n}D_{t} w^{(n)}\,\dx.
\ea
 In what follows, we shall frequently use the estimates \eqref{decay-exp-all-m}, \eqref{CNS-1d-dx-12} and \eqref{CNS-1d-dx-13} to handle
 the terms above.

We first observe that
\ba\label{CNS-1d-Dtwn-11}
 \int_{\TT}\eta D_{t}^2   w^{(n)} D_{t}  w^{(n)}\,\dx  =  \frac 12\frac{\rm d}{\dt} \int_{\TT}\eta |D_{t} w^{(n)}_{x}|^{2}\,\dx.
 \ea
For the second term in \eqref{CNS-1d-Dtwn-10}, direct calculation gives
\ba\label{CNS-1d-Dtwn-12}
 \int_{\TT} D_{t} \eta D_{t} w^{(n)} D_{t} w^{(n)}\,\dx  \leq \|w_{x}\|_{L^{\infty}_{h}}   \int_{\TT} \eta  | D_{t} w^{(n)}|^{2}\,\dx \leq C   \int_{\TT} \eta  | D_{t} w^{(n)}|^{2}\,\dx.
 \ea
Notice that
\be
D_{t} w^{(n)}_{xx} = (D_{t} w^{(n)})_{xx} - (2w_{x}w^{(n)}_{xx} + w_{xx} w^{(n)}_{x}).
 \nn\ee
Then, by using integration by parts, we get that
 \ba\label{CNS-1d-Dtwn-14}
-\nu  \int_{\TT}  D_{t} w^{(n)}_{xx} D_{t} w^{(n)}\,\dx & = \nu\int_{\TT}  | (D_{t} w^{(n)})_{x}|^{2}\,\dx + \int_{\TT} \nu  (2w_{x}w^{(n)}_{xx} + w_{xx} w^{(n)}_{x}) D_{t} w^{(n)}\,\dx\\
&\geq  \nu\int_{\TT} | (D_{t} w^{(n)})_{x}|^{2}  - C  \int_{\TT} \bigl(|w_{xx}|^{2} + |w^{(n)}_{xx}|^{2} + |D_{t} w^{(n)}|^{2}\bigr)\,\dx.
\ea
Similarly, we have
\begin{align*}
&\int_{\TT} D_{t}(p(\eta))^{(n+1)} D_{t} w^{(n)}\,\dx    =  \int_{\TT} \d_{t}(p(\eta))^{(n+1)} D_{t} w^{(n)}\,\dx
 +  \int_{\TT} w (p(\eta))^{(n+2)} D_{t} w^{(n)}\,\dx\\
 & = - \int_{\TT} \d_{t}(p(\eta))^{(n)} (D_{t} w^{(n)})_{x}\,\dx -  \int_{\TT} (p(\eta))^{(n+1)}\bigl( w_{x} D_{t} w^{(n)} +  w
(D_{t} w^{(n)} )_{x}\bigr)\,\dx,
\end{align*}
from which, we infer
 \ba\label{CNS-1d-Dtwn-15}
 \int_{\TT} D_{t}(p(\eta))^{(n+1)} D_{t} w^{(n)}\,\dx   \leq & C \int_{\TT} |\d_{t}(p(\eta))^{(n)}|^{2}\,\dx
 + \int_{\TT} |D_{t} w^{(n)}|^{2}\,\dx  + \frac{\nu}{4} \int_{\TT}  | (D_{t} w^{(n)})_{x}|^{2}\,\dx\\
& + C \bigl(\|w_{x}\|_{L^{\infty}_{h}}+\|w\|_{L^{\infty}_{h}}\bigr) \int_{\TT} | (p(\eta))^{(n+1)}|^{2}\,\dx \\
 \leq & C    E_{n+2,0}(y)   e^{-\a t} + \int_{\TT} |D_{t} w^{(n)}|^{2}\,\dx + \frac{\nu}{4} \int_{\TT}  | (D_{t} w^{(n)})_{x}|^{2}\,\dx.
\ea

It  remains to handle the source term in \eqref{CNS-1d-Dtwn-10}. Indeed in view of \eqref{CNS-1d-Dtwn-10},
 we find that the  most trouble terms are the following ones with highest derivatives on $\eta$ or $w$ in $\tilde{g}_n$:
\ba\label{CNS-1d-Dtwn-16}
&\tilde g_n^{(1)}\eqdefa \eta^{(1)} w^{(n-1)}_{t}  + \eta^{(1)}w w_{x}^{n-1}  + \eta w^{(1)} w^{(n)} = \eta^{(1)}D_{t}  w^{(n-1)} + \eta w^{(1)}  w^{(n)},\\
& \tilde g_n^{(2)}\eqdefa \eta^{(n)}   + \eta^{(n)}  + \eta w^{(n)}w^{(1)}.
\ea

Observe that
\begin{align*}\label{CNS-1d-Dtwn-17}
&D_{t} (\eta^{(1)}D_{t}  w^{(n-1)})   = (D_{t} \eta_{x} ) D_{t}  w^{(n-1)} + \eta_{x} D^2_{t}  w^{(n-1)} \\
& = D_{t} \eta_{x} \eta^{-1} \big(  \nu  w^{(n-1)}_{xx} -  (p(\eta))^{(n)} - \tilde g_{n-1}\big) + \eta_{x} ( D_{t} \eta^{-1})\big(  \nu  w^{(n-1)}_{xx} -  (p(\eta))^{(n)} - \tilde g_{n-1}\big) \\
& \qquad + \eta_{x} \eta^{-1}  D_{t}\big(  \nu  w^{(n-1)}_{xx} -  (p(\eta))^{(n)} - \tilde g_{n-1}\big) .
\end{align*}
It follows from \eqref{decay-exp-all-m}, \eqref{CNS-1d-dx-12} and \eqref{CNS-1d-dx-13} that
\begin{align*}%\label{CNS-1d-Dtwn-18}
&\int_{\TT} D_{t} \eta_{x} \eta^{-1} \big(  \nu  w^{(n-1)}_{xx} -  (p(\eta))^{(n)} - \tilde g_{n-1}\big)  D_{t} w^{(n)}\,\dx\\
& \leq \underline \eta^{-1}\|\big(  \nu  w^{(n-1)}_{xx} -  (p(\eta))^{(n)} - \tilde g_{n-1}\big)\|_{H^{1}_{h}}^{2}
 \int_{\TT} |D_{t} \eta_{x}|^{2}\,\dx + \int_{\TT} |D_{t} w^{(n)} |^{2} \,\dx\\
& \leq C   E_{n+2,0}(y)   e^{-\a t} + \int_{\TT} |D_{t} w^{(n)} |^{2}\,\dx,
\end{align*}
and
\begin{align*}
&\int_{\TT} \eta_{x} ( D_{t} \eta^{-1})\big(  \nu  w^{(n-1)}_{xx} -  (p(\eta))^{(n)} - \tilde g_{n-1}\big)  D_{t} w^{(n)}\,\dx\\
& \leq \|\eta_{x} ( D_{t} \eta^{-1}) \big(  \nu  w^{(n-1)}_{xx} -  (p(\eta))^{(n)} - \tilde g_{n-1}\big)\|_{H^{1}_{h}}^{2}    + \int_{\TT} |D_{t} w^{(n)} |^{2}\,\dx \\
& \leq C   E_{n+2,0}(y)   e^{-\a t} + \int_{\TT} |D_{t} w^{(n)} |^{2}\,\dx .
\end{align*}

Next we deal with the term $\int_{\TT}\eta_{x} \eta^{-1} D_{t}\big(   \nu  w^{(n-1)}_{xx} -  (p(\eta))^{(n)} - \tilde g_{n-1}\big)\,\dx $. We compute
\ba\label{CNS-1d-Dtwn-20}
\eta_{x}\eta^{-1}  D_{t}  w^{(n-1)}_{xx} & = \eta_{x} \eta^{-1} (D_{t} w^{(n)})_{x} - \eta_{x} \eta^{-1} w_{x} w^{(n)}_{x},
\nn \ea
so that
\begin{align*}%
\int_{\TT} \eta_{x}\eta^{-1}  D_{t}  w^{(n-1)}_{xx}  D_{t} w^{(n)}\,\dx & = \int_{\TT} \eta_{x} \eta^{-1} (D_{t} w^{(n)})_{x} D_{t} w^{(n)}\,\dx   - \int_{\TT}  \eta_{x} \eta^{-1} w_{x} w^{(n)}_{x} D_{t} w^{(n)} \,\dx \\
& \leq C   E_{n+2,0}(y)   e^{-\a t} + C \int_{\TT} |D_{t} w^{(n)} |^{2}\,\dx + \frac{\nu}{4} \int_{\TT}  | (D_{t} w^{(n)})_{x}|^{2}\,\dx.
\end{align*}
Similarly, we have
\begin{align*}
\int_{\TT} \eta_{x}\eta^{-1}  D_{t}  (p(\eta))^{(n)} D_{t} w^{(n)}\,\dx & \leq  \underline\eta^{-2} \|\eta_{x}\|_{L^{\infty}_{h}}^{2} \int_{\TT}  |D_{t} (p(\eta))_{n}|^{2}\,\dx + \int_{\TT} | D_{t} w^{(n)}|^{2}\,\dx \\
& \leq C   E_{n+2,0}(y)   e^{-\a t} + C \int_{\TT} |D_{t} w^{(n)} |^{2}\,\dx.
\end{align*}
Finally, by virtue of  the definition of $\tilde g_{n}$ given by \eqref{CNS-1d-Dtwn-8},
and the estimates \eqref{decay-exp-all-m}, \eqref{CNS-1d-dx-12} and \eqref{CNS-1d-dx-13}, we deduce  that
\begin{align*}
\int_{\TT} \eta_{x}\eta^{-1}  D_{t} \tilde g_{n-1}  D_{t} w^{(n)}\,\dx & \leq
 \underline\eta^{-2} \|\eta_{x}\|_{L^{\infty}_{h}}^{2} \int_{\TT}  |D_{t} \tilde g_{n-1}|^{2}\,\dx + \int_{\TT} | D_{t} w^{(n)}|^{2}\,\dx \\
& \leq C   E_{n+2,0}(y)   e^{-\a t} + C \int_{\TT} |D_{t} w^{(n)} |^{2}\,\dx.
\end{align*}
By summarizing the above estimates, we arrive at
\beq \label{CNS-1d-Dtwn-21}
\int_{\TT}D_{t} (\eta^{(1)}D_{t}  w^{(n-1)}) D_{t} w^{(n)}\,\dx \leq C   E_{n+2,0}(y)   e^{-\a t} +  \int_{\TT} \bigl(C|D_{t} w^{(n)} |^{2} + \frac{\nu}{4} | (D_{t} w^{(n)})_{x}|^{2}\bigr)\,\dx.
\eeq
The other terms in \eqref{CNS-1d-Dtwn-16} can be handled along the same line.

By substituting the estimates (\ref{CNS-1d-Dtwn-11}--\ref{CNS-1d-Dtwn-15}) and \eqref{CNS-1d-Dtwn-21} into \eqref{CNS-1d-Dtwn-10},
we conclude the proof of  \eqref{CNS-1d-Dtwn-7}. This finishes the proof of Lemma \ref{lem-Dtwn-3}.
\end{proof}

 \section{Proof of Lemma \ref{lem-GN}}\label{appb}

Let us recall  Lemma \ref{lem-GN}: let $2<p<\infty$ and  $\OO = \TT \times \R$, there exists a constant $C$ depending solely on $p$ such that for any $f\in H^{1}(\OO),$  there holds
\ba\label{GN-O-1}
\|f\|_{L^{p}(\OO)} \leq C \Bigl(  \| f \|_{L^{2}(\OO)}^{\frac 2p} \|\nabla f \|_{L^{2}(\OO)}^{1-\frac{2}{p}} +  \| f \|_{L^{2}(\OO)}^{\frac{1}{2} + \frac{1}{p}}  \|\nabla f \|_{L^{2}(\OO)}^{\frac{1}{2} - \frac{1}{p}} \Bigr).
\ea

\begin{proof}[Proof of Lemma \ref{lem-GN}]
 Let $f \in H^{1}(\OO),$ we split it as follows
$$
f(x,y) = \tilde f(x,y) + \bar f(y), \with \bar f(y)\eqdefa \int_{\TT} f(x,y)\,\dx, \andf \tilde f \eqdefa f - \bar f.
$$
We observe that $\bar f \in H^{1}(\R)$. Indeed, applying H\"older's inequality gives
\ba\label{GN-O-2}
 &\| \bar f \|_{L^{2}(\R)}^{2} = \int_{\R} \bigl| \int_{\TT} f(x,y)\,\dx\bigr|^{2}\,\dy  
 \leq \int_{\R}  \int_{\TT} |f(x,y)|^{2}\,\dx\,\dy   = \|f\|_{L^{2}(\OO)}^{2},  \\
&\| \d_{y}\bar f \|_{L^{2}(\R)}^{2}  = \| \overline {\d_{y} f} \|_{L^{2}(\R)}^{2} \leq \|\d_{y} f\|_{L^{2}(\OO)}^{2}.
\ea
Furthermore, $\bar f \in H^{1}(\OO)$ and $\tilde f \in H^{1}(\OO),$ and there hold
\ba\label{GN-O-3}
 \| \bar f \|_{L^{2}(\OO)} =  \|\bar f\|_{L^{2}(\R)},  \quad  \| \nabla_{x,y}\bar f \|_{L^{2}(\OO)}  = \| \d_{y}\bar f \|_{L^{2}(\R)}.
\ea

For $\tilde f$, we get, by applying the classical two dimensional Gagliardo-Nirenberg interpolation inequality, that
\ba\label{GN-O-4}
\|\tilde f\|_{L^{p}(\OO)} \leq C  \| \tilde f \|_{L^{2}(\OO)}^{\frac 2p} \| \tilde f \|_{H^{1}(\OO)}^{1-\frac{2}{p}} \leq C \| \tilde f \|_{L^{2}(\OO)} + C \| \tilde f \|_{L^{2}(\OO)}^{\frac 2p} \| \nabla \tilde f \|_{L^{2}(\OO)}^{1-\frac{2}{p}}.
\ea
Notice that
$$
\int_{\TT} \tilde f(x,y) \,\dx = 0, \quad \forall y \in \R,
$$
by applying Poincar\'e's inequality, one has
$$
\int_{\TT} |\tilde f(x,y)|^{2} \,\dx \leq   \int_{\TT} | \d_{x} \tilde f(x,y)|^{2} \,\dx, \quad \forall y \in \R.
$$
This leads to
\ba\label{GN-O-5}
\| \tilde f \|_{L^{2}(\OO)}^{2} = \int_{\R}\int_{\TT} |\tilde f(x,y)|^{2} \,\dx \,\dy \leq \int_{\R}\int_{\TT} | \d_{x} \tilde f(x,y)|^{2} \,\dx \,\dy \leq \| \nabla \tilde f \|_{L^{2}(\OO)}^2.
\ea
Plugging \eqref{GN-O-5}  into \eqref{GN-O-4} and using \eqref{GN-O-3},  we find
\ba\label{GN-O-6}
\|\tilde f\|_{L^{p}(\OO)} \leq C \| \tilde f \|_{L^{2}(\OO)}^{\frac 2p} \| \nabla \tilde f \|_{L^{2}(\OO)}^{1-\frac{2}{p}} \leq C \| f \|_{L^{2}(\OO)}^{\frac 2p} \| \nabla f \|_{L^{2}(\OO)}^{1-\frac{2}{p}}.
\ea
While for $\bar f(y),$ we get, by applying the one dimensional Gagliardo-Nirenberg interpolation inequality (see \cite{DELL14, Nir59}) and \eqref{GN-O-2}, that
\ba\label{GN-O-7}
\|\bar f\|_{L^{p}(\R)} \leq C  \| \bar f \|_{L^{2}(\R)}^{\frac{1}{2} + \frac{1}{p}}  \|\d_{y} \bar f \|_{L^{2}(\R)}^{\frac{1}{2} - \frac{1}{p}}  \leq C  \|  f \|_{L^{2}(\OO)}^{\frac{1}{2} + \frac{1}{p}} \|\d_{y}  f \|_{L^{2}(\OO)}^{\frac{1}{2} - \frac{1}{p}} .
\ea

Then \eqref{GN-O-1} follows immediately from \eqref{GN-O-6} and \eqref{GN-O-7}. This finishes proof of Lemma \ref{lem-GN}.
\end{proof}

\section*{Acknowledgments}

Yong Lu has been supported by the Recruitment Program of Global Experts of China. Ping Zhang is partially supported by K.C.Wong Education Foundation and NSF of China under Grants   11731007, 12031006 and 11688101.

%%%%%%%%%%%%%%%%%%%%%%%%%%%%%%%%%%%%%%%%%%%%%%%%%%%%%%%%%%%%%%%%%%%%%%%%%%%%%%%%%%%%%%%%%%%%%%%%%%%%%%%%%

\end{document}